\documentclass[a4paper,reqno]{amsart}

\usepackage[utf8]{inputenc}
\usepackage{amssymb}
\usepackage{enumitem}
\usepackage{mathrsfs}
\usepackage{mathtools}
\usepackage{color}
\usepackage[colorlinks=true]{hyperref}
\hypersetup{urlcolor=green,citecolor=green}

\setlist[enumerate]{label=\emph{(\roman*)}}

\usepackage{environ}

\newtheorem{theorem}{Theorem}[section]

\newtheorem{lemma}[theorem]{Lemma}
\newtheorem{proposition}[theorem]{Proposition}

\theoremstyle{definition}

\newtheorem{remark}[theorem]{Remark}
\numberwithin{equation}{section}
\newcommand{\R}{\mathbb{R}}

\def\RR{\mathbb{R}}
\def\bell{\boldsymbol{\ell}}
\def\ba{\boldsymbol{a}}
\def\bb{\boldsymbol{b}}
\def \d {{\rm{d}}}
\def \vp{\varphi_{1}}
\def \vpp{\varphi_{2}}
\def \E{\mathcal{H}}
\def \pt{\partial_{t}}
\begin{document}

\parindent=0pt

	\title[Multi-solitons for the 4D cubic wave equation]{Construction of excited multi-solitons for the focusing 4D cubic wave equation}
	\author{XU YUAN}
	
	\address{CMLS, \'Ecole polytechnique, CNRS, Institut Polytechnique de Paris, F-91128 Palaiseau Cedex, France.}
	
	\email{xu.yuan@polytechnique.edu}
	\begin{abstract}
    Consider the focusing 4D cubic wave equation
	\begin{equation*}
		\partial_{tt}u-\Delta u-u^{3}=0,\quad \mbox{on}\ (t,x)\in [0,\infty)\times \RR^{4}.
	\end{equation*}
	 The main result states the existence in energy space $\dot{H}^{1}\times L^{2}$ of multi-solitary waves where each traveling wave is generated by Lorentz transform from a specific excited state, with different but collinear Lorentz speeds. The specific excited state is deduced from the non-degenerate sign-changing state constructed in Musso-Wei~\cite{MW}. The proof is inspired by the techniques developed for the 5D energy-critical wave equation and the nonlinear Klein-Gordon equation in a similar context by Martel-Merle~\cite{MM} and C\^ote-Martel~\cite{CMTAMS}. 
	 
	 The main difficulty originates from the strong interactions between solutions in the 4D case compared to other dispersive and wave-type models. To overcome the difficulty, a sharp understanding of the asymptotic behavior of the excited states involved and of the kernel of their linearized operator is needed. 
	\end{abstract}
	\maketitle
\section{Introduction}
\subsection{Main result}
Consider the focusing 4D cubic wave equation,
\begin{equation}\label{equ:wave}
\left\{ \begin{aligned}
&\partial_{tt} u - \Delta u-u^{3} = 0,\ \ (t,x)\in [0,\infty)\times \RR^{4},\\
& (u,\partial_t u)_{|t=0} = (u_0,u_{1})\in  \dot{H}^1\times L^{2}.
\end{aligned}\right.
\end{equation}
Recall that the Cauchy problem for equation~\eqref{equ:wave} is locally well-posed in the energy space $\dot{H}^{1}\times L^{2}$. See \emph{e.g.}~\cite{KMACTA} and references therein. For any $\dot{H}^{1}\times L^{2}$ solution, the energy $E$ and momentum $P$ are conserved, where
\begin{equation*}
E(u,\pt u)=\frac{1}{2}\int_{\RR^{4}}\left(|\nabla u|^{2}+(\pt u)^{2}-\frac{1}{4}u^{4}\right)\d x,\quad 
P(u,\pt u)=\int_{\RR^{4}}\pt u\nabla u\d x.
\end{equation*}
Denote by $\Sigma$ the set of non-zero stationary solutions of~\eqref{equ:wave},
\begin{equation}\label{def:Sig}
\Sigma=\left\{q\in \dot{H}^{1}(\RR^{4})\setminus 0:-\Delta q-q^{3}=0\ \  \mbox{on}\ \RR^{4}\right\}.
\end{equation}
For $q\in \Sigma$ and $\bell\in \RR^{4}$ with $|\bell|<1$, let 
\begin{equation*}
q_{\bell}(x)=q\left(\left(\frac{1}{\sqrt{1-|\bell|^{2}}}-1\right)\frac{\bell(\bell\cdot x)}{|\bell|^{2}}+x\right),
\end{equation*}
then $u(t,x)=q_{\bell}(x-\bell t)$ is a global, bounded solution of~\eqref{equ:wave}.

Standard elliptic arguments (see~\emph{e.g.}~\cite{GT}) show that if $q\in \Sigma$, then $q$ is of class $C^{\infty}$. It is known by the results developed in \cite{Aub,Caff,Obata,Talen} that up to dilation and translation, the only \emph{positive} stationary solution of~\eqref{equ:wave} is the \emph{ground state} $W$, where $W$ has the explicit form
\begin{equation*}
W(x)=\left(1+\frac{|x|^{2}}{8}\right)^{-1}\quad \mbox{on}\ \RR^{4}.
\end{equation*}
Functions $q\in \Sigma$ with $q\ne W$ (up to sign, dilation and translation) are called \emph{excited states}.
The existence of excited states that are nonradial, sign-changing and with arbitrary large energy was first proved by Ding~\cite{Ding} using variational arguments. Later, more explicit constructions of sign-changing solutions have been obtained by del Pino-Musso-Pacard-Pistoia~\cite{PMPP1}. More specifically, they constructed solutions with a centered soliton crowned by negative spikes (rescaled solitons) at the vertices of a regular polygon of radius 1. Then, following a similar general strategy in~\cite{PMPP1}, they constructed in~\cite{PMPP2} sign changing, non radial solutions to the Yamabe-type equation on the sphere $\mathbb{S}^d$ ($d \ge 4$) whose energy is concentrated along special submanifolds of $\mathbb{S}^d$.

For any $q\in \Sigma$, we denote the linearized operator around $q$ by 
\begin{equation*}
\mathcal{L}_{q}=-\Delta-3q^{2}.
\end{equation*}
Set 
\begin{equation*}
\mathcal{Z}_{q}=\left\{f\in \dot{H}^{1}\ \mbox{such that}\ \mathcal{L}_{q}f=0\right\},
\end{equation*}
and
\begin{multline*}
\qquad \widetilde{{\mathcal{Z}}}_{q}=
{\rm{Span}} \bigg\{ q+x\cdot \nabla q;\ 
(x_{i}\partial_{x_{j}}q-x_{j}\partial_{x_{i}}q),1\le i<j\le 4;\\
\partial_{x_{i}}q ; \  -2x_{i}q+|x|^{2}\partial_{x_{i}}q-2x_{i} (x\cdot \nabla q),1\le i\le 4\bigg\}.\qquad\qquad
\end{multline*}
The function space $\widetilde{\mathcal{Z}}_{q}\subset \mathcal{Z}_{q}$ is the null space of $\mathcal{L}_{q}$  that is generated by a family of explicit transformations (see~\cite[Lemma 3.8]{DKMASNSP} and \S\ref{SS:trans}).
Following Duyckaerts-Kenig-Merle~\cite{DKMASNSP}, we call $q\in \Sigma$ a \emph{non-degenerate state} if $\mathcal{Z}_{q}=\widetilde{{\mathcal{Z}}}_{q}$. The first nontrivial example of non-degenerate sign-changing state was obtained in Musso-Wei~\cite{MW} who proved that the sign-changing states constructed in~\cite{PMPP1} are non-degenerate. Later, following a similar strategy, Medina-Musso-Wei~\cite{MMW} constructed a non-degenerate state with maximal rank,~\emph{i.e.} ${\rm{dim}}\widetilde{\mathcal{Z}}_{q}=15$.

\smallskip

The main goal of this article is to construct multi-solitons of~\eqref{equ:wave} based on a specific non-degenerate excited state.
\begin{theorem}[Existence of excited multi-solitons]\label{the:main}
	There exists a non-degenerate excited state $Q\in \Sigma$ such that the following is true. Let $N\ge 2$ and $\boldsymbol{\rm{e}}_{1}$ be the first vector of the canonical basis of $\RR^{4}$. For $n=1,\dots,N$, let $\tau_{n}=\pm 1$ and $\boldsymbol{\ell}_{n}\in \RR^{4}$ be such that 
	\begin{equation*}
	\boldsymbol{\ell}_{n}=\ell_{n}\boldsymbol{\rm{e}}_{1}\ \mbox{with}\ -1<\ell_{1}<\ell_{2}<\dots<\ell_{N}<1.
	\end{equation*}
	Then, there exist $T_{0}>0$ and a solution $\vec{u}=(u,\partial_{t}u)$ of~\eqref{equ:wave} on $[T_{0},+\infty)$ in the energy space $\dot{H}^{1}\times L^{2}$ such that 
	\begin{equation}\label{equ:limu}
	\begin{aligned}
	&\lim_{t\to+\infty}\left\|u(t)-\sum_{n=1}^{N}\tau_{n}Q_{\bell_{n}}(\cdot-\bell_{n}t)\right\|_{\dot{H}^{1}}=0,\\
	&\lim_{t\to+\infty}\left\|\partial_{t}u(t)+\sum_{n=1}^{N}\tau_{n}\ell_{n}\partial_{x_{1}}Q_{\bell_{n}}(\cdot-\bell_{n}t)\right\|_{L^{2}}=0.
	\end{aligned}
	\end{equation}
\end{theorem}

\begin{remark}
	The specific non-degenerate state $Q$ involved in Theorem~\ref{the:main} is obtained from the non-degenerate sign-changing state constructed in~\cite{PMPP1,MW}. Indeed, by rotation, translation and Kelvin transformation on a non-degenerate sign-changing state, we obtain a non-degenerate excited state that decays fast enough at least along $\boldsymbol{\rm{e}}_{1}$ (see more details in \S\ref{SS:trans}). The interaction of two such traveling waves with different speeds along $\boldsymbol{\rm{e}}_{1}$, is formally of order $t^{-4}$ for large time. This rate allows us to close the energy estimates (see more details in \S\ref{SS:Boot} and \S\ref{SS:ener}).
	However, the slow algebraic decay of the linearized operator's kernel leads to another technical difficulty. To overcome it, we need to refine the energy estimates (see more details in \S\ref{SS:dec} and \S\ref{SS:ener}).
\end{remark}

\begin{remark}
	 The Theorem~\ref{the:main} is also true for the more general case where each traveling wave is based on non-degenerate $q$ with the same asymptotic behavior as $Q$ (see Lemma~\ref{le:asyQ} and Lemma~\ref{le:asyker}). In particular, we do not need each traveling wave to be based on the same non-degenerate state.
\end{remark}

\begin{remark}
	Following~\cite[\S1.2]{MM}, we conjecture that in the 4D case, there exists no multi-soliton based on ground state $W$ in the sence of Theorem~\ref{the:main}, for any $N\ge 2$. Heuristically, from $W(x)\sim |x|^{-2}$ as $|x|\to \infty$, the interaction between two solitons of different speeds is $t^{-2}$. Following the modulation and energy method, we find these interactions are too strong and formally that can create diverging terms of the geometric parameters. However, such nonexistence result is open problem for any space dimension $d\ge 3$.
\end{remark}

Historically, the existence of multi-soliton solutions for non-integrable cases was first studied by Merle~\cite{Mnls} for the $L^{2}$ critical nonlinear Schr\"odinger equation (NLS) and Martel~\cite{Ma} for the subcritical and critical generalized Korteweg-de Vries equations (gKdV). Later, the strategy of these works was extended to other dispersive or wave-type models (see~\cite{BGL,CMM,CMkg,MMnls,MRT,XY,Y5De}). In particular, the present work is inspired by the proof of the existence of multi-soliton solutions for the nonlinear Klein-Gordon and 5D energy-critical wave equations in~\cite{CMTAMS,MM}. Technically, such constructions of multi-solitons of the energy-critical wave equation are more challenging than for other dispersive and wave-type models, because of the weak algebraic decay of solitons.

\smallskip

To our knowledge, this article proves the first statement of the existence of multi-soliton solutions for the 4D energy-critical wave equation. The construction of such solutions for the energy-critical wave equation is especially motivated by the soliton resolution conjecture, fully solved in the 3D radial case by Duyckaerts-Kenig-Merle~\cite{DKM},  and for a subsequence of time in 3, 4 and 5D by Duyckaerts-Jia-Kenig-Merle~\cite{DJKM}. Recently, the conjecture in the 4D radial case is proved by Duyckaerts-Kenig-Martel-Merle~\cite{DKMM}. Indeed, the existence of such multi-solitons complements the statements in~\cite{DJKM} by exhibiting examples of solutions containing $N$ solitons at $t\to +\infty$ for any $N\ge 2$. In the same direction, recall that for the 6D energy-critical wave equation, Jendrej~\cite{JJ} proved the existence of radial two-bubble solutions. Later, for the 5D energy-critical wave equation, Jendrej-Martel~\cite{JM} proved the existence of $N$-bubble solutions for any $N\ge 2$. For the 3 and 5D cases, Krieger-Nakanishi-Schlag~\cite{KNS} constructed a center-stable manifold of the ground state. Recall also that, for the 4D case, Hillairet-Rapha\"el~\cite{HiRa} proved the existence of type-II finite time blow-up solutions containing single soliton.

\smallskip

We refer to~\cite{delMW,delMW2,delMW3} for results of the existence of type-II blow up solutions containing single or multi-bubble of the energy-critical heat equation. Particularly related to the present article, the blow up profile in~\cite{delMW3} is the non-degenerate sign-changing state constructed in~\cite{PMPP1,MW} rather than the ground state $W$.

\smallskip

The article is organized as follows. Section~\ref{S:Pre} introduces the technical tools involved in the choice of specific excited state $Q$: the spectral theory of the linearized operator, the asymptotic behavior of $Q$ and of the kernel of its linearized operator. Section~\ref{S:dem} introduces the technical tools involved in a dynamical approach to the $N$-soliton problem for~\eqref{equ:wave}: estimates of the nonlinear interactions between solitons, decomposition by modulation and parameter estimates. Finally, Theorem~\ref{the:main} is proved in Section~\ref{S:Thm} by energy estimates and a suitable compactness argument.
\subsection{Notation}
We fix a smooth, even function $\xi:\RR\to \RR$ satisfying:
\begin{equation}\label{def:xi}
\xi\equiv 1\ \mbox{on}\ [0,1],\quad 
\xi\equiv 0\ \mbox{on}\ [2,\infty),\quad 
\mbox{and}\quad  \xi'\le 0\le \xi\le 1\ \mbox{on}\ [0,\infty).
\end{equation}
For future reference, we state the first-order Taylor expansion involving $C^{\infty}$ function.
Let $f:\RR^{4}\to \RR$ be a $C^{\infty}$ function, 
\begin{equation}\label{equ:tay}
f(x)=f(0)+\sum_{i=1}^{4}\partial_{x_{i}}f(0)x_{i}+O(|x|^{2})\quad \mbox{for all}\ |x|<1.
\end{equation}
Recall that for two $C^{\infty}$ functions $f_{1},f_{2}:\RR^{4}\to \RR$, Leibniz’s formula writes:
\begin{equation}\label{equ:Lei}
\partial_{x}^{\alpha}(f_{1}f_{2})=\sum_{\beta+\gamma=\alpha}C_{\alpha}^{\beta}(\partial_{x}^{\beta}f_{1})
(\partial_{x}^{\gamma}f_{2}),\quad \mbox{for all}\ \alpha\in \mathbb{N}^{4}.
\end{equation}
We denote 
\begin{equation*}
\left( f,g \right)_{L^{2}}:=\int_{\RR^{4}}f(x)g(x)\d x,\quad 
\left(f,g\right)_{\dot{H}^{1}}:=\int_{\RR^{4}}\nabla f(x)\cdot \nabla g(x)\d x.
\end{equation*}
For
\begin{equation*}
\vec{f}=\left( \begin{array}{c}
f_{1}\\ f_{2}
\end{array}\right),\quad \vec{g}=\left( \begin{array}{c}
g_{1}\\ g_{2}
\end{array}\right),
\end{equation*}
denote
\begin{equation*}
\left( \vec{f},{\vec{g}}\right)_{L^{2}}:=\sum_{k=1,2}\left( f_{k},g_{k}\right)_{L^{2}},\quad 
\|\vec{f}\|^{2}_{\mathcal{H}}:=\|f_{1}\|_{\dot{H}^{1}}^{2}+\|f_{2}\|_{L^{2}}^{2}.
\end{equation*}
When $x_{1}$ is seen as the specific coordinate, denote
\begin{equation*}
\bar{x}=(x_{2},x_{3},x_{4}),\quad \overline{\nabla}=(\partial_{x_{2}},\partial_{x_{3}},\partial_{x_{4}})\quad \mbox{and}\quad 
\bar{\Delta}=\sum_{i=2}^{4}\partial_{x_{i}}^{2}.
\end{equation*}
For any $-1<\ell<1$ and $f:\RR^{4}\mapsto \RR$, set 
\begin{equation*}
f_{\ell}(x)=f(x_{\ell})\quad \mbox{where}\quad 
x_{\ell}=\left(\frac{x_{1}}{\sqrt{1-\ell^{2}}},\bar{x}\right).
\end{equation*}
We set the $\Omega_{ij}$ are the angular derivatives that are
\begin{equation*}
\Omega_{ij}=x_{i}\partial_{x_{j}}-x_{j}\partial_{x_{i}}\quad \mbox{for}\ 1\le i<j\le 4.
\end{equation*}
We also set 
\begin{equation*}
\Lambda= {\rm{Id}}+x\cdot \nabla\quad \mbox{and}\quad  
\widetilde{\Omega}_{i}=-2x_{i}+|x|^{2}\partial_{x_{i}}-2x_{i}x\cdot\nabla \quad \mbox{for}\ i=1,\dots,4.
\end{equation*}
Note that, for $q\in \Sigma$,
\begin{equation*}
\widetilde{\mathcal{Z}}_{q}={\rm{Span}}\left\{\Lambda q;\partial_{x_{i}}q,\widetilde{\Omega}_{i}q,1\le i\le 4;\Omega_{ij}q,1\le i<j\le 4\right\}.
\end{equation*}
Denote $\langle x \rangle =\left(1+|x|^{2}\right)^{\frac{1}{2}}$ for $x\in \RR^{4}$. For some $0<\gamma\ll 1$ to be fixed, set 
\begin{equation}\label{def:zeta}
\zeta (x)=\langle x \rangle^{-\gamma}\quad \mbox{on}\ \RR^{4}.
\end{equation}

For $J\in \mathbb{N}^{+}$ and $r>0$, we denote by $\mathcal{B}_{\RR^{J}}(r)$ (respectively, $\mathcal{S}_{\RR^{J}}(r)$) be the ball (respectively, the sphere) of $\RR^{J}$ of center $0$ and of radius $r$.

Recall the Sobolev and Hardy inequalities in dimension 4, for any $f\in \dot{H}^{1}$,
\begin{align}
\|f\|_{L^{4}(\RR^{4})}&\lesssim \|\nabla f\|_{L^{2}(\RR^{4})},\label{est:Sobo}\\
\left\|\frac{f}{|x|}\right\|_{L^{2}(\RR^{4})}&\lesssim \|\nabla f\|_{L^{2}(\RR^{4})}.\label{est:Hardy}
\end{align}
\subsection*{Acknowledgements}
The author would like to thank his advisor, Professor Yvan Martel, for his generous help,
encouragement, and guidance related to this work.

\section{Preliminaries}\label{S:Pre}
\subsection{Transforms and Asymptotic behavior}\label{SS:trans}
Following~\cite{DKMASNSP}, we first recall that $\Sigma$ is invariant under the following four transformations:
\begin{enumerate}
	\item [${\rm{(i)}}$] translation: If $q\in \Sigma$ then $q(x+x_{0})\in \Sigma$ for all $x_{0}\in \RR^{4}$;
	\item [${\rm{(ii)}}$] dilation: If $q\in \Sigma$ then $\lambda q(\lambda x)\in \Sigma$ for all $\lambda\in (0,\infty)$;
	\item [${\rm{(iii)}}$] orthogonal transformation: If $q\in \Sigma$ then $q\left(Ox\right)\in \Sigma$ for all $O\in \mathcal{O}_{4}$.
	\item [${\rm{(iv)}}$] Kelvin transformation: If $q\in \Sigma$ then $\mathrm{K}q\in \Sigma$ where $\mathrm{K}q(x)=\frac{1}{|x|^{2}}q\left(\frac{x}{|x|^{2}}\right)$.
\end{enumerate}
We denote by $\mathcal{M}$ the group of isometries of $\dot{H}^{1}$ generated by the preceding transformations.
From~\cite[Section 3]{DKMASNSP}, it is known $\mathcal{M}$ defines a 15-parameter family of transformations in a
neighborhood of the identity. 

For $\theta=(\theta_{ij})_{1\le i<j\le 4}\in \RR^{6}$, set 
\begin{equation*}
A_{\theta}=\left[a_{i,j}^{\theta}\right]_{1\le i\le j\le 4},\quad \mathcal{R}_{\theta}=e^{A_{\theta}}=\sum_{n=0}^{\infty}\frac{A_{\theta}^{n}}{n!},
\end{equation*}
where $a_{i,i}^{\theta}=0$, $a_{i,j}^{\theta}=\theta_{ij}$ if $i<j$, $a^{\theta}_{j,i}=-\theta_{ij}$ if $i<j$.

For $\mathcal{A}=(\lambda,z,\xi,\theta)\in (0,\infty)\times \RR^{4}\times \RR^{4}\times \RR^{6}$, we introduce the following explicit formula of transform $\mathcal{T}_{\mathcal{A}}\in \mathcal{M}$,
\begin{equation*}
\mathcal{T}_{\mathcal{A}}(f)(x)=\lambda\left|\frac{x}{|x|}-z|x|\right|^{-2}f\left(\xi+\frac{\lambda \mathcal{R}_{\theta}(x-z|x|^{2})}{1-2\langle z,x\rangle +|z|^{2}|x|^{2}}\right)\quad \mbox{for}\ f\in \dot{H}^{1}.
\end{equation*}
Recall that for all $q\in \Sigma$, $\mathcal{\widetilde{Z}}_{q}$ is generated by taking partial derivatives of $\mathcal{T}_{\mathcal{A}}(q)$ with respect to $\mathcal{A}=(\lambda,z,\xi,\theta)$ at $\mathcal{A}=(1,0,0,0)$.
Also, following~\cite[Page 734]{DKMASNSP}, we see that if $q$ is non-degenerate then $\mathcal{T}q$ is also non-degenerate for any $\mathcal{T}\in \mathcal{M}$, \emph{i.e.} if ${\mathcal{Z}_{q}}=\mathcal{\widetilde{Z}}_{q}$ then ${\mathcal{Z}_{\mathcal{T}q}}=\mathcal{\widetilde{Z}}_{\mathcal{T}q}$ for any $\mathcal{T}\in \mathcal{M}$.

We first recall the following existence result proved in~\cite{MW}.
\begin{theorem}[\cite{MW}]\label{thm:q}
	There exists $q\in \Sigma$ which is a non-degenerate sign-changing state.
\end{theorem}
The non-degenerate sign-changing state $q$ plays a crucial role in our proof. Indeed, the specific non-degenerate excited state $Q$ mentioned in Theorem~\ref{the:main} is constructed by $q$ via rotation, translation, and Kelvin transformation.
More precisely:
\begin{lemma}\label{le:asyQ} 
	There exists non-degenerate excited state $Q\in \Sigma$ such that the following estimates hold. 
	\begin{enumerate}
		\item \emph{Asymptotic behavior of $Q$}. For all $|x|>1$,
		\begin{equation}\label{est:asQ}
		\left|Q(x)-\frac{x_{4}}{|x|^{4}}\right|\lesssim \frac{1}{|x|^{4}}.
		\end{equation}
		\item \emph{Bound of $\partial_{x}^{\alpha}Q$}. For all $\alpha\in \mathbb{N}^{4}$, 
		\begin{equation}\label{est:aspQ}
		\left|\partial_{x}^{\alpha}Q(x)\right|\lesssim \langle x\rangle ^{-\left(3+|\alpha|\right)}\quad \mbox{on}\ \  \RR^{4}.
		\end{equation}
	\end{enumerate}
\end{lemma}
\begin{proof}
Let $q\in \Sigma$ be the non-degenerate sign-changing state in~Theorem~\ref{thm:q}. Without loss of generality, we assume $q(0)>0$. Set 
\begin{equation*}
r=\sup\left\{R>0;q(x)>0\ \mbox{holds on}\ \mathcal{B}_{\RR^{4}}(R)\right\}.
\end{equation*}
Since $q$ is sign-changing state, we know that $r>0$ is well-defined finite, and there exists $x_{0}\in \mathcal{S}_{\RR^{4}}(r)$ such that $q(x_{0})=0$. Moreover, by the Hopf's Lemma (see for instance~\cite[Lemma 3.4]{GT}), we have $\nabla q(x_{0})\ne 0$.
Considering translation $x\to x+x_{0}$, dilation and rotating
for $q$, we construct a non-degenerate excited state $\tilde{q}\in \Sigma$ satisfying 
\begin{equation}\label{equ:q0}
\tilde{q}(0)=\partial_{x_{1}}\tilde{q}(0)=\partial_{x_{2}}\tilde{q}(0)=\partial_{x_{3}}\tilde{q}(0)=0\quad \mbox{and}\quad 
\partial_{x_{4}}\tilde{q}(0)=1.
\end{equation}
Next, we consider the Kelvin transform $\mathrm{K}$ to $\tilde{q}$ and denote it by $Q$, \emph{i.e.}
\begin{equation*}
Q(x)=\left(\mathrm{K}\tilde{q}\right)(x)=\frac{1}{|x|^{2}}\tilde{q}\left(\frac{x}{|x|^{2}}\right)\quad \mbox{for}\ x\ne 0.
\end{equation*}
Following~\cite[Section 3]{DKMASNSP}, we know that $Q$ is also a non-degenerate state, \emph{i.e.}
\begin{equation*}
-\Delta Q-Q^{3}=0\ \ \mbox{with}\ Q\in C^{\infty},\quad \mbox{and}\quad 
\mathcal{Z}_{Q}=\widetilde{\mathcal{Z}}_{Q}.
\end{equation*}
	Moreover, from~\cite[Lemma 3.8]{DKMASNSP} and $\mathrm{K}^{2}={\rm{Id}}$, we have, for $i=1,\dots,4$, 
	\begin{equation}\label{equ:OiQ}
	\widetilde{\Omega}_{i}Q(x)=\left(\mathrm{K}\left(\partial_{x_{i}}\tilde{q}\right)\right)(x)=\frac{1}{|x|^{2}}\left((\partial_{x_{i}}\tilde{q})\left(\frac{x}{|x|^{2}}\right)\right)\quad \mbox{for}\ x\ne 0.
	\end{equation}
	Now we start to prove $Q$ satisfying (i) and (ii).
	
	Proof of (i). Applied the Taylor expansion~\eqref{equ:tay} of $\tilde{q}$ at $x=0$ and then using~\eqref{equ:q0}, we have 
	\begin{equation*}
	\tilde{q}(x)=\tilde{q}(0)+\sum_{i=1}^{4}\partial_{x_{i}}\tilde{q}(0)x_{i}+O(|x|^{2})=x_{4}+O(|x|^{2})\quad \forall\ |x|<1.
	\end{equation*}
	Replace $x$ with $\frac{x}{|x|^{2}}$ in the above identity,
	\begin{equation}\label{est:q0}
	\tilde{q}\left(\frac{x}{|x|^{2}}\right)=\frac{x_{4}}{|x|^{2}}+O\left(\frac{1}{|x|^{2}}\right)\quad \forall\ |x|>1.
	\end{equation}
	Therefore, from the definition of Kelvin transform in~\S\ref{SS:trans}, we have 
	\begin{equation*}
	Q(x)=\frac{1}{|x|^{2}}\tilde{q}\left(\frac{x}{|x|^{2}}\right)= \frac{x_{4}}{|x|^{4}}+O\left(\frac{1}{|x|^{4}}\right)\quad \forall \ |x|>1,
	\end{equation*}
	which means~\eqref{est:asQ}.
	
	Proof of (ii). On the one hand, from~\eqref{est:q0}, we know that 
	\begin{equation}\label{est:qx}
	\left|\tilde{q}\left(\frac{x}{|x|^{2}}\right)\right|\lesssim \frac{|x_{4}|}{|x|^{2}}+O\left(\frac{1}{|x|^{2}}\right)\lesssim \langle x \rangle^{-1},\quad \forall \ |x|>1.
	\end{equation}
	Then, from $\tilde{q}\in C^{\infty}$ and chain rule, for any $i=1,\dots,4$ and $|x|>1$, we have 
	\begin{equation*}
	\left|\partial_{x_{i}}\tilde{q}\left(\frac{x}{|x|^{2}}\right)\right|=\left|\sum_{j=1}^{4}\left((\partial_{x_{j}}\tilde{q})\left(\frac{x}{|x|^{2}}\right)\right)\partial_{x_{i}}\left(\frac{x_{j}}{|x|^{2}}\right)\right|\lesssim \langle x\rangle^{-2}.
	\end{equation*}
	Using a similar argument as above for all $\gamma\in \mathbb{N}^{4}$ and combining with~\eqref{est:qx},
	\begin{equation}\label{est:pxq}
	\left|\partial^{\gamma}_{x}\tilde{q}\left(\frac{x}{|x|^{2}}\right)\right|\lesssim \langle x \rangle^{-(1+|\gamma|)}\quad \forall\ |x|>1.
	\end{equation}
	On the other hand, by direct computation and chain rule,
	\begin{equation}\label{est:pxx2}
	\left|\partial_{x}^{\beta}\left(\frac{1}{|x|^{2}}\right)\right|\lesssim \langle x \rangle^{-(2+|\beta|)}
	\quad \forall\ \beta\in \mathbb{N}^{4}\ \mbox{and}\ |x|>1.
	\end{equation}
	Based on~\eqref{est:pxq},~\eqref{est:pxx2} and Leibniz’s formula~\eqref{equ:Lei}, for any $\alpha\in \mathbb{N}^{4}$ and $|x|>1$,
	\begin{equation*}
	\begin{aligned}
	\partial_{x}^{\alpha}Q(x)
	&=\sum_{\beta+\gamma=\alpha}C_{\alpha}^{\beta}\left(\partial^{\gamma}_{x}\tilde{q}\left(\frac{x}{|x|^{2}}\right)\right)\left(\partial_{x}^{\beta}\left(\frac{1}{|x|^{2}}\right)\right)\\
	&=\sum_{\beta+\gamma=\alpha}O\left(\langle x\rangle^{-(1+|\gamma|)}\langle x\rangle^{-(2+|\beta|)}\right)
	=O\left(\langle x\rangle^{-(3+|\alpha|)}\right).
	\end{aligned}
	\end{equation*}
	Combining above estimates with $Q\in C^{\infty}$, we complete the proof of (ii).
\end{proof}

Next, we introduce the following sharp understanding of the kernel of linearized operator around $Q$. 
\begin{lemma}[Asymptotic behavior of kernel functions]\label{le:asyker}
	There exist $K\in \mathbb{N}^{+}$ and linear independent $C^{\infty}$ functions  $\psi,\phi_{1},\dots,\phi_{K}\in \dot{H}^{1}$ such that 
	\begin{equation*}
	\widetilde{\mathcal{Z}}_{Q}={\rm{Span}}\left\{\psi,\phi_{1},\dots,\phi_{K}\right\}.
	\end{equation*}
	Moreover, we have 
	\begin{equation}\label{est:phi}
	\left|\psi(x)-|x|^{-2}\right|\lesssim \langle x\rangle ^{-3}\quad \forall\ |x|\ge 1,
	\end{equation}
	\begin{equation} \label{est:phipsi}
	\langle x\rangle^{-1} |\partial_{x}^{\alpha}\psi(x)|+	\sum_{k=1}^{K}\left|\partial_{x}^{\alpha}\phi_{k}(x)\right|\lesssim \langle x \rangle ^{-(3+\alpha)}\quad \forall\ \alpha\in\mathbb{N}^{4},\ x\in \RR^{4}.
	\end{equation}
\end{lemma}

\begin{remark}
	 Indeed, whatever is the decay on $Q$, its kernel contains the function $\psi$ which has the slow algebraic decay $|x|^{-2}$. This is the main difficulty in our proof. Similar difficulty in~\cite{HiRa} is due to the function $\Lambda W$ where $W\in \Sigma$ is the ground state. In~\cite{LJR}, for the wave maps, such similar difficulty is due to $\Lambda {Q}$ where ${Q}$ is the harmonic map.
\end{remark}
\begin{proof}
	We choose $\psi=\widetilde{\Omega}_{4}Q$ and $\left\{\phi_{1},\dots,\phi_{K}\right\}$ to be a basis of the function space
	\begin{equation*}
	{\rm{Span}}\left\{\Lambda Q;\partial_{x_{i}}Q,1\le i\le 4;\Omega_{ij}Q,1\le i<j\le 4; \widetilde{\Omega}_{i}Q,1\le i\le 3\right\}.
	\end{equation*}
	Note that, from the definition of $\psi$ and $\left\{\phi_{1},\dots,\phi_{K}\right\}$, we have 
	\begin{equation*}
	\widetilde{\mathcal{Z}}_{Q}={\rm{Span}}\left\{\psi,\phi_{1},\dots,\phi_{K}\right\}.
	\end{equation*}
	It remains to show that the $C^{\infty}$ functions $\left\{\psi,\phi_{1},\dots,\phi_{K}\right\}$ satisfy~\eqref{est:phi} and~\eqref{est:phipsi}.
	
	First, from (ii) of Lemma~\ref{le:asyQ}, for any kernel function 
	$$\phi \in {\rm{Span}}\left\{\Lambda Q;\partial_{x_{i}}Q,1\le i\le 4;\Omega_{ij}Q,1\le i<j\le 4\right\},$$ $\alpha\in \mathbb{N}^{4}$ and $x\in \RR^{4}$, we have 
	\begin{equation}\label{est:ker1}
	\left|\partial_{x}^{\alpha}\phi(x)\right|\lesssim \left|\partial_{x}^{\alpha}Q(x)\right|+\langle x\rangle \sum_{i=1}^{4}\left|\partial_{x}^{\alpha}\partial_{x_{i}}Q(x)\right|\lesssim \langle x\rangle^{-\left(3+|\alpha|\right)}.
	\end{equation}
	Second, we claim 
	\begin{equation}\label{est:OmQ}
	\left|\widetilde{\Omega}_{4}Q(x)-|x|^{-2}\right|\lesssim |x|^{-3}\quad \forall \ |x|>1,
	\end{equation}
	and for any $\alpha\in \mathbb{N}^{4}$,
	\begin{equation}\label{est:pxOmQ}
	\sum_{i=1}^{3}\left|\partial_{x}^{\alpha}\widetilde{\Omega}_{i}Q(x)\right|+\langle x \rangle^{-1}\left|\partial_{x}^{\alpha}\widetilde{\Omega}_{4}Q(x)\right|\lesssim \langle x \rangle^{-(3+|\alpha|)}\quad \mbox{on}\ \RR^{4}.
	\end{equation} 
	Applied the Taylor expansion~\eqref{equ:tay} for $\partial_{x_{i}}\tilde{q}$ at $x=0$ and using~\eqref{equ:q0}, we have, for $i=1,\dots,4$,
	\begin{equation}\label{est:pq}
	\left(\partial_{x_{i}}\tilde{q}\right)\left(\frac{x}{|x|^{2}}\right)=\delta_{i4}+O\left(\frac{1}{|x|}\right)\quad 
	\forall\ |x|>1.
	\end{equation}
	Moreover, for $i=1,\dots,4$ and any $\gamma\in \mathbb{N}^{4}$, we have 
	\begin{equation}\label{est:paq}
	\left|\partial_{x}^{\gamma}\left(\left(\partial_{x_{i}}\tilde{q}\right)\left(\frac{x}{|x|^{2}}\right)\right)\right|
	\lesssim \langle x\rangle ^{-(1+|\gamma|)+\delta_{i4}}\quad \forall\ |x|>1.
	\end{equation}
	Recall that, for $i=1,\dots,4$,
	\begin{equation*}
	\widetilde{\Omega}_{i}Q(x)=\frac{1}{|x|^{2}}\left((\partial_{x_{i}}\tilde{q})\left(\frac{x}{|x|^{2}}\right)\right)\quad \forall\ |x|>1.
	\end{equation*}
	Based on above identity and~\eqref{est:pq}, 
	\begin{equation*}
	\widetilde{\Omega}_{4}Q(x)=|x|^{-2}+O\left(|x|^{-3}\right)\quad \forall\ |x|>1,
	\end{equation*}
	which means~\eqref{est:OmQ}. Then, using~\eqref{est:paq} and Leibniz's formula~\eqref{equ:Lei}, for any $\alpha\in \mathbb{N}^{4}$, $|x|>1$ and $i=1,\dots,4$,
	\begin{equation*}
	\begin{aligned}
	\left|\partial_{x}^{\alpha}\widetilde{\Omega}_{i}Q(x)\right|
	&=\sum_{\beta+\gamma=\alpha}C_{\alpha}^{\beta}\partial_{x}^{\beta}\left(\frac{1}{|x|^{2}}\right)\partial_{x}^{\gamma}\left(\left(\partial_{x_{i}}\tilde{q}\right)\left(\frac{x}{|x|^{2}}\right)\right)\\
	&=\sum_{\beta+\gamma=\alpha}O\left(\langle x\rangle^{-(2+|\beta|)} \langle x\rangle^{-(1+\gamma)+\delta_{i4}}\right)=O\left( \langle x\rangle^{-(3+\alpha)+\delta_{i4}}\right).
	\end{aligned}
	\end{equation*}
	Based on above estimates and $Q\in C^{\infty}$, we obtain~\eqref{est:pxOmQ}. Note that, the $C^{\infty}$ functions $\left\{\psi,\phi_{1},\dots,\phi_{K} \right\}$ satisfying~\eqref{est:phi} and~\eqref{est:phipsi} is consequence of~\eqref{est:ker1},~\eqref{est:OmQ} and~\eqref{est:pxOmQ}. 
\end{proof}

\subsection{Spectral theory}\label{SS:spec}
In this section, we introduce the spectral properties of the linearized operator around $Q$ and $Q_{\ell}$. For $-1<\ell<1$, we define
\begin{equation*}
\mathcal{L}=-\Delta-3Q^{2},\quad \mathcal{L}_{\ell}=-(1-\ell^{2})\partial_{x_{1}}^{2}-\bar{\Delta}-3Q_{\ell}^{2}.
\end{equation*}

\begin{lemma}[Spectral properties of $\mathcal{L}$] ${\rm{(i)}}$ \emph{Spectrum}. The self-adjoint operator $\mathcal{L}$ has essential spectrum $[0,\infty)$, a finite number $J\ge 1$ of negative eigenvalues $($counted with multiplicity$)$, and its kernel is $\mathcal{\widetilde{Z}}_{Q}$. Let $\left(Y_{j}\right)_{j=1,\dots,J}$ be an $L^{2}$ orthogonal family of eigenvectors of $\mathcal{L}$ corresponding to the eigenvalues $(-\lambda_{j}^{2})_{j=1,\dots,J}$, i.e. for $j,j'=1,\dots,J$, 
	\begin{equation*}
	\left(Y_{j},Y_{j'}\right)_{L^{2}}=\delta_{jj'}\quad \mbox{and}\quad 
	\mathcal{L}Y_{j}=-\lambda_{j}^{2}Y_{j},\quad \lambda_{j}>0.
	\end{equation*}
	It holds, for all $j=1,\dots,J$ and $\alpha\in \mathbb{N}^{4}$, 
	\begin{equation}\label{est:Yj}
	\left|\partial_{x}^{\alpha}Y_{j}(x)\right|\lesssim e^{-\lambda_{j} |x|}\quad \mbox{on}\ \RR^{4}.
	\end{equation}

${\rm{(ii)}}$ \emph{Cancellation.} We have 
\begin{equation}\label{equ:can}
\left(f_{1}f_{2}f_{3},Q\right)_{L^{2}}=0\quad \mbox{for all}\ f_{1},f_{2},f_{3}\in \mathcal{\widetilde{Z}}_{Q}.
\end{equation}
\end{lemma}
\begin{proof}
	Proof of (i). The spectral properties of $\mathcal{L}$ are consequence of~\cite[Claim 3.5]{DKMASNSP} and ~\cite[Theorem 5.8.1]{Davies}. The exponential decay of eigenfunctions $Y_{j}$ follows from
	standard elliptic arguments, see~\cite[Theorem 3.2]{Agmon}.
	
	Proof of (ii). The cancellation properties~\eqref{equ:can} follows from taking twice partial derivatives of
	the identity $-\mathcal{T}_{\mathcal{A}}Q-(\mathcal{T}_{\mathcal{A}}Q)^{3}=0$ with respect to $\mathcal{A}=(\lambda,z,\xi,\theta)$ at $\mathcal{A}=(1,0,0,0)$. See details in the proof of~\cite[Lemma 2.1 (iii)]{Y5De}.
	\end{proof}

For $-1<\ell<1$, let 
\begin{equation*}
\begin{aligned}
\mathcal{H}_{\ell}=\left(\begin{array}{cc}
-\Delta-3Q_{\ell}^{2} & -\ell\partial_{x_{1}}\\
\ell\partial_{x_{1}} & 1
\end{array}\right),\quad 
{\mathrm{J}}=\left(\begin{array}{cc}
0 & 1\\
-1 & 0
\end{array}\right).
\end{aligned}
\end{equation*}
Similarly, for $k=,1\dots,K$,
\begin{equation*}
\vec{\psi}_{\ell}=\left(\begin{array}{c}
\psi(x_{\ell})\\
-\ell\partial_{x_{1}} \psi(x_{\ell})
\end{array}\right)\quad \mbox{and}\quad \vec{\phi}_{k,\ell}=\left(\begin{array}{c}
\phi_{k}(x_{\ell})\\
-\ell\partial_{x_{1}} \phi_{k}(x_{\ell})
\end{array}\right),
\end{equation*}
and for $j=1,\dots,J$,
\begin{equation*}
Y_{j,\ell}=Y_{j}(x_{\ell}),\quad 
\vec{\Upsilon}_{j,\ell}^{\pm}=\left(\begin{array}{c}
\Upsilon_{j,\ell}^{\pm,1}\\[3pt]
\Upsilon_{j,\ell}^{\pm,2}
\end{array}\right),\quad 
\vec{Z}_{j,\ell}^{\pm}=\mathcal{H}_{\ell}\vec{\Upsilon}^{\pm}_{j,\ell},
\end{equation*}
where
\begin{equation*}
\Upsilon_{j,\ell}^{\pm,1}=Y_{j,\ell}e^{\mp\frac{\ell\lambda_{j}}{\sqrt{1-\ell^{2}}}x_{1}},\quad 
\Upsilon_{j,\ell}^{\pm,2}=-\left(\ell\partial_{x_{1}}Y_{j,\ell}\mp\frac{\lambda_{j}}{\sqrt{1-\ell^{2}}}Y_{j,\ell}\right)e^{\mp\frac{\ell\lambda_{j}}{\sqrt{1-\ell^{2}}}x_{1}}.
\end{equation*}
We recall the following spectral properties of $\mathcal{H}_{\ell}$.
\begin{lemma} Let $-1<\ell<1$.
	\begin{enumerate}
	\item \emph{Decay properties of $\Upsilon_{j,\ell}^{\pm,1}$ and $\Upsilon_{j,\ell}^{\pm,2}$.}
	We have, on $\mathbb{R}^{4}$, for $j=1,\dots,J$,
	\begin{equation}\label{est:decY}
	\left|\Upsilon_{j,\ell}^{\pm,1}\right|+\left|\Upsilon_{j,\ell}^{\pm,2}\right|\lesssim e^{-\frac{1}{2}(1-|\ell|)^{\frac{1}{2}}\lambda_{j}|x|}.
	\end{equation}
	
	\item \emph{Kernel of $\mathcal{H}_{\ell}$}. We have 
	\begin{equation}\label{equ:ker}
	{\rm{Ker}}\mathcal{H}_{\ell}={\rm{Span}}\left(\vec{\psi}_{\ell},\vec{\phi}_{1,\ell},\dots,\vec{\phi}_{K,\ell}\right).
	\end{equation}
	\item \emph{Identities of $\vec{Z}_{j,\ell}^{\pm}$ and $\vec{Y}_{j,\ell}^{\pm}$.} For $j=1,\dots,J$, we have 
	\begin{equation}\label{equ:zy}
	\vec{Z}_{j,\ell}^{\pm}=\mp \lambda_{j}(1-\ell^{2})^{\frac{1}{2}}{\mathrm{J}}\vec{\Upsilon}_{j,\ell}^{\pm}.
	\end{equation}
	\item \emph{Coercivity of $\mathcal{H}_{\ell}$.} There exists $c>0$ such that, for all $\vec{v}\in \dot{H}^{1}\times L^{2}$,
	\begin{equation}\label{est:coer}
	\left(\mathcal{H}_{\ell}\vec{v},\vec{v}\right)_{L^{2}}\ge c \|\vec{v}\|_{\mathcal{H}}^{2}-c^{-1}\left(\sum_{\phi\in \mathcal{\widetilde{Z}}_{Q}}\left(\vec{v},\vec{\phi}_{\ell}\right)^{2}_{\mathcal{H}}+\sum_{\pm,j=1}^{J}\left(\vec{v},\vec{Z}^{\pm}_{j,\ell}\right)_{L^{2}}^{2}\right).
	\end{equation}
\end{enumerate}
\end{lemma}
\begin{proof}
	Proof of (i). We prove (i) for $\Upsilon_{j,\ell}^{\pm,1}$; the estimate for $\Upsilon_{j,\ell}^{\pm,2}$ are proved similarly. From the exponential decay property of $Y_{j}$ in~\eqref{est:Yj} and the definition of $\Upsilon_{j,\ell}^{\pm,1}$,
	\begin{equation*}
	\left|\Upsilon_{j,\ell}^{\pm,1}\right|\lesssim e^{-\lambda_{j}\left(\left(\frac{x^{2}_{1}}{1-\ell^{2}}+|\bar{x}|^{2}\right)^{\frac{1}{2}}-\frac{|\ell|}{\sqrt{1-\ell^{2}}}|x_{1}|\right)}, \quad \mbox{on}\ \RR^{4}.
	\end{equation*}
	By an elementary computation, we have 
	\begin{equation*}
	\begin{aligned}
	&\left(\left(\frac{x^{2}_{1}}{1-\ell^{2}}+|\bar{x}|^{2}\right)^{\frac{1}{2}}-\frac{|\ell|}{\sqrt{1-\ell^{2}}}|x_{1}|\right)^{2}\\
	&=\frac{1+\ell^{2}}{1-\ell^{2}}x_{1}^{2}+|\bar{x}|^{2}-2|\ell|\frac{|x_{1}|}{\sqrt{1-\ell^{2}}}\left(\frac{x^{2}_{1}}{1-\ell^{2}}+|\bar{x}|^{2}\right)^{\frac{1}{2}}\\
	&\ge \frac{(1-|\ell|)^{2}}{1-\ell^{2}}x_{1}^{2}+(1-|\ell|)|\bar{x}|^{2}> \frac{1}{4}(1-|\ell|)|x|^{2},
	\end{aligned}
	\end{equation*}
	which implies~\eqref{est:decY} for $\Upsilon_{j,\ell}^{\pm,1}$.
	
	Proof of (ii) and (iii). The proof of the identities directly follows from an elementary computation (see~\emph{e.g.}~\cite[Lemma 2.4 (i)]{Y5De} and~\cite[Claim 1 (ii)]{MM}) and we omit it.
	
	Proof of (iv). The proof of the coercivity property relies on an argument based on the spectral theory of $\mathcal{L}$ and compactness argument (see \emph{e.g.}~\cite{CMcorr} and~\cite[Proposition 2.3]{Y5De}), and we omit it.
	\end{proof}

Third, we introduce the following localized coercivity property of $\mathcal{H}_{\ell}$.
The proof is similar to~\cite[Lemma 2.2 (ii)]{MM}, but it is given for the sake of completeness and the readers' convenience.
\begin{lemma}[Localized coercivity of $\mathcal{H}_{\ell}$] \label{le:locacoer}
	For $\gamma>0$ small enough in~\eqref{def:zeta}, there exists $\mu>0$ such that for any $\vec{v}=(v,z)\in \dot{H}^{1}\times L^{2}$, 
	\begin{equation}\label{est:localcoer}
	\begin{aligned}
	&\int_{\RR^{4}}\left(|\nabla v|^{2}\zeta^{2}-3Q_{\ell}^{2}v^{2}+z^{2}\zeta^{2}+2\ell (\partial_{x_{1}} v)z \zeta^{2}\right) \d x\\
	&\ge \mu \int_{\RR^{4}}\left(|\nabla v|^{2}+z^{2}\right)\zeta^{2}\d x-\mu^{-1}\left(\sum_{\phi\in \mathcal{\widetilde{Z}}_{Q}}\left(\vec{v},\vec{\phi}_{\ell}\right)^{2}_{\mathcal{H}}+\sum_{\pm,j=1}^{J}\left(\vec{v},\vec{Z}^{\pm}_{j,\ell}\right)_{L^{2}}^{2}\right).
	\end{aligned}
	\end{equation}
\end{lemma}
\begin{proof}
    First, we apply~\eqref{est:coer} on $\vec{v}\zeta=(v\zeta,z\zeta)\in \dot{H}^{1}\times L^{2}$,
	\begin{equation}\label{est:coerHzeta}
	\left(\mathcal{H}_{\ell}(\vec{v}\zeta),\vec{v}\zeta\right)_{L^{2}}\ge c \|\vec{v}\zeta \|_{\E}^{2}-c^{-1}\left(\sum_{\phi\in \mathcal{\widetilde{Z}}_{Q}}\left(\vec{v}\zeta,\vec{\phi}_{\ell}\right)^{2}_{\mathcal{H}}+\sum_{\pm,j=1}^{J}\left(\vec{v}\zeta,\vec{Z}^{\pm}_{j,\ell}\right)_{L^{2}}^{2}\right).
	\end{equation}
	Note that, from the definition of $\mathcal{H}_{\ell}$ and integration by parts, 
	\begin{equation}\label{equ:Hvzeta}
	\begin{aligned}
	\left(\mathcal{H}_{\ell}(\vec{v}\zeta),\vec{v}\zeta\right)_{L^{2}}
	=&\int_{\RR^{4}}\left(|\nabla v|^{2}\zeta^{2}-3Q_{\ell}^{2}v^{2}+z^{2}\zeta^{2}+2\ell (\partial_{x_{1}} v)z \zeta^{2}\right) \d x\\
	&-\int_{\RR^{4}}v^{2}\zeta \Delta \zeta \d x+2\ell\int_{\RR^{4}}v z \zeta \partial_{x_{1}} \zeta \d x+3\int_{\RR^{4}}Q_{\ell}^{2}v^{2}(1-\zeta^{2})\d x.
	\end{aligned}
	\end{equation}
	Note also that, from the definition of $\zeta$ in~\eqref{def:zeta}, 
	\begin{equation*}
	\partial_{x_{1}}\zeta =-\gamma \frac{x_{1}}{1+|x|^{2}}\zeta\quad \mbox{and}\quad 
	\Delta \zeta =-\gamma((2-\gamma)|x|^{2}+4)\frac{\zeta}{(1+|x|^{2})^{2}}.
	\end{equation*}
	Therefore, by the Hardy inequality~\eqref{est:Hardy}, we have 
	\begin{equation*}
    \begin{aligned}
    \left|\int_{\RR^{4}}v^{2}\zeta \Delta \zeta \d x\right|&\lesssim \gamma \int_{\RR^{4}}\frac{(v\zeta)^{2}}{\langle x\rangle^{2}}\d x\lesssim \gamma \|v\zeta\|^{2}_{\dot{H}^{1}},\\
    \left|\int_{\RR^{4}}v z \zeta \partial_{x_{1}} \zeta \d x\right|&\lesssim \gamma \int_{\RR^{4}}\left|\frac{v\zeta}{\langle x \rangle}\right||z\zeta|\d x\lesssim \gamma \|\vec{v}\zeta\|_{\E}^{2}.
    \end{aligned}
    \end{equation*}
    Then, by the decay property of $Q$ in Lemma~\ref{le:asyQ} and the Hardy inequality~\eqref{est:Hardy},
    \begin{equation*}
    \begin{aligned}
    \left|\int_{\RR^{4}}Q_{\ell}^{2}v^{2}(1-\zeta^{2})\d x\right|
    &\lesssim \|Q_{\ell}^{2}\langle x\rangle^{2+\gamma} (1-(1+|x|^{2})^{-\gamma})\|_{L^{\infty}}\int_{\RR^{4}} \frac{(v\zeta)^{2}}{\langle x\rangle^{2}}\d x\\
    &\lesssim C(\gamma)\|v\zeta\|_{\dot{H}^{1}}^{2}\quad \mbox{where}\ C(\gamma)\to 0\ \mbox{as}\ \gamma\to 0.
    \end{aligned}
	\end{equation*}
	Combining above three estimates, for $\gamma>0$ small enough, we have 
	\begin{equation}\label{est:zetaerr}
	\left|-\int_{\RR^{4}}v^{2}\zeta \Delta \zeta \d x+2\ell\int_{\RR^{4}}v z \zeta \partial_{x_{1}} \zeta \d x+3\int_{\RR^{4}}Q_{\ell}^{2}v^{2}(1-\zeta^{2})\d x\right|\le \frac{c}{2}\|\vec{v}\zeta\|_{\E}^{2}.
	\end{equation}
	Using similar arguments to $\|\vec{v}\zeta\|_{\E}^{2}$, $(\vec{v}\zeta, \vec{\phi}_{\ell})^{2}_{\E}$ and $(\vec{v}\zeta,\vec{Z}^{\pm}_{j,\ell})^{2}_{L^{2}}$, we have
	\begin{equation}\label{est:vzeta}
	\begin{aligned}
	\int_{\RR^{4}}(|\nabla v|^{2}+z^{2})\zeta^{2}\d x&\le 2\|\vec{v}\zeta\|_{\E}^{2},\\
	(\vec{v}\zeta, \vec{\phi}_{\ell})_{\E}^{2}+(\vec{v}\zeta,\vec{Z}^{\pm}_{j,\ell})^{2}_{L^{2}}&\le 
	2(\vec{v}, \vec{\phi}_{\ell})_{\E}^{2}+2(\vec{v},\vec{Z}^{\pm}_{j,\ell})^{2}_{L^{2}}+\frac{c^{2}}{4}\|\vec{v}\zeta\|_{\E}^{2},
	\end{aligned}
	\end{equation}
	for $\gamma$ small enough. Combining~\eqref{est:coerHzeta},~\eqref{equ:Hvzeta},~\eqref{est:zetaerr} and~\eqref{est:vzeta}, we obtain~\eqref{est:localcoer}.
	\end{proof}

\section{Dynamics close to the sum of $N$ solitons }\label{S:dem}
Let $Q$ be the non-degenerate state chosen in Lemma~\ref{le:asyQ} and $N\ge 2$. For any $n=1,\dots,N$, let $\tau_{n}=\pm 1$ and $\bell_{n}=\ell_{n}\boldsymbol{\rm{e}}_{1}$ where $-1<\ell_{n}<1$ and $\ell_{n}\ne \ell_{n'}$ for any $n\ne n'$. Denote by $I$ and $I^{0}$ the following two sets of indices,
\begin{equation*}
\begin{aligned}
I&=\left\{(n,j):n=1,\dots,N,j=1,\dots,J\right\},\quad \ \ |I|={\rm{Card}} I=NJ,\\
I^{0}&=\left\{(n,k):n=1,\dots,N,k=1,\dots,K\right\},\quad |I^{0}|={\rm{Card}} I^{0}=NK.
\end{aligned}
\end{equation*}
Define
\begin{equation*}
Q_{n}(t,x)=\tau_{n}Q_{\ell_{n}}(x-\bell_{n}t),\quad 
\vec{Q}_{n}=\left(
\begin{array}{c}
Q_{n}\\
-\ell_{n}\partial_{x_{1}}Q_{n}
\end{array}
\right).
\end{equation*}
Similarly, we set,
\begin{equation*}
\Psi_{n}(t,x)=\psi_{\ell_{n}}(x-\bell_{n}t),\quad \vec{\Psi}_{n}(t,x)=\vec{\psi}_{\ell_{n}}(x-\bell_{n}t),
\end{equation*}
for $(n,k)\in I^{0}$,
\begin{equation*}
\Phi_{n,k}(t,x)=\phi_{k,\ell_{n}}(x-\bell_{n}t),\quad \vec{\Phi}_{n,k}(t,x)=\vec{\phi}_{k,\ell_{n}}(x-\bell_{n}t),
\end{equation*}
and for $(n,j)\in I$, 
\begin{equation*}
\vec{Z}_{n,j}^{\pm}(t,x)=\vec{Z}_{j,\ell_{n}}^{\pm}(x-\bell_{n} t)=\mp \lambda_{j}(1-\ell_{n}^{2})^{\frac{1}{2}}{\rm{J}}\vec{\Upsilon}_{j,\ell_{n}}^{\pm}(x-\bell_{n} t).
\end{equation*}
Consider time dependent $C^{1}$ functions $\boldsymbol{a}$ and $\boldsymbol{b}$ of the forms,
\begin{equation*}
\boldsymbol{a}=(a_{n})_{n=1,\dots,N}\in \RR^{N},\quad 
\boldsymbol{b}=(b_{n,k})_{(n,k)\in I^{0}}\in \RR^{NK}\quad \mbox{with}\ |\boldsymbol{a}|+|\boldsymbol{b}|\ll 1.
\end{equation*}
We introduce
\begin{equation*}
R=\sum_{n=1}^{N}Q_{n},\quad U=\sum_{n=1}^{N}a_{n}\Psi_{n},\quad V=\sum_{(n,k)\in I^{0}}b_{n,k}\Phi_{n,k},
\end{equation*}
and nonlinear interactions term
\begin{equation}\label{def:G}
G=\left(R+U+V\right)^{3}-\sum_{n=1}^{N}Q_{n}^{3}-3\sum_{n=1}^{N}a_{n}Q_{n}^{2}\Psi_{n}-3\sum_{(n,k)\in I^{0}}b_{n,k}Q_{n}^{2}\Phi_{n,k}.
\end{equation}
\subsection{Nonlinear interactions}
In this section, we introduce the estimates on the nonlinear interaction terms. In particular, due to sufficient decay along with the direction $\textbf{e}_{1}$, the order of the main term of interactions is $t^{-4}$ (see Lemma~\ref{le:Q1Q2}).

First, we expand $G$ of the components of the $R$, $U$ and $V$ by an elementary computation.
\begin{lemma}\label{le:G}
We have 
\begin{equation}\label{equ:G}
G=G_{1}+G_{2}+G_{3}=G_{1}+\sum_{n=1}^{N}G_{2,n}+\sum_{i=1}^{4}G_{3,i},
\end{equation}
where
\begin{equation*}
\begin{aligned}
G_{1}&=3\sum_{n\ne n'}Q_{n}^{2}Q_{n'}+6\sum_{n_{1}<n_{2}<n_{3}}Q_{n_{1}}Q_{n_{2}}Q_{n_{3}},\\
G_{2,n}&=3Q_{n}\bigg(a_{n}\Psi_{n}+\sum_{k=1}^{K}b_{n,k}\Phi_{n,k}\bigg)^{2}\quad  \mbox{for}\ n=1,\dots,N,
\end{aligned}
\end{equation*}
and
\begin{align*}
G_{3,1}&=3\sum_{n\ne n'}Q_{n}\bigg(a_{n'}\Psi_{n'}+\sum_{k=1}^{K}b_{n',k}\Phi_{n',k}\bigg)^{2},\\
G_{3,2}&=(U+V)^{3}=U^{3}+V^{3}+3U^{2}V+3UV^{2},\\
G_{3,3}&=3\sum_{n\ne n'}Q_{n}^{2}\bigg(a_{n'}\Psi_{n'}+\sum_{k=1}^{K}b_{n',k}\Phi_{n',k}\bigg)+3\sum_{n\ne n'}Q_{n}Q_{n'}(U+V),\\
G_{3,4}&=3R\sum_{n\ne n'}\bigg(a_{n}\Psi_{n}+\sum_{k=1}^{K}b_{n,k}\Phi_{n,k}\bigg)\bigg(a_{n'}\Psi_{n'}+\sum_{k'=1}^{K}b_{n',k'}\Phi_{n',k'}\bigg).
\end{align*}
\end{lemma}

\begin{remark}\label{re:G1G2G3}
	Note that $G_{1}$ is a pure interaction term and $G_{3}$ is a combination of the interaction terms and cubic terms, while $G_{2}$ is a pure non-interaction quadratic term and each component $G_{2,n}$ satisfies the traveling equation $\partial_{t} G_{2,n}\approx -\ell_{n}\partial_{x_{1}}G_{2,n}$ at the main order.
\end{remark}
\begin{proof}
	First, we decompose
	\begin{equation*}
	\begin{aligned}
	G
	=&(U+V)^{3}+\bigg(R^{3}-\sum_{n=1}^{N}Q_{n}^{3}\bigg)+3R(U+V)^{2}\\
	&+3\bigg(R^{2}(U+V)-\sum_{n=1}^{N}a_{n}Q_{n}^{2}\Psi_{n}
	-\sum_{(n,k)\in I^{0}}b_{n,k}Q_{n}^{2}\Phi_{n,k}\bigg).
	\end{aligned}
	\end{equation*}
	For the first two terms, by an elementary computation, 
	\begin{equation*}
	\begin{aligned}
	(U+V)^{3}&=U^{3}+V^{3}+3U^{2}V+3UV^{2},\\
	R^{3}-\sum_{n=1}^{N}Q_{n}^{3}&=3\sum_{n\ne n'}Q_{n}^{2}Q_{n'}+6\sum_{n_{1}<n_{2}<n_{3}}Q_{n_{1}}Q_{n_{2}}Q_{n_{3}}.
	\end{aligned}
	\end{equation*}
	For the third term, by the definition of $U$ and $V$, we decompose
	\begin{equation*}
	\begin{aligned}
	(U+V)^{2}=&\sum_{n=1}^{N}\bigg(a_{n}\Psi_{n}+\sum_{k=1}^{K}b_{n,k}\Phi_{n,k}\bigg)^{2}\\
	&+\sum_{n\ne n'}\bigg(a_{n}\Psi_{n}+\sum_{k=1}^{K}b_{n,k}\Phi_{n,k}\bigg)\bigg(a_{n'}\Psi_{n'}+\sum_{k'=1}^{K}b_{n',k'}\Phi_{n',k'}\bigg).
	\end{aligned}
	\end{equation*}
	Therefore,
	\begin{equation*}
	\begin{aligned}
	3R(U+V)^{2}
	=&3\sum_{n=1}^{N}Q_{n}\bigg(a_{n}\Psi_{n}+\sum_{k=1}^{K}b_{n,k}\Phi_{n,k}\bigg)^{2}\\
	&+3\sum_{n\ne n'}Q_{n}\bigg(a_{n'}\Psi_{n'}+\sum_{k=1}^{K}b_{n',k}\Phi_{n',k}\bigg)^{2}\\
	&+3R\sum_{n\ne n'}\bigg(a_{n}\Psi_{n}+\sum_{k=1}^{K}b_{n,k}\Phi_{n,k}\bigg)
	\bigg(a_{n'}\Psi_{n'}+\sum_{k'=1}^{K}b_{n',k'}\Phi_{n',k'}\bigg).
	\end{aligned}
	\end{equation*}
	For the last term, we have 
	\begin{equation*}
	\begin{aligned}
	R^{2}(U+V)
	&=\sum_{n=1}^{N}Q_{n}^{2}(U+V)+\sum_{n\ne n'}Q_{n}Q_{n'}(U+V).
	\end{aligned}
	\end{equation*}
	Thus, from the definition of $U$ and $V$, 
	\begin{equation*}
	\begin{aligned}
	&3\bigg(R^{2}(U+V)-\sum_{n=1}^{N}a_{n}Q_{n}^{2}\Psi_{n}
	-\sum_{(n,k)\in I^{0}}b_{n,k}Q_{n}^{2}\Phi_{n,k}\bigg)\\
	&=3\sum_{n\ne n'}Q_{n}^{2}\bigg(a_{n'}\Psi_{n'}+\sum_{k=1}^{K}b_{n',k}\Phi_{n',k}\bigg)+3\sum_{n\ne n'}Q_{n}Q_{n'}(U+V).
	\end{aligned}
	\end{equation*}
	Combining the above identities, we obtain the decomposition~\eqref{equ:G}.
\end{proof}
Second, we introduce a technical lemma for future reference.
\begin{lemma}\label{le:int}
	Let $f_{1}$ and $f_{2}$ be continuous functions such that 
	\begin{equation}\label{est:f1f2}
	|f_{1}(x)|+|f_{2}(x)|\lesssim \langle x \rangle ^{-2}\quad \mbox{on}\ \RR^{4}.
	\end{equation}
	Define
	\begin{equation*}
	f_{n_{1}}(t,x)=f_{1}\left(x-\bell_{n_{1}}t\right),\quad 
	f_{n_{2}}(t,x)=f_{2}\left(x-\bell_{n_{1}}t\right).
	\end{equation*}
	Let $0<\alpha_{1}\le \alpha_{2}$ be such that $\alpha_{1}+\alpha_{2}>2$. There exists $T_{0}\gg 1$ such that, for all $n_{1},n_{2}=1,\dots,N$ with $n_{1}\ne n_{2}$ and $t\ge T_{0}$, the following hold.
	\begin{enumerate}
		\item If $\alpha_{2}>2$,
		\begin{equation}\label{est:tech1}
		\int_{\RR^{4}}\left|f_{n_{1}}\right|^{\alpha_{1}}\left|f_{n_{2}}\right|^{\alpha_{2}}\d x\lesssim t^{-2\alpha_{1}}.
		\end{equation}
		
		\item If $\alpha_{2}<2$,
		\begin{equation}\label{est:tech2}
		\int_{\RR^{4}}\left|f_{n_{1}}\right|^{\alpha_{1}}\left|f_{n_{2}}\right|^{\alpha_{2}}\d x\lesssim t^{4-2\left( \alpha_{1}+\alpha_{2} \right) }.
		\end{equation}
		
			\item If $\alpha_{2}=2$,
		\begin{equation}\label{est:tech3}
		\int_{\RR^{4}}\left|f_{n_{1}}\right|^{\alpha_{1}}\left|f_{n_{2}}\right|^{\alpha_{2}}\d x\lesssim t^{-2\alpha_{1}}\log t.
		\end{equation}
	\end{enumerate}
\end{lemma}
\begin{proof} For $k=1,2$, set 
	\begin{equation*}
	r_{k}=x-\bell_{n_{k}}t \quad \mbox{and}\quad 
	\Omega_{n_{k}}(t,x)=\left\{x\in \RR^{4}:|r_{k}|\le 10^{-1}\left|\bell_{n_{1}}-\bell_{n_{2}}\right|t\right\}.
	\end{equation*}
	Let $T_{0}\gg 1$ large enough. For $t\ge T_{0}$, from the decay property of $f_{k}$ in~\eqref{est:f1f2},
		\begin{equation}\label{est:pointf11}
	\left|f_{n_{1}}(t,x)\right|\lesssim \langle  r_{1}\rangle^{-2}\lesssim t^{-2},\quad \quad \quad \ \ \quad \mbox{for}\ x\in \Omega_{n_{1}}^{C}(t,x),
	\end{equation}
		\begin{equation}\label{est:pointf21}
	\left|f_{n_{2}}(t,x)\right|\lesssim \langle  r_{2}\rangle^{-2}\lesssim t^{-2},\quad \quad \quad  \ \ \quad \mbox{for}\ x\in \Omega^{C}_{n_{2}}(t,x),
	\end{equation}
		\begin{equation}\label{est:pointf12}
	\left|f_{n_{1}}(t,x)\right|\lesssim \langle  r_{1}\rangle^{-2}\lesssim \left(\langle r_{2} \rangle +t\right)^{-2},\quad \mbox{for}\ x\in \Omega_{n_{2}}(t,x),
	\end{equation} 
	\begin{equation}\label{est:pointf22}
	\left|f_{n_{2}}(t,x)\right|\lesssim \langle  r_{2}\rangle^{-2}\lesssim \left(\langle r_{1} \rangle +t\right)^{-2},\quad \mbox{for}\ x\in \Omega_{n_{1}}(t,x).
	\end{equation}
	Proof of (i). 
	\emph{Case $0<\alpha_{1}\le 2<\alpha_{2}$.} From~\eqref{est:pointf11},
	\begin{equation*}
	\int_{\Omega_{n_{1}}^{C}}|f_{n_{1}}|^{\alpha_{1}}|f_{n_{2}}|^{\alpha_{2}}\d x\lesssim t^{-2\alpha_{1}}\int_{\Omega_{n_{1}}^{C}}|f_{n_{2}}|^{\alpha_{2}}\d x\lesssim t^{-2\alpha_{1}}.
	\end{equation*}
	Next, by~\eqref{est:f1f2},~\eqref{est:pointf22} and change of variable,
	\begin{equation*}
	\begin{aligned}
	\int_{\Omega_{n_1}}|f_{n_{1}}|^{\alpha_{1}}|f_{n_{2}}|^{\alpha_{2}}\d x
	&\lesssim \int_{\Omega_{n_1}}\langle r_{1}\rangle^{-2\alpha_{1}}\left(\langle r_{1}\rangle +t\right)^{-2\alpha_{2}}\d x\\
	&\lesssim \int_{\RR^{4}}\langle x\rangle^{-2\alpha_{1}}\left(\langle x\rangle +t\right)^{-2\alpha_{2}}\d x\\
	&\lesssim t^{4-2(\alpha_{1}+\alpha_{2})}\log t\lesssim t^{-2\alpha_{1}}.
	\end{aligned}
	\end{equation*}
	Combining the above estimates, we obtain~\eqref{est:tech1}.
	
	\emph{Case $2<\alpha_{1}\le\alpha_{2}$.} 
	By~\eqref{est:pointf11} and~\eqref{est:pointf22}, we have  
	\begin{equation*}
	\begin{aligned}
	\int_{\Omega^{C}_{n_1}}|f_{n_{1}}|^{\alpha_{1}}|f_{n_{2}}|^{\alpha_{2}}\d x
	&\lesssim t^{-2\alpha_{1}}\int_{\Omega^{C}_{n_1}}|f_{n_{2}}|^{\alpha_{2}}\d x\lesssim t^{-2\alpha_{1}},\\
	\int_{\Omega_{n_1}}|f_{n_{1}}|^{\alpha_{1}}|f_{n_{2}}|^{\alpha_{2}}\d x
	&\lesssim t^{-2\alpha_{2}}\int_{\Omega_{n_1}}|f_{n_{1}}|^{\alpha_{1}}\d x\lesssim t^{-2\alpha_{2}}\lesssim t^{-2\alpha_{1}},
	\end{aligned}
	\end{equation*}
	which implies~\eqref{est:tech1}.
	
	Proof of (ii). First, by~\eqref{est:f1f2},~\eqref{est:pointf12} and~\eqref{est:pointf22}, as before,
	\begin{equation*}
	\begin{aligned}
	&\int_{\Omega_{n_1}}|f_{n_{1}}|^{\alpha_{1}}|f_{n_{2}}|^{\alpha_{2}}\d x+\int_{\Omega_{n_2}}|f_{n_{1}}|^{\alpha_{1}}|f_{n_{2}}|^{\alpha_{2}}\d x\\
	&\lesssim \int_{\RR^{4}}\left(\langle x\rangle^{-2\alpha_{1}} (\langle x\rangle+t)^{-2\alpha_{2}}+\langle x\rangle^{-2\alpha_{2}} (\langle x\rangle+t)^{-2\alpha_{1}}\right)\d x\lesssim t^{4-2(\alpha_{1}+\alpha_{2})}.
	\end{aligned}
	\end{equation*}
	Next, by the H\"older inequality, we have 
	\begin{equation*}
	\begin{aligned}
	\int_{\left(\Omega_{n_1}\cup \Omega_{n_2}\right)^{C}}|f_{n_{1}}|^{\alpha_{1}}|f_{n_{2}}|^{\alpha_{2}}\d x
	&\lesssim \left(\int_{\Omega_{n_1}^{C}}|f_{n_{1}}|^{\alpha_{1}+\alpha_{2}}\d x\right)^{\frac{\alpha_{1}}{\alpha_{1}+\alpha_{2}}}\left(\int_{\Omega_{n_2}^{C}}|f_{n_{2}}|^{\alpha_{1}+\alpha_{2}}\d x\right)^{\frac{\alpha_{2}}{\alpha_{1}+\alpha_{2}}}\\
	&\lesssim \int_{|x|>\frac{1}{10}|\bell_{n_{1}}-\bell_{n_{2}}|t}\langle x\rangle^{-2(\alpha_{1}+\alpha_{2})}\d x\lesssim t^{4-2(\alpha_{1}+\alpha_{2})}.
	\end{aligned}
	\end{equation*}
	Combining the above estimates, we obtain~\eqref{est:tech2}.
	
	Proof of (iii). Using~\eqref{est:f1f2},~\eqref{est:pointf12} and~\eqref{est:pointf22} again,
	\begin{equation*}
	\begin{aligned}
	&\int_{\Omega_{n_1}}|f_{n_{1}}|^{\alpha_{1}}|f_{n_{2}}|^{2}\d x+\int_{\Omega_{n_2}}|f_{n_{1}}|^{\alpha_{1}}|f_{n_{2}}|^{2}\d x\\
	&\lesssim \int_{\RR^{4}}\left(\langle x\rangle^{-2\alpha_{1}} (\langle x\rangle+t)^{-4}+\langle x\rangle^{-4} (\langle x\rangle+t)^{-2\alpha_{1}}\right)\d x\lesssim t^{-2\alpha_{1}}\log t.
	\end{aligned}
	\end{equation*}
	Using a similar argument as in the proof of (ii), we have 
	\begin{equation*}
	\int_{\left(\Omega_{n_1}\cup \Omega_{n_2}\right)^{C}}|f_{n_{1}}|^{\alpha_{1}}|f_{n_{2}}|^{2}\d x
	\lesssim \int_{|x|>\frac{1}{10}|\bell_{n_{1}}-\bell_{n_{2}}|t}\langle x\rangle^{-2(\alpha_{1}+2)}\d x
	\lesssim t^{-2\alpha_{1}}.
	\end{equation*}
	Combining the above estimates, we obtain~\eqref{est:tech3}.
	\end{proof}
Now, we prove the following estimates on the nonlinear interaction terms.
\begin{lemma}\label{le:Q1Q2}
There exists $T_{0}\gg1 $, such that the following estimates hold for $t\ge T_{0}$.
\begin{enumerate}
	\item \emph{Estimate on $G_{1}$}. We have
	\begin{equation}\label{est:G1}
	\|G_{1}\|_{L^{2}}\lesssim  t^{-4}.
	\end{equation}
	\item \emph{Estimate on $G_{2}$}. We have
	\begin{equation}\label{est:G2}
	\|G_{2}\|_{L^{2}}\lesssim |\boldsymbol{a}|^{2}+|\boldsymbol{b}|^{2}.
	\end{equation}
	\item \emph{Estimate on $G_{3}$}. We have
	\begin{equation}\label{est:G3}
\|G_{3}\|_{L^{2}}\lesssim |\boldsymbol{a}|^{2}+|\boldsymbol{b}|^{3}+t^{-4}.
	\end{equation}
	\item \emph{Estimate on $G$.} We have 
	\begin{equation}\label{est:G}
	\|G\|_{L^{2}}\lesssim |\ba|^{2}+|\bb|^{2}+t^{-4}.
	\end{equation}
\end{enumerate}
\end{lemma}
\begin{remark}
	As mentioned in Remark~\ref{re:G1G2G3}, $G_{2}$ has a slow decay different from $G_{1}$ and $G_{3}$. This leads us to introduce an additional technical quantity $\mathcal{G}$ in the refined energy functional later (see also Remark~\ref{re:EPGJ}). Since the algebraic decay of $\psi$ and $\phi_{k}$ are different in Lemma~\ref{le:asyker}, the decay of interaction terms related to $\ba$ and $\bb$ are different, which implies the power of $\ba$ and $\bb$ are different in~\eqref{est:G3}.
\end{remark}
\begin{proof}
	Proof of (i). Let $t\ge T_{0}\gg 1$. We prove the following estimate,
	\begin{equation}\label{est:nnen}
	\sum_{n\ne n'}\|Q_{n}^{2}Q_{n'}\|_{L^{2}}\lesssim t^{-4}.
	\end{equation}
	Note that~\eqref{est:G1} is a consequence of~\eqref{est:nnen} and the AM-GM inequality.
	
	 First, for any $n\ne n'$, we consider change of variable
	\begin{equation*}
	\left(\frac{x_{1}-\ell_{n}t}{\sqrt{1-\ell^{2}_{n}}},\bar{x}\right)\mapsto 
	(y_{1},\bar{y})\in \RR^{4}.
	\end{equation*}
	It follows that
	\begin{equation*}
	\int_{\RR^{4}}Q_{n}^{4}Q_{n'}^{2}\d x=H_{1}+H_{2},
	\end{equation*}
	where
	\begin{equation*}
	\begin{aligned}
	H_{1}&=(1-\ell^{2}_{n})^{\frac{1}{2}}\int_{\RR^{3}}\int_{I_{1}}Q^{4}(x_{1},\bar{x})
	Q^{2}\left(\frac{(1-\ell^{2}_{n})^{\frac{1}{2}}x_{1}-(\ell_{n'}-\ell_{n})t}{\sqrt{1-\ell_{n'}^{2}}},\bar{x}\right)\d x_{1}\d \bar{x},\\
	H_{2}&=(1-\ell^{2}_{n})^{\frac{1}{2}}\int_{\RR^{3}}\int_{I_{2}}Q^{4}(x_{1},\bar{x})
	Q^{2}\left(\frac{(1-\ell^{2}_{n})^{\frac{1}{2}}x_{1}-(\ell_{n'}-\ell_{n})t}{\sqrt{1-\ell_{n'}^{2}}},\bar{x}\right)\d x_{1}\d \bar{x},
	\end{aligned}
	\end{equation*}
	and
	\begin{equation*}
	\begin{aligned}
	I_{1}&=\left\{x_{1}\in \RR:(1-\ell_{n}^{2})^{\frac{1}{2}}|x_{1}|\ge\frac{1}{2}|\ell_{n}-\ell_{n'}|t\right\},\\
	I_{2}&=\left\{x_{1}\in \RR:(1-\ell^{2}_{n})^{\frac{1}{2}}|x_{1}|< \frac{1}{2}|\ell_{n}-\ell_{n'}|t\right\}.
	\end{aligned}
	\end{equation*}
	\emph{Estimate on $H_{1}$.} First, from~\eqref{est:aspQ}, for any $x=(x_{1},\bar{x})\in I_{1}\times \RR^{3}$,
	\begin{equation*}
	|Q(x_{1},\bar{x})|\lesssim (t+|\bar{x}|)^{-3}\lesssim t^{-3}.
	\end{equation*}
	Based on above estimate and $Q\in L^{2}$, we have 
	\begin{equation*}
     H_{1}\lesssim \|Q\|^{4}_{L^{\infty}(I_{1}\times \RR^{3})}\int_{\RR^{3}}\int_{I_{1}}Q^{2}\left(\frac{(1-\ell^{2}_{n})^{\frac{1}{2}}x_{1}-(\ell_{n'}-\ell_{n})t}{\sqrt{1-\ell_{n'}^{2}}},\bar{x}\right)\d x_{1}\d \bar{x}\lesssim  t^{-12}.
	\end{equation*}
	\emph{Estimate on $H_{2}$.} Note that, for any $x_{1}\in I_{2}$,
	\begin{equation*}
	\begin{aligned}
	\left|\frac{(1-\ell^{2}_{n})^{\frac{1}{2}}x_{1}
		-(\ell_{n'}-\ell_{n})t}{(1-\ell_{n'}^{2})^{\frac{1}{2}}}\right|
	&\ge\left|\frac{(\ell_{n'}-\ell_{n})t}{(1-\ell_{n'}^{2})^{\frac{1}{2}}}\right|-\left|\frac{(1-\ell^{2}_{n})^{\frac{1}{2}}x_{1}}{(1-\ell_{n'}^{2})^{\frac{1}{2}}}\right|\ge \frac{|\ell_{n'}-\ell_{n}|t}{2(1-\ell_{n'}^{2})^{\frac{1}{2}}}.
	\end{aligned}
	\end{equation*} By~\eqref{est:asQ}, for any $x=(x_{1},\bar{x})\in I_{2}\times \RR^{3}$, we have 
	\begin{equation*}
	\left|Q\left(\frac{(1-\ell^{2}_{n})^{\frac{1}{2}}x_{1}-(\ell_{n'}-\ell_{n})t}{\sqrt{1-\ell_{n'}^{2}}},\bar{x}\right)\right|\lesssim \frac{|\bar{x}|}{|\bar{x}|^{4}+t^{4}}+t^{-4}\lesssim t^{-4}(|x|+1).
	\end{equation*}
	Based on above estimate and~\eqref{est:aspQ}, 
	\begin{equation*}
	H_{2}\lesssim t^{-8}\int_{\RR^{3}}\int_{\RR}\langle x\rangle^{-12}(|x|^{2}+1)\d x_{1}\d \bar{x}\lesssim t^{-8}.
	\end{equation*}
	Combining the estimates on $H_{1}$ and $H_{2}$, we obtain~\eqref{est:nnen}.
	
		Proof of (ii). First, from~\eqref{est:aspQ} and~\eqref{est:phipsi}, we know that 
	\begin{equation*}
	|Q|\psi^{2}+\sum_{k=1}^{K}|Q|\phi_{k}^{2}\in L^{2}.
	\end{equation*}
	Next, from the AM-GM inequality,
	\begin{equation*}
	|G_{2}|\lesssim  \sum_{n=1}^{N}|a_{n}|^{2}|Q_{n}||\Psi_{n}|^{2}+\sum_{(n,k)\in I}|Q_{n}||b_{n,k}|^{2}|\Phi_{n,k}|^{2}.
	\end{equation*}
	It follows that,
	\begin{equation*}
	\|G_{2}\|_{L^{2}}\lesssim \bigg(\sum_{n=1}^{N}a_{n}^{2}+\sum_{(n,k)\in I^{0}}b_{n,k}^{2}\bigg)\left(\|Q\psi^{2}\|_{L^{2}}+\sum_{k=1}^{K}\|Q\phi_{k}^{2}\|_{L^{2}}\right)\lesssim |\boldsymbol{a}|^{2}+|\boldsymbol{b}|^{2}.
	\end{equation*}
	
	Proof of (iii). \textbf{Step 1.} Estimate on $G_{3,1}$.  We claim
	\begin{equation}\label{est:G31}
	\|G_{3,1}\|_{L^{2}}\lesssim |\ba|^{8}+|\bb|^{8}+t^{-4}.
	\end{equation}
	From the AM-GM inequality, 
	\begin{equation*}
	|G_{3,1}|\lesssim \sum_{n\ne n'}|Q_{n}|\bigg(a_{n'}^{2}\Psi^{2}_{n'}+\sum_{k=1}^{K}b_{n',k}^{2}\Phi^{2}_{n',k}\bigg).
	\end{equation*}
	Note that, by~\eqref{est:aspQ},~\eqref{est:phipsi},~\eqref{est:tech1} and the Young's inequality,
	\begin{equation*}
	\begin{aligned}
	\sum_{n\ne n'}\|a_{n'}^{2}\Psi^{2}_{n'}Q_{n}\|_{L^{2}}
	&\lesssim |\ba|^{2}t^{-3}\lesssim |\ba|^{8}+t^{-4},\\
	\sum_{n\ne n'}\sum_{k=1}^{K}\|b_{n',k}^{2}\Phi^{2}_{n',k}Q_{n}\|_{L^{2}}
	&\lesssim |\bb|^{2}t^{-3}\lesssim |\bb|^{8}+t^{-4}.
	\end{aligned}
	\end{equation*}
	Combining above estimates, we find~\eqref{est:G31}.
	
	\textbf{Step 2.} Estimate on $G_{3,2}$. We claim
	\begin{equation}\label{est:G32}
	\|G_{3,2}\|_{L^{2}}\lesssim |\boldsymbol{a}|^{3}+|\boldsymbol{b}|^{3}.
	\end{equation}
	First, from~\eqref{est:ker1}, we know that 
	\begin{equation*}
	|\psi|+\sum_{k=1}^{K}|\phi_{k}|\in L^{6}.
	\end{equation*}
	Note that, from the AM-GM inequality, 
	\begin{equation*}
	|G_{3,2}|\lesssim |U|^{3}+|V|^{3}
	\lesssim \sum_{n=1}^{N}|a_{n}|^{3}|\Psi_{n}|^{3}+\sum_{(n,k)\in I^{0}}|b_{n,k}|^{3}|\Phi_{n,k}|^{3}.
	\end{equation*}
	It follows that,
	\begin{equation*}
	\|G_{3,2}\|_{L^{2}}\lesssim (|\ba|^{3}+|\bb|^{3})\left(\|\psi\|_{L^{6}}^{3}+\sum_{k=1}^{K}\|\phi_{k}\|_{L^{6}}^{3}\right)
	\lesssim |\ba|^{3}+|\bb|^{3},
	\end{equation*}
	which means~\eqref{est:G32}.

	\textbf{Step 3.} Estimate on $G_{3,3}$. We claim
	\begin{equation}\label{est:G33}
	\|G_{3,3}\|_{L^{2}}\lesssim |\ba|^{2}+|\bb|^{4}+t^{-4}.
	\end{equation}
	First, using the AM-GM inequality again,
	\begin{equation*}
	\begin{aligned}
	|G_{3,3}|\lesssim &\sum_{n\ne n'}|a_{n'}Q_{n}^{2}\Psi_{n'}|+\sum_{n\ne n'}|a_{n}Q_{n}Q_{n'}\Psi_{n}|\\
	&+\sum_{k=1}^{K}\sum_{n\ne n'}|b_{n',k}Q_{n}^{2}\Phi_{n',k}|
	+\sum_{k=1}^{K}\sum_{n\ne n'}|b_{n,k}Q_{n}Q_{n'}\Phi_{n,k}|.
	\end{aligned}
	\end{equation*}
	Second, using~\eqref{est:aspQ},~\eqref{est:phipsi},~\eqref{est:tech1} and the Young's inequality again,
	\begin{equation*}
	\sum_{n\ne n'}\|a_{n'}Q_{n}^{2}\Psi_{n'}\|_{L^{2}}+\sum_{n\ne n'}\|a_{n}Q_{n}Q_{n'}\Psi_{n}\|_{L^{2}}
	\lesssim |\ba|(t^{-2}+t^{-3})\lesssim |\ba|^{2}+t^{-4},
	\end{equation*}
	\begin{equation*}
	\sum_{k=1}^{K}\sum_{n\ne n'}\|b_{n',k}Q_{n}^{2}\Phi_{n',k}\|_{L^{2}}
	+\sum_{k=1}^{K}\sum_{n\ne n'}\|b_{n,k}Q_{n}Q_{n'}\Phi_{n,k}\|_{L^{2}}
	\lesssim |\bb|t^{-3}\lesssim |\bb|^{4}+t^{-4}.
	\end{equation*}
	Combining above estimates, we obtain~\eqref{est:G33}.
	
	\textbf{Step 4.} Estimate on $G_{3,4}$. We claim
	\begin{equation}\label{est:G34}
	\|G_{3,4}\|_{L^{2}}\lesssim |\ba|^{4}+|\bb|^{4}+t^{-4}.
		\end{equation}
	By the AM-GM inequality and the definition of $R$,  
	\begin{equation*}
	\begin{aligned}
	|G_{3,4}|
	\lesssim &\sum_{n\ne n'}|Q_{n}|\bigg(a_{n'}^{2}\Psi_{n'}^{2}+\sum_{k=1}^{K}b_{n',k}^{2}\Phi_{n',k}^{2}\bigg)\\
	 &+\sum_{n\ne n'}|a_{n}||Q_{n}||\Psi_{n}|\bigg(|a_{n'}||\Psi_{n'}|
	 +\sum_{k=1}^{K}|b_{n',k}||\Phi_{n',k}|\bigg)\\
	 &+\sum_{n\ne n'}\sum_{k=1}^{K}|b_{n,k}||Q_{n}||\Phi_{n,k}|\bigg(|a_{n'}||\Psi_{n'}|
	 +\sum_{k=1}^{K}|b_{n',k'}||\Phi_{n',k'}|\bigg).
	\end{aligned}
	\end{equation*}
	Arguing as in the Step 2 and Step 3, using~\eqref{est:aspQ},~\eqref{est:phipsi},~\eqref{est:tech1} and the Young's inequality, we obtain~\eqref{est:G34}.
	
	We see that~\eqref{est:G3} follows from~\eqref{est:G31},~\eqref{est:G32},~\eqref{est:G33} and~\eqref{est:G34}.
	
	Proof of (iv). Estimate~\eqref{est:G} is a consequence of~\eqref{est:G1},~\eqref{est:G2} and~\eqref{est:G3}.
	\end{proof}

\subsection{Decomposition of the solution around multi-solitary wave}\label{SS:dec}
In this section, we recall general results on solutions of~\eqref{equ:wave} that are close to sum of $N\ge 2$ decoupled solitary waves.

First, we recall a decomposition result for solutions of~\eqref{equ:wave}.
\begin{proposition}\label{pro:dec}
	There exist $T_{0}\gg 1$ and $0<\gamma_{0}\ll 1$ such that if $\vec{u}(t)=(u(t),\partial_{t} u(t))$ is a solution of~\eqref{equ:wave} on $[T_{1},T_{2}]$, where $1\ll T_{0}\le T_{1}<T_{2}<\infty$, such that
	\begin{equation}\label{est:T1T2}
	\sup_{t\in [T_{1},T_{2}]}\|\vec{u}(t)-\sum_{n=1}^{N}\vec{Q}_{n}(t)\|_{}<\gamma_{0},
	\end{equation}
	then there exist $C^{1}$ functions $\ba=(a_{1},\dots,a_N)$ and $\bb=(b_{n,k})_{(n,k)\in I^{0}}$ such that, $\vec{\varphi}$ being defined by 
	\begin{equation}\label{def:vp}
	\vec{\varphi}=
	\left(\begin{array}{c}
	\vp\\ \vpp
	\end{array}\right)=\vec{u}-\sum_{n=1}^{N}\vec{Q}_{n}-\sum_{n=1}^{N}a_{n}\vec{\Psi}_{n}-\sum_{(n,k)\in I^{0}}b_{n,k}\vec{\Phi}_{n,k},
	\end{equation}
	it satisfies
	\begin{equation*}
	\|\vec{\varphi}\|_{\E}+|\ba|+|\bb|\lesssim \gamma_{0},
	\end{equation*}
	and
	\begin{equation}\label{equ:orth1}
	\begin{aligned}
	\left(\vec{\varphi},\vec{\Phi}_{n',k}\right)_{\E}&=0, \quad \forall(n',k)\in I^{0},\\
	\left(\vec{\varphi},\vec{\Psi}_{n}\right)_{\E}&=0,\quad \forall\  n=1,\dots,N.
	\end{aligned}
	\end{equation}
	\end{proposition}
\begin{proof}
	The proof of the decomposition proposition relies on a standard argument based on Lemma~\ref{le:asyker}, (ii) and (iii) of Lemma~\ref{le:int} and the invertibility of Gram matrix (see~\emph{e.g.}~\cite[Lemma 3]{CMTAMS} and~\cite[Proposition 3.4]{Y5De}), and we omit it.
	\end{proof}
Set 
\begin{equation*}
\begin{aligned}
{\rm{Mod}}_{1}&={\rm{Mod}}_{1,1}+{\rm{Mod}}_{1,2},\\
{\rm{Mod}}_{2}&={\rm{Mod}}_{2,1}+{\rm{Mod}}_{2,2},
\end{aligned}
\end{equation*}
where
\begin{equation*}
\begin{aligned}
{\rm{Mod}_{1,1}}&=\sum_{n=1}^{N}\dot{a}_{n}{\Psi}_{n},\quad \quad \quad \ \ 
{\rm{Mod}}_{2,1}=-\sum_{n=1}^{N}\dot{a}_{n}\ell_{n}\partial_{x_{1}}\Psi_{n},\\
{\rm{Mod}_{1,2}}&=\sum_{(n,k)\in I^{0}}\dot{b}_{n,k}{\Phi}_{n,k},\quad 
{\rm{Mod}}_{2,2}=-\sum_{(n,k)\in I^{0}}\dot{b}_{n,k}\ell_{n}\partial_{x_{1}}\Phi_{n,k}.
\end{aligned}
\end{equation*}
We deduce the equation of $\vec{\varphi}$ from~\eqref{equ:wave} and~\eqref{def:vp}.
\begin{lemma}[Equation of $\vec{\varphi}$]
	The function $\vec{\varphi}$ satisfies 
\begin{equation}\label{equ:vp}
\left\{ \begin{aligned}
&\partial_t\vp=\vpp-{\rm{Mod}}_{1},\\
&\partial_{t}\vpp=\Delta\vp+(R+U+V+\vp)^{3}-(R+U+V)^{3}
-{\rm{Mod}_{2}}+G.
\end{aligned}\right.
\end{equation}
\end{lemma}
\begin{proof}
	First, from the definition of $Q_{n}$, $\Psi_{n}$ and $\Phi_{n,k}$, 
	\begin{equation*}
	\begin{aligned}
	\partial_{t}Q_{n}&=-\ell_{n}\partial_{x_{1}}Q_{n},\\
	\partial_{t}(a_{n}\Psi_{n})&=-\ell_{n}\partial_{x_{1}}\Psi_{n}+\dot{a}_{n}\Psi_{n},\\
	\partial_{t}(b_{n,k}\Phi_{n,k})&=-\ell_{n}\partial_{x_{1}}\Phi_{n,k}+\dot{b}_{n,k}\Phi_{n,k},
	\end{aligned}
	\end{equation*}
	for $n=1,\dots,N$ and $(n,k)\in I^{0}$.
	Therefore, by the definition of $\vec{\varphi}=(\vp,\vpp)$ in~\eqref{def:vp}, we have 
	\begin{equation*}
	\begin{aligned}
	\partial_{t} \vp
	&=\partial_{t} u-\sum_{n=1}^{N}\pt Q_{n}-\sum_{n=1}^{N}\pt (a_{n}\Psi_{n})-\sum_{(n,k)\in I^{0}}\pt (b_{n,k}\Phi_{n,k})\\
	&=\vpp-\sum_{n=1}^{N}\dot{a}_{n}\Psi_{n}-\sum_{(n,k)\in I^{0}}\dot{b}_{n,k}\Phi_{n,k}.
	\end{aligned}
	\end{equation*}
	Second, using the definition of $Q_{n}$, $\Psi_{n}$ and $\Phi_{n,k}$ again, 
	\begin{equation*}
	\begin{aligned}
	\partial_{t}(\ell_{n}\partial_{x_{1}}Q_{n})
	&=-\ell_{n}^{2}\partial_{x_{1}}^{2}Q_{n},\\
	\partial_{t}(a_{n}\ell_{n}\partial_{x_{1}}\Psi_{n})
	&=-{a}_{n}\ell_{n}^{2}\partial_{x_{1}}^{2}\Psi_{n}+\dot{a}_{n}\ell_{n}\partial_{x_{1}}\Psi_{n},\\
	\pt (b_{n,k}\ell_{n}\partial_{x_{1}}\Phi_{n,k})
	&=-b_{n,k}\ell_{n}^{2}\partial^{2}_{x_{1}}\Phi_{n,k}+
	\dot{b}_{n,k}\ell_{n}\partial_{x_{1}}\Phi_{n,k},
	\end{aligned}
	\end{equation*}
	for $n=1,\dots,N$ and $(n,k)\in I^{0}$. Thus, by~\eqref{equ:wave} and the definition of ${\rm{Mod}}_{2}$,
	\begin{equation*}
	\begin{aligned}
	\pt \vpp
	&=\partial_{tt}u+\sum_{n=1}^{N}\partial_{t}(\ell_{n}\partial_{x_{1}}Q_{n})
	+\sum_{n=1}^{N}\partial_{t}(a_{n}\ell_{n}\partial_{x_{1}}\Psi_{n})
	+\sum_{(n,k)\in I^{0}}\pt (b_{n,k}\ell_{n}\partial_{x_{1}}\Phi_{n,k})\\
	&=\Delta u+u^{3}-\sum_{n=1}^{N}\ell_{n}^{2}\partial_{x_{1}}^{2}Q_{n}
	-\sum_{n=1}^{N}{a}_{n}\ell_{n}^{2}\partial_{x_{1}}^{2}\Psi_{n}
	-\sum_{(n,k)\in I^{0}}b_{n,k}\ell_{n}^{2}\partial^{2}_{x_{1}}\Phi_{n,k}-{\rm{Mod}}_{2}.
	\end{aligned}
	\end{equation*}
	Note that, from $-(1-\ell_{n}^{2})\partial_{x_{1}}^{2}Q_{n}-\bar{\Delta}Q_{n}-Q_{n}^{3}=0$,
	$-(1-\ell_{n}^{2})\partial_{x_{1}}^{2}\Psi_{n}-\bar{\Delta}\Psi_{n}-3Q_{n}^{2}\Psi_{n}=0$,
	$-(1-\ell_{n}^{2})\partial_{x_{1}}^{2}\Phi_{n,k}-\bar{\Delta}\Phi_{n,k}-3Q_{n}^{2}\Phi_{n,k}=0$ and the definition of $G$ in~\eqref{def:G},
	\begin{equation*}
	\begin{aligned}
	&\Delta u+u^{3}-\sum_{n=1}^{N}\ell_{n}^{2}\partial_{x_{1}}^{2}Q_{n}
	-\sum_{n=1}^{N}{a}_{n}\ell_{n}^{2}\partial_{x_{1}}^{2}\Psi_{n}
	-\sum_{(n,k)\in I^{0}}b_{n,k}\ell_{n}^{2}\partial^{2}_{x_{1}}\Phi_{n,k}\\
	&=\Delta \vp +(R+U+V+\vp)^{3}-(R+U+V)^{3}+G.
	\end{aligned}
	\end{equation*}
	Based on the above identities, we obtain
	\begin{equation*}
	\pt \vpp=\Delta \vp+(R+U+V+\vp)^{3}-(R+U+V)^{3}-{\rm{Mod}}_{2}+G.
	\end{equation*}
	\end{proof}
We decompose
\begin{equation*}
(R+U+V+\vp)^{3}-(R+U+V)^{3}=3\sum_{n=1}^{N}Q_{n}^{2}\vp+R_{1}+R_{2},
\end{equation*}
where
\begin{equation*}
\begin{aligned}
R_{1}&=3(R+U+V)\vp^{2}+\vp^{3},\\
R_{2}&=3\left((U+V)^{2}+2R(U+V)\right)\vp+3\sum_{n\ne n'}Q_{n}Q_{n'}\vp.
\end{aligned}
\end{equation*}
Based on the above decomposition, the equation of $\vec{\varphi}$ in~\eqref{equ:vp} can be rewritten as
\begin{equation}\label{equ:vvp}
\pt \vec{\varphi}=\vec{\mathcal{L}}\vec{\varphi}-\vec{{\rm{Mod}}}+\vec{G}+\vec{R}_{1}+\vec{R}_{2},
\end{equation}
where
\begin{equation*}
\vec{\mathcal{L}}=\left(\begin{array}{cc}
0&1\\
\Delta+3\sum_{n=1}^{N}Q_{n}^{2}&0
\end{array}\right),\quad 
{\vec{\rm{Mod}}}=\left(\begin{array}{c}
{\rm{Mod}}_{1}\\
{\rm{Mod}}_{2}
\end{array}\right),
\end{equation*}
and
\begin{equation*}
\vec{G}=\left(\begin{array}{c}
0\\G
\end{array}\right),\quad 
\vec{R}_{1}=\left(\begin{array}{c}
0\\R_{1}
\end{array}\right),\quad 
\vec{R}_{2}=\left(\begin{array}{c}
0\\R_{2}
\end{array}\right).
\end{equation*}
Now, we derive the control of $\dot{\ba}$ and $\dot{\bb}$ from the orthogonality conditions
~\eqref{equ:orth1}.
\begin{lemma}\label{le:equab}
	In the context of Proposition~\ref{pro:dec}, we have
	\begin{equation}\label{est:ab}
	|\dot{\ba}|+|\dot{\bb}|\lesssim \|\vec{\varphi}\|_{\E}+|\ba|^{2}+|\bb|^{2}+t^{-4}.
	\end{equation}
\end{lemma}
\begin{proof}
	From~\eqref{equ:vvp} and the orthogonality condition~\eqref{equ:orth1}, we have 
	\begin{equation*}
	\begin{aligned}
	0
	=\frac{\d}{\d t}\left(\vec{\varphi},\vec{\Psi}_{n}\right)_{\E}
	=&\left(\pt \vec{\varphi},\vec{\Psi}_{n}\right)_{\E}+\left(\vec{\varphi},\pt \vec{\Psi}_{n}\right)_{\E}\\
	=&\left(\vec{\mathcal{L}}\vec{\varphi},\vec{\Psi}_{n}\right)_{\E}-
	\left(\vec{{\rm{Mod}}},\vec{\Psi}_{n}\right)_{\E}+\left(\vec{G},\vec{\Psi}_{n}\right)_{\E}\\
	&+\left(\vec{R}_{1},\vec{\Psi}_{n}\right)_{\E}+\left(\vec{R}_{2},\vec{\Psi}_{n}\right)_{\E}
	-\ell_{n}\left(\vec{\varphi},\partial_{x_{1}}\vec{\Psi}_{n}\right)_{\E}.
	\end{aligned}
	\end{equation*}
	By integration by parts and the decay property of $\psi$ in~\eqref{est:phipsi}, the first term is
	\begin{equation*}
	\begin{aligned}
	\left(\vec{\mathcal{L}}\vec{\varphi},\vec{\Psi}_{n}\right)_{\E}
	=&\left(\vp,\bigg(\Delta+3\sum_{n'=1}^{N}Q_{n'}^{2}\bigg)(-\ell_{n}\partial_{x_{1}}\Psi_{n})\right)_{L^{2}}
	-\left(\vpp,\Delta\Psi_{n}\right)_{L^{2}}=O\left(\|\vec{\varphi}\|_{\E}\right).
	\end{aligned}
	\end{equation*}
	Next, from the decay properties of $\psi$ and $\phi$ in~\eqref{est:phipsi}, and Lemma~\ref{le:int} (ii) and (iii), 
	\begin{equation*}
	\begin{aligned}
	\left(\vec{\rm{Mod}},\vec{\Psi}_{n}\right)_{\E}
	=&\dot{a}_{n}\left(\vec{\Psi}_{n},\vec{\Psi}_{n}\right)_{\E}+\sum_{k=1}^{K}\dot{b}_{n,k}\left(\vec{\Phi}_{n,k},\vec{\Psi}_{n}\right)_{\E}+O\left(t^{-2}(|\dot{\ba}|+|\dot{\bb}|)\right).
	\end{aligned}
	\end{equation*}
    Using~\eqref{est:G}, the decay property of $\psi$ and the Cauchy-Schwarz inequality,
    \begin{equation*}
    \begin{aligned}
    \left|\left(\vec{G},\vec{\Psi}_{n}\right)_{\E}\right|
    &\lesssim \|G\|_{L^{2}}\|\partial_{x_{1}}\Psi_{n}\|_{L^{2}}\lesssim |\ba|^{2}+|\bb|^{2}+t^{-4}.
    \end{aligned}
    \end{equation*}
    Then, from the expansion of $R_{1}$ and $R_{2}$,  and the Sobolev inequality~\eqref{est:Sobo},
    \begin{equation*}
    \begin{aligned}
    \left|\left(\vec{R}_{1},\vec{\Psi}_{n}\right)_{\E}\right|
    &\lesssim \|\vp\|_{L^{4}}^{2}\|\partial_{x_{1}}\Psi_{n}\|_{L^{4}}\left(\|\vp\|_{L^{4}}+\|R+U+V\|_{L^{4}}\right)\lesssim \|\vec{\varphi}\|_{\E}^{2},\\
    \left|\left(\vec{R}_{2},\vec{\Psi}_{n}\right)_{\E}\right|&\lesssim \|\vp\|_{L^{4}}\|\partial_{x_{1}}\Psi_{n}\|_{L^{4}}\left(\|U\|^{2}_{L^{4}}+\|V\|_{L^{4}}^{2}+\|Q\|_{L^{4}}^{2}\right)\lesssim \|\vec{\varphi}\|_{\E}.
   \end{aligned}
    \end{equation*}
    Last, using the Cauchy-Schwarz inequality again,
    \begin{equation*}
    \begin{aligned}
    \left|\left(\vec{\varphi},\partial_{x_{1}}\vec{\Psi}_{n}\right)_{\E}\right|
    \lesssim \left(\|\nabla \vp\|_{L^{2}}+\|\vpp\|_{L^{2}}\right)\|\partial_{x_{1}}\Psi_{n}\|_{H^{1}}\lesssim 
    \|\vec{\varphi}\|_{\E}.
    \end{aligned}
    \end{equation*}
    Gathering the above estimates, we have 
    \begin{equation*}
    \begin{aligned}
    &\dot{a}_{n}\left(\vec{\Psi}_{n},\vec{\Psi}_{n}\right)_{\E}+\sum_{k=1}^{K}\dot{b}_{n,k}\left(\vec{\Phi}_{n,k},\vec{\Psi}_{n}\right)_{\E}+O\left(t^{-2}(|\dot{\ba}|+|\dot{\bb}|)\right)\\
    &=O\left(\|\vec{\varphi}\|_{\E}+t^{-4}+|\ba|^{2}+|\bb|^{2}\right).
    \end{aligned}
    \end{equation*} 
    Proceeding similarly for all $n=1,\dots,N$ and $(n,k)\in I^{0}$, and then using the invertibility of Gram matrix, we obtain~\eqref{est:ab}.
	\end{proof}
Fix 
\begin{equation}\label{def:sigma}
\sigma=\frac{1}{10}\min_{n\ne n'}|\ell_{n}-\ell_{n'}|>0.
\end{equation}
For $n=1,\dots,N$, let 
\begin{align}
\sigma_{n}=2\pi^{2}(1-\ell_{n}^{2})^{\frac{1}{2}},\quad 
c_{n}(t)=\frac{{\left(\vp,\Psi_{n}\xi_{n}\right)_{L^{2}}}}{\sigma_{n} \log t},\label{def:sigmaRn}
\end{align}
where
\begin{equation}\label{def:xin}
\xi_{n}(t,x)=\xi\left( (\sigma t)^{-1}\left|\bigg(\frac{x_{1}-\ell_{n} t}{\sqrt{1-\ell_{n}^{2}}},\bar{x}\bigg)\right|\right),\quad \mbox{for}\ n=1,\dots,N.
\end{equation}
For $n=1,\dots,N$, we also set
\begin{equation*}
\Omega_{n}(t,x)=\left\{x=(x_{1},\bar{x})\in \RR^{4}:\sigma t\le \left|\left(\frac{x_{1}-\ell_{n} t}{\sqrt{1-\ell^{2}_{n}}},\bar{x}\right)\right|\le 2\sigma t\right\}.
\end{equation*}
Note that, from the definition of $\xi_{n}$ and $\Omega_{n}$, we have the following estimates (${\textbf{1}}_{K}$ denotes the characteristic function of the compact set $K$),
\begin{equation}\label{est:ptpxxin}
\left|\pt \xi_{n}(t,x)\right|+|\partial_{x_{1}} \xi_{n}(t,x)|\lesssim t^{-1}\textbf{1}_{\Omega_{n}}.
\end{equation}
Note also that, from the decay property of $\psi$ in~\eqref{est:phipsi} and the definition of $\Psi_{n}$ and $\xi_{n}$, we have the following estimates,
\begin{align}
\|\Psi_{n}\xi_{n}\|_{L^{2}}&\lesssim \left(\int_{|x|<2\sigma t}\langle x \rangle^{-4}\d x\right)^{\frac{1}{2}}\lesssim \left(\int_{0}^{2\sigma t}\frac{r^{3}}{1+r^{4}}\d r\right)^{\frac{1}{2}}\lesssim \log ^{\frac{1}{2}}t,\label{est:Psi2}\\
\|\Psi_{n}\xi_{n}\|_{L^{\frac{4}{3}}}&\lesssim \left(\int_{|x|<2\sigma t}\langle x\rangle^{-\frac{8}{3}}\d x\right)^{\frac{3}{4}}\lesssim \left(\int_{0}^{2\sigma t}\frac{r^{3}}{1+r^{\frac{8}{3}}}\d r\right)^{\frac{3}{4}}\lesssim t,\label{est:Psi34}\\
\|\Psi_{n}{\textbf{1}}_{\Omega_{n}}\|_{L^{\frac{4}{3}}}&\lesssim \left(\int_{\sigma t<|x|<2\sigma t}\langle x\rangle^{-\frac{8}{3}}\d x\right)^{\frac{3}{4}}\lesssim \left(\int_{\sigma t}^{2\sigma t}\frac{r^{3}}{1+r^{\frac{8}{3}}}\d r\right)^{\frac{3}{4}}\lesssim t.\label{est:Psi134}
\end{align}

We introduce the following technical lemma.
\begin{lemma}\label{le:Psixi}
	For $n=1,\dots,N$ and $n'\ne n$, we have 
	\begin{align}
	\left(\Psi_{n'},\Psi_{n}\xi_{n}\right)_{L^{2}}&=O(1),\label{est:Psin'n}\\
	\left(\Psi_{n},\Psi_{n}\xi_{n}\right)_{L^{2}}&=\sigma_{n}\log t +O(1) \label{est:Psinn}.
	\end{align}
	\end{lemma}
\begin{proof}
	Proof of~\eqref{est:Psin'n}. For any $n'\ne n$, we consider the change of variable,
	\begin{equation*}
	\left(\frac{x_{1}-\ell_{n}t}{\sqrt{1-\ell_{n}^{2}}},\bar{x}\right)\mapsto (y_{1},\bar{y}).
	\end{equation*}
	It follows that 
	\begin{equation*}
	\left(\Psi_{n'},\Psi_{n}\xi_{n}\right)_{L^{2}}=(1-\ell_{n}^{2})^{\frac{1}{2}}\int_{\RR^{4}}\psi(x)\xi\left(\frac{|x|}{\sigma t}\right)\psi\left(\frac{(1-\ell_{n}^{2})^{\frac{1}{2}}x_{1}-(\ell_{n'}-\ell_{n})t}{\sqrt{1-\ell_{n'}^{2}}},\bar{x}\right)\d x.
	\end{equation*}
	From the definitions of $\xi$, $\sigma$ and $\xi_{n}$ respectively in~\eqref{def:xi},~\eqref{def:sigma} and~\eqref{def:xin}, for any $x=(x_{1},\bar{x})\in {\rm{Supp}} \xi ((\sigma t)^{-1}|\cdot|)$, we have 
	\begin{equation*}
	\left|(1-\ell_{n}^{2})^{\frac{1}{2}}x_{1}-(\ell_{n'}-\ell_{n})t\right|\ge 
	\left(\left|\ell_{n'}-\ell_{n}\right|-2\sigma (1-\ell_{n}^{2})^{\frac{1}{2}}\right)t\ge \frac{1}{2}|\ell_{n'}-\ell_{n}|t.
	\end{equation*}
	Based on the above inequality and the decay property of $\psi$ in~\eqref{est:phipsi}, for any $x=(x_{1},\bar{x})\in {\rm{Supp}} \xi ((\sigma t)^{-1}|\cdot|)$,
	\begin{equation*}
	\left|\psi\left(\frac{(1-\ell_{n}^{2})^{\frac{1}{2}}x_{1}-(\ell_{n'}-\ell_{n})t}{\sqrt{1-\ell_{n'}^{2}}},\bar{x}\right)\right|\lesssim |(1-\ell_{n}^{2})^{\frac{1}{2}}x_{1}-(\ell_{n'}-\ell_{n})t|^{-2}\lesssim t^{-2}.
	\end{equation*}
	Therefore, using the decay property of $\psi$ again, we have 
	\begin{equation*}
	\left|\left(\Psi_{n'},\Psi_{n}\xi_{n}\right)_{L^{2}}\right|\lesssim t^{-2}\int_{|x|\le 2\sigma t}|\psi(x)|\d x\lesssim t^{-2}\int_{|x|\le 2\sigma t}\langle x \rangle^{-2}\d x\lesssim 1,
	\end{equation*}
	which means~\eqref{est:Psin'n}.
	
	Proof of~\eqref{est:Psinn}. For $n=1,\dots,N$, we consider the change of variable again,
		\begin{equation*}
	\left(\frac{x_{1}-\ell_{n}t}{\sqrt{1-\ell_{n}^{2}}},\bar{x}\right)\mapsto (y_{1},\bar{y}).
	\end{equation*}
	It follows that 
	\begin{equation*}
	\left(\Psi_{n},\Psi_{n} \xi_{n}\right)_{L^{2}}=(1-\ell_{n}^{2})^{\frac{1}{2}}
	\int_{\RR^{4}}\psi^{2}(x)\xi\left(\frac{|x|}{\sigma t}\right)\d x.
	\end{equation*}
	Using the definition of $\xi$ in~\eqref{def:xi} again, we decompose
	\begin{equation*}
	\int_{\RR^{4}}\psi^{2}(x)\xi\left(\frac{|x|}{\sigma t}\right)\d x=H_{3}+H_{4}+H_{5}+H_{6},
	\end{equation*}
	where
	\begin{equation*}
	\begin{aligned}
	H_{3}&=\int_{|x|<1}\psi^{2}(x)\d x,\quad \quad \ 
	H_{4}=\int_{\sigma t<|x|<2\sigma t}\psi^{2}(x)\xi\left(\frac{|x|}{\sigma t}\right)\d x,\\
	H_{5}&=\int_{1<|x|<\sigma t}|x|^{-4}\d x,\quad 
	H_{6}=\int_{1<|x|<\sigma t}\left(\psi^{2}(x)-|x|^{-4}\right)\d x.
	\end{aligned}
	\end{equation*}
	By the asymptotic and  decay properties of $\psi$ in~\eqref{est:phi} and~\eqref{est:phipsi}, we have 
	\begin{equation*}
	\begin{aligned}
	\left|H_{3}\right|+|H_{4}|+|H_{6}|
	&\lesssim \int_{|x|<1}\langle x \rangle ^{-4}\d x 
	+\int_{\sigma t<|x|<2\sigma t}\langle x\rangle ^{-4}\d x
	+\int_{1<|x|<\sigma t}\langle x \rangle^{-5}\d x\\
	&\lesssim \int_{0}^{1}\frac{r^{3}}{1+r^{4}}\d r+\int_{\sigma t}^{2\sigma t}\frac{r^{3}}{1+r^{4}}\d r
	+\int_{1}^{\sigma t}\frac{r^{3}}{1+r^{5}}\d r\lesssim 1.
	\end{aligned}
	\end{equation*}
	Next, using the spherical coordinate transformation in dimension 4,
	\begin{equation*}
	H_{5}=2\pi^{2}\int_{1}^{\sigma t}\frac{1}{r}\d r=2\pi^{2}\log (\sigma t)=2\pi^{2}\log t+O(1).
	\end{equation*}
	Combining above estimates and using the definition of $\sigma_{n}$ in~\eqref{def:sigmaRn}, we obtain~\eqref{est:Psinn}.
\end{proof}
Now, we derive a refined equation of $\dot{\ba}$ from Lemma~\ref{le:equab} and Lemma~\ref{le:Psixi}.
\begin{lemma}[Refined equation of $\dot{\ba}$] For $n=1,\dots,N$, we have 
	\begin{equation}\label{est:Rdota}
	\left|\dot{a}_{n}+\dot{c}_{n}\right|\lesssim \|\vec{\varphi}\|_{\E}\log^{-\frac{1}{2}}t+|\ba|^{2}+|\bb|^{2}+t^{-4}.
	\end{equation}
\end{lemma}
\begin{proof}
	First, by~\eqref{equ:vp} and direct computation,
	\begin{equation*}
	\begin{aligned}
	\frac{\d}{\d t}\left(\vp,\Psi_{n}\xi_{n}\right)_{L^{2}}
	=&\left(\pt\vp,\Psi_{n}\xi_{n}\right)_{L^{2}}+\left(\vp,\Psi_{n}\pt \xi_{n}\right)_{L^{2}}-\ell_{n}\left(\vp,\xi_{n}\partial_{x_{1}}\Psi_{n}\right)_{L^{2}}\\
	=&\left(\vpp,\Psi_{n}\xi_{n}\right)_{L^{2}}-
	\left({\rm{Mod}}_{1,1},\Psi_{n}\xi_{n}\right)_{L^{2}}
	-\left({\rm{Mod}}_{1,2},\Psi_{n}\xi_{n}\right)_{L^{2}}\\
	&+\left(\vp,\Psi_{n}\pt \xi_{n}\right)_{L^{2}}-\ell_{n}\left(\vp,\xi_{n}\partial_{x_{1}}\Psi_{n}\right)_{L^{2}}.
	\end{aligned}
	\end{equation*}
	From the Cauchy-Schwarz inequality and~\eqref{est:Psi2},
	\begin{equation*}
	\left|\left(\vpp,\Psi_{n}\xi_{n}\right)_{L^{2}}\right|
	\lesssim \|\vpp\|_{L^{2}}\|\Psi_{n}\xi_{n}\|_{L^{2}}
	\lesssim \|\vec{\varphi}\|_{\E}\log^{\frac{1}{2}}t .
	\end{equation*}
	Next, by~\eqref{est:ab},~\eqref{est:Psin'n} and~\eqref{est:Psinn},
	\begin{equation*}
	\begin{aligned}
	\left({\rm{Mod}}_{1,1},\Psi_{n}\xi_{n}\right)_{L^{2}}
	&=\dot{a}_{n}\left(\Psi_{n},\Psi_{n}\xi_{n}\right)_{L^{2}}+\sum_{n'\ne n}\dot{a}_{n'}\left(\Psi_{n'},\Psi_{n}\xi_{n}\right)_{L^{2}}\\
	&=\left(\sigma_{n} \log t\right) \dot{a}_{n}+O\left(\|\vec{\varphi}\|_{\E}+|\ba|^{2}+|\bb|^{2}+t^{-4}\right).
	\end{aligned}
	\end{equation*}
    Using~\eqref{est:ab}, the Young's inequality and the decay properties of $\psi$ and $\phi_{k}$ in~\eqref{est:phipsi},
	\begin{equation*}
	\begin{aligned}
	\left|\left({\rm{Mod}}_{1,2},\Psi_{n}\xi_{n}\right)_{L^{2}}\right|
	&\lesssim |\dot{\bb}|\sum_{(n',k)\in I^{0}}\int_{\RR^{4}}\left(|\Phi_{n',k}|^{\frac{5}{3}}+|\Psi_{n}|^{\frac{5}{2}}\right)\d x\\
	&\lesssim \left(\|\vec{\varphi}\|_{\E}+|\ba|^{2}+|\bb|^{2}+t^{-4}\right)\int_{\RR^{4}}\langle x\rangle^{-5}\d x\\
	&\lesssim \left(\|\vec{\varphi}\|_{\E}+|\ba|^{2}+|\bb|^{2}+t^{-4}\right).
	\end{aligned}
		\end{equation*}
		Note that, from~\eqref{est:ptpxxin},~\eqref{est:Psi134} and the Sobolev inequality~\eqref{est:Sobo}, 
		\begin{equation*}
		\left|\left(\vp,\Psi_{n}\pt \xi_{n}\right)_{L^{2}}\right|\lesssim t^{-1}\|\vp\|_{L^{4}}\|\Psi_{n}{\textbf{1}}_{\Omega_{n}}\|_{L^{\frac{4}{3}}}\lesssim \|\vec{\varphi}\|_{\E}.
		\end{equation*}
		Then, by integration by parts, 
		\begin{equation*}
		\left(\vp,\xi_{n}\partial_{x_{1}}\Psi_{n}\right)_{L^{2}}
		=-\left(\partial_{x_{1}}\vp,\Psi_{n}\xi_{n}\right)_{L^{2}}
		-\left(\vp,\Psi_{n}\partial_{x_{1}}\xi_{n}\right)_{L^{2}}.
		\end{equation*}
		Thus, using~\eqref{est:ptpxxin},~\eqref{est:Psi2},~\eqref{est:Psi134}, the Sobolev inequality~\eqref{est:Sobo} and the Cauchy-Schwarz inequality again,
		\begin{equation*}
		\begin{aligned}
		\left|\left(\vp,\xi_{n}\partial_{x_{1}}\Psi_{n}\right)_{L^{2}}\right|
		&\lesssim 
		\left|\left(\partial_{x_{1}}\vp,\Psi_{n}\xi_{n}\right)_{L^{2}}\right|
		+\left|\left(\vp,\Psi_{n}\partial_{x_{1}}\xi_{n}\right)_{L^{2}}\right|\\
		&\lesssim \|\nabla \vp\|_{L^{2}}\|\Psi_{n}\xi_{n}\|_{L^{2}}+t^{-1}\|\vp\|_{L^{4}}\|\Psi_{n}{\textbf{1}}_{\Omega_{n}}\|_{L^{\frac{4}{3}}}\lesssim \|\vec{\varphi}\|_{\E}\log^{\frac{1}{2}}t.
		\end{aligned}
		\end{equation*}
		Gathering above estimates, we have
		\begin{equation}\label{est:dota1}
		\left|(\sigma_{n} \log t)\dot{a}_{n}+\frac{\d}{\d t}\left(\vp,\Psi_{n}\xi_{n}\right)_{L^{2}}\right|
		\lesssim \|\vec{\varphi}\|_{\E}\log ^{\frac{1}{2}}t+|\ba|^{2}+|\bb|^{2}+t^{-4}.
		\end{equation}
		Last, using~\eqref{est:Psi34},
		\begin{align}
		\dot{c}_{n}
		&=\left(\sigma_{n}\log t\right)^{-1}\frac{\d}{\d t}\left(\vp,\Psi_{n}\xi_{n}\right)_{L^{2}}
		-\left(\sigma_{n}t\log^{2}t\right)^{-1} \left(\vp,\Psi_{n}\xi_{n}\right)_{L^{2}}\nonumber\\
		&=\left(\sigma_{n}\log t\right)^{-1}\frac{\d}{\d t}\left(\vp,\Psi_{n}\xi_{n}\right)_{L^{2}}
		+O\left(\|\vec{\varphi}\|_{\E}\log^{-2} t\right).\label{est:dotc}
		\end{align}
		Dividing both sides of the inequality~\eqref{est:dota1} by $\sigma_{n} \log t$ and then using~\eqref{est:dotc}, we obtain~\eqref{est:Rdota}.
\end{proof}

Last, we consider the equation of the exponential directions.

\begin{lemma}[Exponential directions]
	In the context of Proposition~\ref{pro:dec}, for $(n,j)\in I$, let $z_{n,j}^{\pm}=\left(\vec{\varphi},\vec{Z}_{n,j}^{\pm}\right)_{L^{2}}$ and $\alpha_{n,j}=\lambda_{j}(1-\ell_{n}^{2})^{\frac{1}{2}}>0$. Then, we have 
	\begin{equation}\label{est:dz}
	\left|\frac{\d}{\d t}z_{n,j}^{\pm}\pm \alpha_{n,j} z_{n,j}^{\pm}\right|\lesssim 
	\|\vec{\varphi}\|_{\E}^{2}+|\ba|^{2}+|\bb|^{2}+t^{-4}.
	\end{equation}
\end{lemma}
\begin{proof}
	First, by~\eqref{equ:vvp}, we compute
	\begin{align*}
	\frac{\d}{\d t}z_{n,j}^{\pm}=&\left(\vec{\mathcal{L}}\vec{\varphi},\vec{Z}_{n,j}^{\pm}\right)_{L^{2}}
	-\ell_{n}\left(\vec{\varphi},\partial_{x_{1}}\vec{Z}_{n,j}^{\pm}\right)_{L^{2}}\\
	&+\left(\vec{G},\vec{Z}_{n,j}^{\pm}\right)_{L^{2}}+\left(\vec{R}_{1},\vec{Z}_{n,j}^{\pm}\right)_{L^{2}}+\left(\vec{R}_{2},\vec{Z}_{n,j}^{\pm}\right)_{L^{2}}-\left(\vec{{\rm{Mod}}},\vec{Z}_{n,j}^{\pm}\right)_{L^{2}}.
	\end{align*}
	From~\eqref{equ:zy} and integration by parts,
	\begin{equation*}
	\begin{aligned}
	&\left(\vec{\mathcal{L}}\vec{\varphi},\vec{Z}_{n,j}^{\pm}\right)_{L^{2}}
	-\ell_{n}\left(\vec{\varphi},\partial_{x_{1}}\vec{Z}_{n,j}^{\pm}\right)_{L^{2}}\\
	&=-\left(\vec{\varphi},\left(\mathcal{H}_{\ell_{n}}{\rm{J}}\vec{Z}^{\pm}_{j,\ell_{n}}\right)(\cdot-\bell_{n} t)\right)_{L^{2}}+3\sum_{n'\ne n}\left(\vp,Q_{n'}^{2}\Upsilon^{\pm}_{n,j}\right)_{L^{2}}\\
	&=\mp \alpha_{n,j}z_{n,j}^{\pm}+3\sum_{n'\ne n}\left(\vp,Q_{n'}^{2}\Upsilon^{\pm}_{n,j}\right)_{L^{2}},
	\end{aligned}
	\end{equation*}
	where
	\begin{equation*}
	\Upsilon_{n,j}^{\pm}=\pm \lambda_{j}(1-\ell_{n}^{2})^{\frac{1}{2}}\Upsilon_{j,\ell_{n}}^{\pm,1}(\cdot-\bell_{n} t).
	\end{equation*}
	Note that, by Lemma~\ref{le:int} (i) and the decay property of $Q$ and ${\Upsilon}^{\pm,1}_{j,\ell}$ in~\eqref{est:aspQ} and~\eqref{est:decY},
	\begin{equation*}
	\sum_{n'\ne n}\left|\left({\varphi}_{1},Q_{n'}^{2}\Upsilon^{\pm}_{n,j}\right)_{L^{2}}\right|
	\lesssim \sum_{n'\ne n}\|\vp\|_{L^{4}}\|Q_{n'}^{2}\Upsilon^{\pm}_{n,j}\|_{L^{\frac{4}{3}}}\lesssim 
	\|\vec{\varphi}\|_{\E}^{2}+t^{-12}.
	\end{equation*}
	Next, from~\eqref{est:decY},~\eqref{est:G} and the expression of $R_{1}$,
	\begin{equation*}
	\begin{aligned}
	\left|\left(\vec{G},\vec{Z}_{n,j}^{\pm}\right)_{L^{2}}\right|
	&\lesssim \|G\|_{L^{2}}\|\Upsilon^{\pm}_{n,j}\|_{L^{2}}\lesssim |\ba|^{2}+|\bb|^{2}+t^{-4},\\
	\left|\left(\vec{R}_{1},\vec{Z}_{n,j}^{\pm}\right)_{L^{2}}\right|
	&\lesssim 
	\|\vp\|_{L^{4}}^{2}\left(\|\vp\|_{L^{4}}+\|R+U+V\|_{L^{4}}\right)\|\Upsilon^{\pm}_{n,j}\|_{L^{4}}\lesssim \|\vec{\varphi}\|_{\E}^{2}.
	\end{aligned}
	\end{equation*}
	Using the decay property of $Q$ in~\eqref{est:aspQ} and Lemma~\ref{le:int} again,
	\begin{equation*}
	\begin{aligned}
	\left|\left(\vec{R}_{2},\vec{Z}_{n,j}^{\pm}\right)_{L^{2}}\right|
	&\lesssim \|\vp\|_{L^{4}}\|\Upsilon^{\pm}_{n,j}\|_{L^{2}}\left(\|U+V\|_{L^{8}}^{2}+\|R(U+V)\|_{L^{4}}\right)\\
	&\quad +\sum_{n'\ne n''}\|\vp\|_{L^{4}}\|\Upsilon^{\pm}_{n,j}\|_{L^{2}}\|Q_{n'}Q_{n''}\|_{L^{4}}\\
	&\lesssim \|\vec{\varphi}\|_{\E}\left(|\ba|+|\bb|+t^{-3}\right)\lesssim \|\vec{\varphi}\|_{\E}^{2}+|\ba|^{2}+|\bb|^{2}+t^{-6}.
	\end{aligned}
	\end{equation*}
	Last, using~\eqref{equ:ker},~\eqref{equ:zy}, change of variable and integration by parts,
	\begin{equation*}
	\begin{aligned}
	&\left(\vec{\Psi}_{n},\vec{Z}_{n,j}^{\pm}\right)_{L^{2}}
	=\left(\vec{\psi}_{\ell_{n}},\mathcal{H}_{\ell_{n}}\vec{\Upsilon}_{j,\ell_{n}}^{\pm}\right)_{L^{2}}
	=\left(\mathcal{H}_{\ell_{n}}\vec{\psi}_{\ell_{n}},\vec{\Upsilon}_{j,\ell_{n}}^{\pm}\right)_{L^{2}}=0,\\
	&\left(\vec{\Phi}_{n,k},\vec{Z}_{n,j}^{\pm}\right)_{L^{2}}
	=\left(\vec{\phi}_{k,\ell_{n}},\mathcal{H}_{\ell_{n}}\vec{\Upsilon}_{j,\ell_{n}}^{\pm}\right)_{L^{2}}
	=\left(\mathcal{H}_{\ell_{n}}\vec{\phi}_{k,\ell_{n}},\vec{\Upsilon}_{j,\ell_{n}}^{\pm}\right)_{L^{2}}=0.
	\end{aligned}
	\end{equation*}
	Therefore, by~\eqref{est:phipsi},~\eqref{est:decY},~\eqref{est:ab} and the Lemma~\ref{le:int} (i),
	\begin{equation*}
	\begin{aligned}
	\left(\vec{{\rm{Mod}}},\vec{Z}_{n,j}^{\pm}\right)_{L^{2}}
	&=\sum_{n'\ne n}\dot{a}_{n'}\left(\vec{\Psi}_{n'},\vec{Z}_{n,j}^{\pm}\right)_{L^{2}}+
	\sum_{n'\ne n}\sum_{k=1}^{K}\dot{b}_{n',k}\left(\vec{\Phi}_{n',k},\vec{Z}_{n,j}^{\pm}\right)_{L^{2}}\\
	&=O\left((\|\vec{\varphi}\|_{\E}+|\ba|^{2}+|\bb|^{2}+t^{-4})t^{-2}\right)\\
	&=O\left(\|\vec{\varphi}\|_{\E}^{2}+|\ba|^{4}+|\bb|^{4}+t^{-4}\right).
	\end{aligned}
	\end{equation*}
	Gathering above estimates, we obtain~\eqref{est:dz}.
\end{proof}

\section{Proof of Theorem~\ref{the:main}}\label{S:Thm}
In this section, we prove the existence of a solution $\vec{u}=(u,\partial_{t}u)$ of~\eqref{equ:wave} satisfying~\eqref{equ:limu} in Theorem~\ref{the:main}. We argue by compactness and obtain the solution $\vec{u}$ as the weak limit of a sequence of approximate excited multi-solitons $\vec{u}_{m}$.

First, we recall a technical lemma which constructs well-prepared initial data at $t=T\gg 1$ with sufficient freedom related to unstable directions.

\begin{lemma}\label{le:ini}
	There exists $T_{0}\gg 1$ such that for any $T\ge T_{0}$ the following is true. 
	For any $\boldsymbol{z}=(z_{n,j})_{(n,j)\in I}\in \mathcal{B}_{\RR^{|I|}}(T^{-\frac{7}{2}})$, there exist
	\begin{equation*}
	Z=(\tilde{z}_{n,j})_{(n,j)\in I},\quad A=(\tilde{a}_{n})_{n=1,\dots,N} \quad 
	\mbox{and}\quad B=(\tilde{b}_{n,k})_{(n,k)\in I^{0}},
	\end{equation*}
	satisfying
	\begin{equation}\label{est:ZAB}
	\sum_{(n,j)\in I}\left|\tilde{z}_{n,j}\right|+\sum_{n=1}^{N}\left|\tilde{a}_{n}\right|
	+\sum_{(n,k)\in I^{0}}\left|\tilde{b}_{n,k}\right|\lesssim T^{-\frac{7}{2}},
	\end{equation}
	such that the function $\vec{\varphi}(T)=(\vp(T),\vpp(T))\in \dot{H}^{1}\times L^{2}$ defined by
	\begin{equation}\label{def:vpT}
	\vec{\varphi}(T)=\sum_{(n,j)\in I}\tilde{z}_{n,j}\vec{Z}_{n,j}^{+}(T)
	+\sum_{n=1}^{N}\tilde{a}_{n}\vec{\Psi}_{n}(T)+\sum_{(n,k)\in I^{0}}\tilde{b}_{n,k}\vec{\Phi}_{n,k}(T)
	\end{equation} 
	satisfies for all $m=1,\dots,K$, $(n,k)\in I^{0}$ and $(n,j)\in I$,
	\begin{equation*}
	(\vec{\varphi}(T),\vec{\Psi}_{n}(T))_{\E}=(\vec{\varphi}(T),\vec{\Phi}_{n,k}(T))_{\E}=0 \quad \mbox{and}\quad  
	\left(\vec{\varphi}(T),\vec{Z}_{n,j}^{+}(T)\right)_{L^{2}}={z}_{n,j}.
	\end{equation*}
	Moreover, the initial data defined by $\vec{u}(T)=\sum_{n=1}^{N}\vec{Q}_{n}(T)+\vec{\varphi}(T)$ is modulated in the sense of Proposition~\ref{pro:dec} with $a_{n}(T)=b_{n,k}(T)=0$ and $z_{n,j}^{+}(T)=z_{n,j}$ for any $n=1,\dots,N$, $(n,k)\in I^{0}$ and $(n,j)\in I$.
\end{lemma}

\begin{proof}
	The proof of the technical lemma relies on a standard argument based on the Lemma~\ref{le:ini} and the invertibility of Gram matrix (see~\emph{e.g.}~\cite[Lemma 4.1]{Y5De}), and we omit it.
	\end{proof}

Let $S_{m}=m^{2}$ for any $m\in \mathbb{N}$. For $\boldsymbol{z}_{m}=(z_{n,j}^{m})_{(n,j)}\in \mathcal{B}_{\RR^{|I|}}(S^{-\frac{7}{2}}_{m})$, we consider the backward solution $\vec{u}_{m}\in C\left([T_{\rm{min}}(m),S_{m}];\dot{H}^{1}\times L^{2}\right)$ of~\eqref{equ:wave} with the initial data $\vec{u}_{m}(S_{m})$ given by the Lemma~\ref{le:ini}. Since $\vec{u}_{m}(S_{m})\in \dot{H}^{2}\times \dot{H}^{1}$, by persistence of regularity (see for example~\cite[Remark 2.9]{KMACTA}), we also have $\vec{u}_{m}\in C\left([T_{\rm{min}}(m),S_{m}];\dot{H}^{2}\times \dot{H}^{1}\right)$. Such regularity will allow
energy computations in \S\ref{SS:ener} without a density argument.

The following Proposition is the main part of the proof of Theorem~\ref{the:main}.
\begin{proposition}[Uniform estimates]\label{pro:uni}
	Under the assumptions of Theorem~\ref{the:main}, there exist $m_{0}\in \mathbb{N}$, $T_{0}>0$ and $C>0$ such that the following is true. For any $m\ge m_{0}$, there exists $\boldsymbol{z}_{m}\in \mathcal{B}_{\RR^{|I|}}(S_{m}^{-\frac{7}{2}})$ such that the solution $\vec{u}_{m}$ of~\eqref{equ:wave} with the initial data $\vec{u}_{m}(S_{m})$ given by the Lemma~\ref{le:ini} is well-defined $\dot{H}^{1}\times L^{2}$ on the time interval $[T_{0},S_{m}]$. Moreover, the solution $\vec{u}_{m}$ satisfies
	\begin{equation}\label{est:uni}
	\|\vec{u}_{m}(t)-\sum_{n=1}^{N}\vec{Q}_{n}(t)\|_{\E}\le Ct^{-2},\quad \mbox{for all}\ t\in [T_{0},S_{m}].
	\end{equation}
\end{proposition}

The rest of the section is organized as follows. First, in \S\ref{SS:Boot}-\S\ref{SS:Pro}, we devote to the proof of Proposition~\ref{pro:uni} based on the energy method. Last, in \S\ref{SS:thm}, we prove the Theorem~\ref{the:main} from Proposition~\ref{pro:uni} by compactness argument.
\subsection{Bootstrap setting}\label{SS:Boot}
For $0<t\le S_{m}$, as long as $\vec{u}_{m}(t)$ is well-defined in $\dot{H}^{1}\times L^{2}$ and satisfies~\eqref{est:T1T2}, we decompose $\vec{u}_{m}(t)$ as in Proposition~\ref{pro:dec}. In particular, we denote by $\vec{\varphi}(t)$ and $(a_{n}(t),b_{n,k}(t))_{(n,k)\in I^{0}}$ the remainder term and the parameters of the decomposition of $\vec{u}_{m}(t)$.

The proof of Proposition~\ref{pro:uni} is based on the following bootstrap estimates: for $C_{0}\gg 1$ to be chosen,
\begin{equation}\label{est:boot}
\left\{\begin{aligned}
&|\ba|\le C_{0}^{2}t^{-2}\log^{-\frac{1}{2}}t,\\
&|\bb|\le C_{0}^{2}t^{-2},\quad  \|\vec{\varphi}\|_{\E}\le C_{0}t^{-3},\\
&\sum_{(n,j)\in I}|{z}^{-}_{n,j}|^{2}\le t^{-6},\quad 
\sum_{(n,j)\in I}|{z}^{+}_{n,j}|^{2}\le t^{-7}.
\end{aligned}
\right.
\end{equation}
For any $\boldsymbol{z}_{m}\in \mathcal{B}_{\RR^{|I|}}(S_{m}^{-\frac{7}{2}})$, set 
\begin{equation*}
T^{*}=T_{m}^{*}(\boldsymbol{z}_{m})=\inf \left\{t\in [T_{0},S_{m}]; \vec{u}_{m}\ \mbox{satisfies}~\eqref{est:T1T2} \ \mbox{and}\ ~\eqref{est:boot}\ \mbox{on}\ [t,S_{m}]\right\}.
\end{equation*}
Note that, for the proof of Proposition~\ref{pro:uni}, we just need to prove that there exist $T_{0}\gg 1$ (independent with $m$) large enough and at least one choice of $\boldsymbol{z}_{m}=(z^{m}_{n,j})_{(n,j)\in I}\in \mathcal{B}_{\RR^{|I|}}(S^{-\frac{7}{2}}_{m})$ such that $T^{*}=T_{0}$.

In the rest of the section, the implied constants in $\lesssim $ and $O$ do not depend on the constant $C_{0}$ appearing in the bootstrap assumption~\eqref{est:boot}.

\subsection{Energy functional}\label{SS:ener}
Recall that,
\begin{equation*}
-\bar{\ell}<\ell_{1}<\ell_{2}<\dots<\ell_{N}<\bar{\ell}\quad \mbox{where}\quad  \bar{\ell}=\max \left(|\ell_{1}|,|\ell_{N}|\right)<1.
\end{equation*}
Fix
\begin{equation}\label{def:delta}
0<\delta=\frac{1}{40}(1-\bar{\ell})\times \min (\ell_{2}-\ell_{1},\dots,\ell_{N}-\ell_{N-1})<\frac{1}{20}.
\end{equation}
We denote
\begin{equation*}
\begin{aligned}
&{\bar{\ell}}_{n}=\ell_{n}+\delta(\ell_{n+1}-\ell_{n}),\quad \mbox{for $n=1,\dots,N-1$},\\
&\underline{\ell}_{n}=\ell_{n}-\delta(\ell_{n+1}-\ell_{n}),\quad \mbox{for $n=2,\dots,N$}.
\end{aligned}
\end{equation*}
For $t>0$, denote
\begin{equation}\label{def:Omega}
\Omega(t)=\left(\big(\bar{\ell}_{1}t,\underline{\ell}_{2}t\big)\cup\ldots\cup\big(\bar{\ell}_{N-1}t,\underline{\ell}_{N}t\big)\right)\times \mathbb{R}^{3},
\quad \Omega^{C}(t)=\mathbb{R}^{4}\setminus \Omega(t).
\end{equation}
We define the continuous function $\chi_{N}(t,x)=\chi_{N}(t,x_{1})$ as follow (see~\cite[section 4.3]{MM} for a similar choice of cut-off function),
for $t>0$,
\begin{equation}\label{defchiN}
\left\{\begin{aligned}
&\hbox{$\chi_N(t,x) =  \ell_1$ for $x_1\in (-\infty,  \bar{\ell}_1 t]$},\\
& \hbox{$\chi_N(t,x) = \ell_n $ for $x_1\in [\underline{\ell}_n t, \bar{\ell}_n t]$, for $n\in \{2,\ldots,N-1\}$,}
\\ & \hbox{$\chi_N(t,x)= \ell_N$ for $x_1\in [\underline{\ell}_{N} t,+\infty)$},\\
& \chi_N(t,x) = \frac{x_1}{(1-2\delta)t}  - \frac {\delta}{1-2 \delta} (\ell_{n+1}+\ell_n) 
\hbox{ for $x_1 \in [\bar{\ell}_{n} t,\underline{\ell}_{n+1} t ]$, $n\in \{1,\ldots,N-1\}$}.
\end{aligned}
\right.
\end{equation}
Note that 
\begin{equation}\label{derchi}\left\{\begin{aligned}
\overline{\nabla}\chi_{N}(t,x)&=0,\quad  \mbox{for}\ x\in \Omega(t),\\
 \partial_{x_1} \chi_N(t,x)&= \frac{1}{(1-2\delta)t},  \quad \hbox{for $x\in \Omega(t)$},\\
 \partial_{t} \chi_N(t,x)&= -\frac 1t \frac{x_1}{(1-2\delta)t}, \quad \hbox{for $x\in \Omega(t)$},\\
\partial_t \chi_N(t,x)&=0,\quad  \nabla \chi_N(t,x)=0, \quad \hbox{for $x\in \Omega^C(t)$}.\\
\end{aligned} \right.\end{equation}

We define
\begin{equation*}
\mathcal{K}(t)=\mathcal{E}(t)+\mathcal{P}(t)+\mathcal{G}(t)+\mathcal{J}(t),
\end{equation*}
where
\begin{equation*}
\begin{aligned}
\mathcal{G}(t)&=-2\int_{\RR^{4}}\vp(t,x) G_{2}(t,x)\d x,\\
\mathcal{P}(t)&=2\int_{\RR^{4}}\big(\chi_{N}(t,x)\partial_{x_{1}}\vp(t,x))\vpp(t,x)\d x,\\
\mathcal{E}(t)&=\int_{\RR^{4}}\left\{|\nabla\vp|^{2}+\vpp^{2}-\frac{1}{2}\left(R+U+V+\vp\right)^{4}\right.\\
&\quad \quad \quad  \quad\left. +\frac{1}{2}\left(R+U+V\right)^{4}+2(R+U+V)^{3}\vp\right\}\d x,
\end{aligned}
\end{equation*}
and
\begin{equation*}
\mathcal{J}=\sum_{n=1}^{N}\mathcal{J}_{n},\quad \mbox{with}\
\mathcal{J}_{n}=2a_{n}\int_{\RR^{4}}\left(\ell_{n}\partial_{x_{1}}\vp-\vpp\right)(\ell_{n}-\chi_{N})\partial_{x_{1}}\Psi_{n}\d x.
\end{equation*}

\begin{remark}\label{re:EPGJ}
	The quantity $\mathcal{E}$ and $\mathcal{P}$ are standard energy and momentum functionals for the remainder term $\vec{\varphi}$. We introduce $\mathcal{G}$ here to remove the effect of $G_{2}$ on the time variation of the quantity $\mathcal{E}+\mathcal{P}$ (see Lemma~\ref{le:timeEPGJ} (iii)). Moreover, we introduce $\mathcal{J}$ to remove the effect of $\left(\Psi_{n}\right)_{n=1,\dots,N}$ which is the main difficulty in the energy estimate for the 4D case (see Lemma~\ref{le:timeEPGJ} (iv)).
\end{remark}

Denote
\begin{equation*}
\mathcal{N}_{\Omega}=\int_{\Omega}\left(|\nabla \vp|^{2}+\vpp^{2}+2\chi_{N}(\partial_{x_{1}}\vp)\vpp\right)\d x,\quad 
\mathcal{N}_{\Omega^{C}}=\int_{\Omega^{C}}\left(|\nabla \vp|^{2}+\vpp^{2}\right)\d x.
\end{equation*}
For $n=1,\dots,N$, we set 
\begin{equation*}
\zeta_{n}(t,x)=\zeta(x_{1}-\ell_{n}t,\bar{x})=\left(1+\left((x_{1}-\ell_{n} t)^{2}+|\bar{x}|^{2}\right)^{\frac{1}{2}}\right)^{-\gamma},
\end{equation*}
where $\gamma$ is introduced in Lemma~\ref{le:locacoer}.

Note that, from $|\chi_{N}|\le \bar{\ell}<1$,
\begin{equation}\label{est:Nomega}
\mathcal{N}_{\Omega}\ge (1-\bar{\ell})\int_{\Omega}\left(|\nabla \vp|^{2}+\vpp^{2}\right)\d x.
\end{equation}
Note also that, from the definition of $\chi_{N}$ and $\zeta_{n}$, 
\begin{equation}\label{est:zetan}
\sum_{n=1}^{N}\|(\chi_{N}-\ell_{n})\zeta^{2}_{n}\|_{L^{\infty}}+\sum_{n=1}^{N}\|\zeta_{n}^{2}\|_{L^{\infty}(\Omega)}+\left(\sum_{n=1}^{N}\zeta_{n}^{2}-1\right)\lesssim t^{-2\gamma}.
\end{equation}
First, we introduce a technical lemma for future reference.
\begin{lemma}
	Let $f$ be a continuous function such that
	\begin{equation}\label{est:W}
	|f(x)|\lesssim \langle x\rangle^{-(3+\alpha)}\quad \mbox{on}\ \RR^{4},
	\end{equation}
	where $\alpha\ge 0$. For all $n=1,\dots,N$, we have
	\begin{equation}\label{est:L1W2}
	\|f_{n}\|^{2}_{L^{2}(\Omega)}+\|(\ell_{n}-\chi_{N})f_{n}\|^{2}_{L^{2}}+\int_{\RR^{4}}|\ell_{n}-\chi_{N}|f_{n}^{2}\d x\lesssim t^{-(2+2\alpha)},
	\end{equation}
	where $f_{n}(t,x)=f(x_{1}-\ell_{n} t,\bar{x})$.
\end{lemma}

\begin{proof} We prove the estimate~\eqref{est:L1W2} for~$\|f_{n}\|^{2}_{L^{2}(\Omega)}$; other estimates are proved same. First, from the definition of $\delta$ and $\Omega$ respectively in~\eqref{def:delta} and~\eqref{def:Omega}, we have 
	\begin{equation*}
	|x_{1}-\ell_{n}t|\ge \delta^{2}t,\quad \mbox{for any}\ x=(x_{1},\bar{x})\in \Omega.
	\end{equation*}
	Therefore, from the decay assumption of $f$ in~\eqref{est:W} and change of variables,
	\begin{equation*}
	\begin{aligned}
	\|f_{n}\|^{2}_{L^{2}(\Omega)}
	&\lesssim \int_{\RR^{3}}\int_{|x_{1}-\ell_{n} t|\ge \delta^{2}t}f^{2}(x_{1}-\ell_{n} t,\bar{x})\d x_{1}\d \bar{x}\\
	&\lesssim \int_{|x|\ge \delta^{2}t} \langle x \rangle^{-(6+2\alpha)}\d x\lesssim \int_{\delta^{2}t}^{\infty}r^{-(3+2\alpha)}\d r\lesssim t^{-(2+2\alpha)}.
	\end{aligned}
	\end{equation*}
	The proof of~\eqref{est:L1W2} is complete.
\end{proof}

Second, we introduce the coercivity property of $\mathcal{K}$.
\begin{lemma} There exists $\nu>0$ such that the following estimates hold.
	\begin{enumerate}
		\item \emph{Bound on $\mathcal{G}$ and $\mathcal{J}$}. We have 
		\begin{equation}\label{est:G+J}
		\left|\mathcal{G}\right|+\left|\mathcal{J}\right|\lesssim C_{0}^{5}t^{-7}+C_{0}^{3}t^{-6}\log^{-\frac{1}{2}}t.
		\end{equation}
		
		\item \emph{Coercivity of $\mathcal{E}+\mathcal{P}$}. We have 
		\begin{equation}\label{est:coerEP}
		\mathcal{E}+\mathcal{P}\ge \mathcal{N}_{\Omega}+\nu \mathcal{N}_{\Omega^{C}}-\nu^{-1}\left(t^{-6}+C_{0}^{4}t^{-(6+2\gamma)}\right).
		\end{equation}
		
		\item \emph{Coercivity of $\mathcal{K}$.} We have 
		\begin{equation}\label{est:coerK}
		\mathcal{K}\ge \mathcal{N}_{\Omega}+\nu \mathcal{N}_{\Omega^{C}}-\nu^{-1}\left(t^{-6}+C_{0}^{5}t^{-6}\log^{-\frac{1}{2}}t\right).
		\end{equation}
	\end{enumerate}
\end{lemma}
\begin{proof}
	Proof of (i). Note that, from~\eqref{est:boot},~\eqref{est:L1W2}, the definition of $G_{2}$ in Lemma~\ref{le:G} and the Sobolev inequality~\eqref{est:Sobo}, we have
	\begin{equation*}
	\begin{aligned}
	\left|\mathcal{G}\right|&\lesssim \|\vp\|_{L^{4}}\|G_{2}\|_{L^{\frac{4}{3}}}\lesssim \left(|\ba|^{2}+|\bb|^{2}\right)\|\vec{\varphi}\|_{\E}\lesssim C_{0}^{5}t^{-7},\\
	\left|\mathcal{J}\right|&\lesssim \sum_{n=1}^{N}|\ba|\left(\|\nabla \vp\|_{L^{2}}+\|\vpp\|_{L^{2}}\right)
	\|(\ell_{n}-\chi_{N})\partial_{x_{1}}\Psi_{n}\|_{L^{2}}\lesssim C^{3}_{0}t^{-6}\log^{-\frac{1}{2}}t,
	\end{aligned}
	\end{equation*}
	which implies~\eqref{est:G+J}.
	
	Proof of (ii). We decompose 
	\begin{equation*}
	\mathcal{E}+\mathcal{P}=\mathcal{F}_{1}+\mathcal{F}_{2}+\mathcal{F}_{3},
	\end{equation*}
	where
	\begin{equation*}
	\begin{aligned}
	\mathcal{F}_{1}&=-\frac{1}{2}\int_{\RR^{4}}\vp^{3}\left(\vp+4(R+U+V)\right)\d x,\\
	\mathcal{F}_{2}&=-3\int_{\RR^{4}}\vp^{2}\bigg(\sum_{n\ne n'}Q_{n}Q_{n'}+2R(U+V)+(U+V)^{2}\bigg)\d x,\\
	\mathcal{F}_{3}&=\int_{\RR^{4}}\left(|\nabla \vp|^{2}-3\sum_{n=1}^{N}Q_{n}^{2}\vp^{2}+\vpp^{2}+2\chi_{N}(\partial_{x_{1}}\vp)\vpp\right)\d x.
	\end{aligned}
	\end{equation*}
	
	\emph{Estimate on $\mathcal{F}_{1}$}. From~\eqref{est:boot} and the Sobolev inequality~\eqref{est:Sobo},
	\begin{equation*}
	\left|\mathcal{F}_{1}\right|\lesssim \|\vp\|_{L^{4}}^{3}\left(\|\vp\|_{L^{4}}+\|R+U+V\|_{L^{4}}\right)\lesssim C_{0}^{3}t^{-9}.
	\end{equation*}
	
	\emph{Estimate on $\mathcal{F}_{2}$}. Note that, using~\eqref{est:boot}, the Sobolev inequality~\eqref{est:Sobo}, the Lemma~\ref{le:int} and the decay property of $Q$ in~\eqref{est:aspQ} again,
	\begin{equation*}
	\begin{aligned}
	\left|\mathcal{F}_{2}\right|
	&\lesssim \|\vp\|_{L^{4}}^{2}\bigg(\sum_{n\ne n'}\|Q_{n}Q_{n'}\|_{L^{2}}+\|R\|_{L^{4}}\|U+V\|_{L^{4}}+\|U+V\|_{L^{4}}^{2}\bigg)\\
	&\lesssim \|\vec{\varphi}\|_{\E}^{2}\left(t^{-3}+|\ba|+|\bb|\right)
	\lesssim C_{0}^{2}t^{-6}\left(t^{-3}+C_{0}^{2}t^{-2}\log^{-\frac{1}{2}}t+C_{0}^{2}t^{-2}\right)\lesssim C_{0}^{4}t^{-8}.
	\end{aligned}
	\end{equation*}
	\emph{Estimate on $\mathcal{F}_{3}$}. We decompose
	\begin{equation*}
	\mathcal{F}_{3}=\mathcal{N}_{\Omega}+\mathcal{F}_{3,1}+\mathcal{F}_{3,2}+\mathcal{F}_{3,3}+\mathcal{F}_{3,4},
	\end{equation*}
	where
	\begin{equation*}
	\begin{aligned}
	\mathcal{F}_{3,1}&=2\sum_{n=1}^{N}\int_{\RR^{4}}(\chi_{N}-\ell_{n})(\partial_{x_{1}}\vp)\vpp \zeta^{2}_{n}\d x,\\
	\mathcal{F}_{3,2}&=-\int_{\Omega}\left(|\nabla \vp|^{2}+\vpp^{2}+2\chi_{N}(\partial_{x_{1}}\vp)\vpp\right)\bigg(\sum_{n=1}^{N}\zeta_{n}^{2}\bigg)\d x,
	\end{aligned}
	\end{equation*}
	and
	\begin{equation*}
	\begin{aligned}
	\mathcal{F}_{3,3}&=\int_{\Omega^{C}}\left(|\nabla \vp|^{2}+\vpp^{2}+2\chi_{N}(\partial_{x_{1}}\vp)\vpp\right)\bigg(1-\sum_{n=1}^{N}\zeta_{n}^{2}\bigg)\d x,\\
	\mathcal{F}_{3,4}&=\sum_{n=1}^{N}\int_{\RR^{4}}\left(|\nabla\vp|^{2}\zeta_{n}^{2}-3Q_{n}^{2}\vp^{2}+\vpp^{2}\zeta_{n}^{2}+2\ell_{n}(\partial_{x_{1}}\vp)\vpp\zeta_{n}^{2}\right)\d x.
	\end{aligned}
	\end{equation*}
	Note that, from~\eqref{est:boot} and~\eqref{est:zetan},
	\begin{equation*}
	\left|\mathcal{F}_{3,1}\right|+\left|\mathcal{F}_{3,2}\right|\lesssim 
	\sum_{n=1}^{N}\left(\|(\chi_{N}-\ell_{n})\zeta_{n}^{2}\|_{L^{\infty}}+\|\zeta_{n}^{2}\|_{L^{\infty}(\Omega)}\right)\|\vec{\varphi}\|^{2}_{\E}\lesssim C_{0}^{2}t^{-(6+2\gamma)}.
	\end{equation*}
	Then, using~\eqref{est:boot} and~\eqref{est:zetan} again,
	\begin{equation*}
	\mathcal{F}_{3,3}\ge \int_{\Omega^{C}}\left(|\nabla \vp|^{2}+\vpp^{2}+2\chi_{N}(\partial_{x_{1}}\vp)\vpp\right)\bigg|1-\sum_{n=1}^{N}\zeta_{n}^{2}\bigg|\d x
	+O\left(C_{0}^{2}t^{-(6+2\gamma)}\right).
	\end{equation*}
	Last, from~\eqref{est:localcoer},~\eqref{equ:orth1},~\eqref{est:boot} for $\mathcal{F}_{3,4}$ and $|\chi_{N}|<1$, we have 
	\begin{equation*}
	\begin{aligned}
	\mathcal{F}_{3,4}&\ge \mu \int_{\RR^{4}}\left(|\nabla \vp|^{2}+\vpp^{2}\right)\bigg(\sum_{n=1}^{N}\zeta_{n}^{2}\bigg)\d x-\mu^{-1}t^{-6}\\
	&\ge \frac{\mu}{2} \int_{\Omega^{C}}
	\left(|\nabla \vp|^{2}+\vpp^{2}+2\chi_{N}(\partial_{x_{1}}\vp)\vpp\right)\bigg(\sum_{n=1}^{N}\zeta_{n}^{2}\bigg)\d x-\mu^{-1}t^{-6}.
	\end{aligned}
	\end{equation*}
	Combining above estimates and taking $\nu>0$ small enough, we obtain~\eqref{est:coerEP}.
	
	Proof of (iii). Estimate~\eqref{est:coerK} is a consequence of~\eqref{est:G+J} and~\eqref{est:coerEP}.
	\end{proof}

Third, we derive the time variation of $\mathcal{K}$.
\begin{lemma}\label{le:timeEPGJ}
	There exists $T_{0}$ large enough such that the following estimates hold.
	\begin{enumerate}
		\item \emph{Time variation of $\mathcal{E}$.} We have 
		\begin{equation}\label{est:dE}
		\begin{aligned}
		\frac{\d}{\d t}\mathcal{E}
		=&-2\sum_{n=1}^{N}\dot{a}_{n}\int_{\RR^{4}}(\ell_{n}\partial_{x_{1}}\vp-\vpp)\ell_{n}\partial_{x_{1}}\Psi_{n}\d x\\
		&+2\int_{\RR^{4}}\vp\left(\Delta {\rm{Mod}}_{1,2}+3R^{2}{\rm{Mod}}_{1,2}\right)\d x
		-2\int_{\RR^{4}}\vpp {\rm{Mod}}_{2,2}\d x\\
		&+2\int_{\RR^{4}}\vpp G_{2}\d x+6\sum_{n=1}^{N}\int_{\RR^{4}}\left(\ell_{n}Q_{n}\partial_{x_{1}}Q_{n}\right)\vp^{2}\d x
		+O(C_{0}t^{-7}).
		\end{aligned}
		\end{equation}
		
		\item \emph{Time variation of $\mathcal{P}$}. We have 
		\begin{equation}\label{est:dP}
		\begin{aligned}
		\frac{\d}{\d t}\mathcal{P}
		=&-\frac{1}{(1-2\delta)t}\int_{\Omega}
		\left((\partial_{x_{1}}\vp)^{2}+\vpp^{2}
		+2\frac{x_{1}}{t}(\partial_{x_{1}}\vp)\vpp-|\overline{\nabla}\vp|^{2}\right)\d x\\
		&-6\sum_{n=1}^{N}\int_{\RR^{4}}\ell_{n}(Q_{n}\partial_{x_{1}} Q_{n})\vp^{2}\d x+2\int_{\RR^{4}}\chi_{N}(\partial_{x_{1}}\vp)G_{2}\d x\\
		&-2\int_{\RR^{4}}\chi_{N}{\rm{Mod}}_{2,2}\partial_{x_{1}}\vp\d x-2\int_{\RR^{4}}\left(\chi_{N} \partial_{x_{1}}{\rm{Mod}}_{1,2}\right)\vpp\d x
		\\
		&+2\sum_{n=1}^{N}\dot{a}_{n}\int_{\RR^{4}}\left(\ell_{n}\partial_{x_{1}}\vp-\vpp\right)\chi_{N}\partial_{x_{1}}\Psi_{n} \d x+O(C_{0}t^{-7}).
		\end{aligned}
		\end{equation}
		
		\item \emph{Time variation of $\mathcal{G}$}. We have
		\begin{equation}\label{est:dG}
		\frac{\d}{\d t}\mathcal{G}=-2\int_{\RR^{4}}\vpp G_{2}\d x-2\int_{\RR^{4}}\chi_{N}(\partial_{x_{1}}\vp)G_{2}\d x+O(t^{-7}).
		\end{equation}
		
		\item \emph{Time variation of $\mathcal{J}_{n}$}. For $n=1,\dots,N$, we have 
		\begin{equation}\label{est:dtJ}
		\frac{\d }{\d t}\mathcal{J}_{n}=2\dot{a}_{n}\int_{\RR^{4}}
		\left(\ell_{n}\partial_{x_{1}}\vp-\vpp\right)\left(\ell_{n}-\chi_{N}\right)\partial_{x_{1}}\Psi_{n} \d x
		+O(t^{-7}).
		\end{equation}
	\end{enumerate}
	
\end{lemma}
\begin{proof}
	Proof of (i). We decompose
	\begin{equation*}
	\frac{\d}{\d t}\mathcal{E}=\mathcal{I}_{1}+\mathcal{I}_{2}+\mathcal{I}_{3},
	\end{equation*}
	where
	\begin{equation*}
	\begin{aligned}
	\mathcal{I}_{1}&=2\int_{\RR^{4}}\left(\nabla \vp\cdot \nabla \pt \vp+\vpp\pt \vpp\right)\d x,\\
	\mathcal{I}_{2}&=-2\int_{\RR^{4}}\pt \vp\big(\left(R+U+V+\vp\right)^{3}-\left(R+U+V\right)^{3}\big)\d x,\\
	\mathcal{I}_{3}&=-2\int_{\RR^{4}}\left(\pt R+\pt U+\pt V\right)\left(3(R+U+V)\vp^{2}+\vp^{3}\right)\d x.
	\end{aligned}
	\end{equation*}
	
	\smallskip
	\emph{Estimate on $\mathcal{I}_{1}$}. We claim
	\begin{equation}\label{est:I1}
	\begin{aligned}
	\mathcal{I}_{1}
	=&-2\sum_{n=1}^{N}\dot{a}_{n}\int_{\RR^{4}}\left(\ell_{n}\partial_{x_{1}}\vp-\vpp\right)\ell_{n}\partial_{x_{1}}\Psi_{n}\d x+2\int_{\RR^{4}}\vpp G_{2}\d x\\
	&-6\sum_{n=1}^{N}\dot{a}_{n}\int_{\RR^{4}}Q_{n}^{2}\Psi_{n}\vp\d x+2\int_{\RR^{4}}\vp \Delta {\rm{Mod}}_{1,2}\d x-2\int_{\RR^{4}}\vpp {\rm{Mod}}_{2,2}\d x\\
	&+2\int_{\RR^{4}}\vpp\left((R+U+V+\vp)^{3}-(R+U+V)^{3}\right)\d x+O(C_{0}t^{-7}).
	\end{aligned}
	\end{equation}
	Indeed, from~\eqref{equ:vp} and integration by parts, 
	\begin{equation*}
	\begin{aligned}
		\mathcal{I}_{1}
	=&2\int_{\RR^{4}}\vp \Delta {\rm{Mod}}_{1,1}\d x-2\int_{\RR^{4}}\vpp {\rm{Mod}}_{2,1}\d x\\
	&+2\int_{\RR^{4}}\vp \Delta {\rm{Mod}}_{1,2}\d x-2\int_{\RR^{4}}\vpp {\rm{Mod}}_{2,2}\d x+2\int_{\RR^{4}}\vpp G_{2}\d x\\
	&+2\int_{\RR^{4}}\vpp\left((R+U+V+\vp)^{3}-(R+U+V)^{3}\right)\d x+2\sum_{i=1,3}\int_{\RR^{4}}\vpp G_{i}\d x.
	\end{aligned}
	\end{equation*}
	By integration by parts, $-(1-\ell^{2}_{n})\partial_{x_{1}}^{2}\Psi_{n}-\bar{\Delta}\Psi_{n}-3Q_{n}^{2}\Psi_{n}=0$ and the definition of $\rm{Mod}_{1,1}$ and $\rm{Mod}_{2,1}$, 
	\begin{equation*}
	\begin{aligned}
	&2\int_{\RR^{4}}\vp \Delta {\rm{Mod}}_{1,1}\d x-2\int_{\RR^{4}}\vpp {\rm{Mod}}_{2,1}\d x\\
	&=-2\sum_{n=1}^{N}\dot{a}_{n}\int_{\RR^{4}}\left(\ell_{n}\partial_{x_{1}}\vp-\vpp\right)\ell_{n}\partial_{x_{1}}\Psi_{n}\d x
	-6\sum_{n=1}^{N}\dot{a}_{n}\int_{\RR^{4}}Q_{n}^{2}\Psi_{n}\vp\d x.
	\end{aligned}
	\end{equation*}
	Then, from~\eqref{est:G1},~\eqref{est:G3},~\eqref{est:boot} and the Cauchy-Schwarz inequality, for $T_{0}$ large enough (depending on $C_{0}$), we have
	\begin{equation*}
	\begin{aligned}
	\sum_{i=1,3}\left|\int_{\RR^{4}}\vpp G_{i}\d x\right|
	&\lesssim \|\vpp\|_{L^{2}}\left(\|G_{1}\|_{L^{2}}+\|G_{3}\|_{L^{2}}\right)\\
	&\lesssim \|\vec{\varphi}\|_{\E}\left(|\ba|^{2}+|\bb|^{3}+t^{-4}\right)\\
	&\lesssim C_{0}t^{-3}\left(C_{0}^{4}t^{-4}\log^{-1}t+C_{0}^{6}t^{-6}+t^{-4}\right)\lesssim C_{0}t^{-7}.
	\end{aligned}
	\end{equation*}
	 We see that~\eqref{est:I1} follows from above identities and estimate.
	 
	\vspace*{1mm}
	\emph{Estimate on $\mathcal{I}_{2}$.} We claim
	\begin{equation}\label{est:I2}
	\begin{aligned}
	\mathcal{I}_{2}=
	&-2\int_{\RR^{4}}\vpp\left((R+U+V+\vp)^{3}-(R+U+V)^{3}\right)\d x\\
	&+6\sum_{n=1}^{N}\dot{a}_{n}\int_{\RR^{4}}Q_{n}^{2}\Psi_{n}\vp\d x+6\int_{\RR^{4}}\left(R^{2}{\rm{Mod}}_{1,2}\right)\vp\d x+O(t^{-7}).
	\end{aligned}
	\end{equation}
	First, using~\eqref{equ:vp} again, we decompose,
\begin{equation*}
\begin{aligned}
\mathcal{I}_{2}=
&-2\int_{\RR^{4}}\vpp\left((R+U+V+\vp)^{3}-(R+U+V)^{3}\right)\d x\\
&+6\sum_{n=1}^{N}\dot{a}_{n}\int_{\RR^{4}}Q_{n}^{2}\Psi_{n}\vp\d x+6\int_{\RR^{4}}\left(R^{2}{\rm{Mod}}_{1,2}\right)\vp\d x+\mathcal{I}_{2,1}+\mathcal{I}_{2,2}+\mathcal{I}_{2,3},
\end{aligned}
\end{equation*}
where
\begin{equation*}
\begin{aligned}
\mathcal{I}_{2,1}&=6\sum_{n=1}^{N}\dot{a}_{n}\int_{\RR^{4}}\Psi_{n}(R^{2}-Q^{2}_{n})\vp \d x,\\
\mathcal{I}_{2,2}&=2\int_{\RR^{4}}{\rm{Mod}}_{1}\vp^{2}\left(\vp+3(R+U+V)\right)\d x,\\
\mathcal{I}_{2,3}&=6\int_{\RR^{4}}{\rm{Mod}}_{1}\vp\left(2(U+V)R+(U+V)^{2}\right)\d x.
\end{aligned}
\end{equation*}
By an elementary computation,
\begin{equation*}
\mathcal{I}_{2,1}
=6\sum_{n=1}^{N}\sum_{n'\ne n}\dot{a}_{n}\int_{\RR^{4}}
\left(2\Psi_{n}Q_{n}Q_{n'}+\Psi_{n}Q_{n'}^{2}\right)\vp \d x.
\end{equation*}
Therefore, from~\eqref{est:ab},~\eqref{est:boot}, the Lemma~\ref{le:int} (i), the Sobolev inequality~\eqref{est:Sobo} and the decay properties of $\psi$ and $Q$ respectively in~\eqref{est:aspQ} and~\eqref{est:phipsi}, for $T_{0}$ large enough, 
\begin{equation*}
\begin{aligned}
\left|\mathcal{I}_{2,1}\right|
&\lesssim \sum_{n\ne n'}|\dot{\ba}|\|\vp\|_{L^{4}}\left(\|\Psi_{n}Q_{n}Q_{n'}\|_{L^{\frac{4}{3}}}+\|\Psi_{n}Q_{n'}^{2}\|_{L^{\frac{4}{3}}}\right)\\
&\lesssim t^{-2}\|\vec{\varphi}\|_{\E}\left(\|\vec{\varphi}\|_{\E}+|\ba|^{2}+|\bb|^{2}+t^{-4}\right)\\
&\lesssim C_{0}t^{-5}\left(C_{0}t^{-3}+C_{0}^{4}t^{-4}\log^{-1}t+C_{0}^{4}t^{-4}+t^{-4}\right)\lesssim t^{-7}.
\end{aligned}
\end{equation*}
Then, using~\eqref{est:ab},~\eqref{est:boot} and the Sobolev inequality in~\eqref{est:Sobo} again,
\begin{equation*}
\begin{aligned}
\left|\mathcal{I}_{2,2}\right|
&\lesssim \|\vp\|^{2}_{L^{4}}\|{\rm{Mod}}_{1}\|_{L^{4}}\left(\|\vp\|_{L^{4}}+\|R+U+V\|_{L^{4}}\right)\\
&\lesssim \|\vec{\varphi}\|_{\E}^{2}
\left(\|\vec{\varphi}\|_{\E}+|\ba|^{2}+|\bb|^{2}+t^{-4}\right)\\
&\lesssim C_{0}^{2}t^{-6}\left(C_{0}t^{-3}+C_{0}^{4}t^{-4}\log^{-1}t+C_{0}^{4}t^{-4}+t^{-4}\right)\lesssim t^{-7},
\end{aligned}
\end{equation*}
\begin{equation*}
\begin{aligned}
\left|\mathcal{I}_{2,3}\right|
&\lesssim \|\vp\|_{L^{4}}\|{\rm{Mod}}_{1}\|_{L^{4}}\|U+V\|_{L^{4}}\left(\|R\|_{L^{4}}+\|U+V\|_{L^{4}}\right)\\
&\lesssim \|\vec{\varphi}\|_{\E}(|\ba|+|\bb|)\left(\|\vec{\varphi}\|_{\E}+|\ba|^{2}+|\bb|^{2}+t^{-4}\right)\\
&\lesssim C^{3}_{0}t^{-5}\left(C_{0}t^{-3}+C_{0}^{4}t^{-4}\log^{-1}t+C_{0}^{4}t^{-4}+t^{-4}\right)\lesssim t^{-7}.
\end{aligned}
\end{equation*}
Combining above estimates, we obtain~\eqref{est:I2}

\vspace*{1mm}
\emph{Estimate on $\mathcal{I}_{3}$.} We claim
\begin{equation}\label{est:I3}
\mathcal{I}_{3}=6\sum_{n=1}^{N}\int_{\RR^{4}}\left(\ell_{n}Q_{n}\partial_{x_{1}}Q_{n}\right)\vp^{2}\d x
+O(t^{-7}).
\end{equation}
From $\pt R=-\sum_{n=1}^{N}\ell_{n}\partial_{x_{1}} Q_{n}$ and an elementary computation, we decompose
\begin{equation*}
\mathcal{I}_{3}=6\sum_{n=1}^{N}\int_{\RR^{4}}\left(\ell_{n}Q_{n}\partial_{x_{1}}Q_{n}\right)\vp^{2}\d x+\mathcal{I}_{3,1}+\mathcal{I}_{3,2}+\mathcal{I}_{3,3},
\end{equation*}
where
\begin{equation*}
\begin{aligned}
\mathcal{I}_{3,1}&=6\sum_{n\ne n'}\int_{\RR^{4}}\left(\ell_{n}Q_{n'}\partial_{x_{1}}Q_{n}\right)\vp^{2}\d x,\\
\mathcal{I}_{3,2}&=-6\int_{\RR^{4}}R\vp^{2}\left(\pt U+\pt V\right)\d x,\\
\mathcal{I}_{3,3}&=-2\int_{\RR^{4}}\vp^{2}\left(\pt R+\pt U+\pt V\right)(3(U+V)+\vp)\d x.
\end{aligned}
\end{equation*}
By~\eqref{est:boot}, the Lemma~\ref{le:int} (i) and the Sobolev inequality~\eqref{est:Sobo},
\begin{equation*}
\left|\mathcal{I}_{3,1}\right|\lesssim \sum_{n\ne n'}\|Q_{n'}\partial_{x_{1}}Q_{n}\|_{L^{2}}\|\vp\|_{L^{4}}^{2}\lesssim C_{0}^{2}t^{-9}.
\end{equation*}
Recall that,
\begin{align*}
\pt V&=-\sum_{(n,k)\in I^{0}}b_{n,k}\ell_{n}\partial_{x_{1}} \Phi_{n,k}+{\rm{Mod}}_{1,2},\\
\pt U&=-\sum_{n=1}^{N}a_{n}\ell_{n}\partial_{x_{1}} \Psi_{n}+{\rm{Mod}}_{1,1},\quad 
\pt R=-\sum_{n=1}^{N}\ell_{n}\partial_{x_{1}} Q_{n}.
\end{align*}
Therefore, using~\eqref{est:ab},~\eqref{est:boot} and the Sobolev inequality~\eqref{est:Sobo} again,
\begin{equation*}
\begin{aligned}
\left|\mathcal{I}_{3,2}\right|
&\lesssim \|R\|_{L^{4}}\|\vp\|^{2}_{L^{4}}\|\pt U+\pt V\|_{L^{4}}\\
&\lesssim \|\vec{\varphi}\|_{\E}^{2}(\|\vec{\varphi}\|_{\E}+|\ba|+|\bb|+t^{-4})\\
&\lesssim C_{0}^{2}t^{-6}\left(C_{0}t^{-3}+C_{0}^{2}t^{-2}+t^{-4}\right)\lesssim C_{0}^{4}t^{-8},
\end{aligned}
\end{equation*}
\begin{equation*}
\begin{aligned}
\left|\mathcal{I}_{3,3}\right|
&\lesssim \|\vp\|_{L^{4}}^{2}\left(\|U+V\|_{L^{4}}+\|\vp\|_{L^{4}}\right)\left(\|\pt R+\pt U+\pt V\|_{L^{4}}\right)\\
&\lesssim \|\vec{\varphi}\|_{\E}^{2}(|\ba|+|\bb|+\|\vec{\varphi}\|_{\E})\lesssim C_{0}^{4}t^{-8}.
\end{aligned}
\end{equation*}
Combining the above estimates, we obtain~\eqref{est:I3} for $T_{0}$ large enough.

We see that~\eqref{est:dE} follows from~\eqref{est:I1},~\eqref{est:I2} and~\eqref{est:I3}.

Proof of (ii). We decompose
\begin{equation*}
\begin{aligned}
\frac{\d}{\d t}\mathcal{P}=&2\int_{\RR^{4}}\left(\pt \chi_{N}\right)(\partial_{x_{1}} \vp)\vpp \d x+2\int_{\RR^{4}}\chi_{N}\left(\partial_{x_{1}}\pt \vp\right)\vpp\d x\\
&+
2\int_{\RR^{4}}\chi_{N}\left(\partial_{x_{1}}\vp\right)\pt \vpp\d x=
\mathcal{I}_{4}+\mathcal{I}_{5}+\mathcal{I}_{6}.
\end{aligned}
\end{equation*}

\smallskip
\emph{Estimate on $\mathcal{I}_{4}$}. By~\eqref{derchi}, we have 
\begin{equation}\label{est:I4}
\mathcal{I}_{4}=-\frac{1}{(1-2\delta)t}\int_{\Omega}2\frac{x_{1}}{t}(\partial_{x_{1}}\vp)\vpp \d x.
\end{equation}

\smallskip
\emph{Estimate on $\mathcal{I}_{5}$.} By~\eqref{equ:vp},~\eqref{derchi} and integration by parts,
\begin{align}
\mathcal{I}_{5}
&=2\int_{\RR^{4}}\chi_{N}(\partial_{x_{1}}\vpp)\vpp\d x-2\int_{\RR^{4}}\chi_{N}\vpp\left(\partial_{x_{1}}{\rm{Mod}}_{1,1}+\partial_{x_{1}}{\rm{Mod}}_{1,2}\right)\d x\nonumber\\
&=-\frac{1}{(1-2\delta)t}\int_{\Omega}\vpp^{2}\d x+2\int_{\RR^{4}}\vpp\left(-\chi_{N}\partial_{x_{1}}{\rm{Mod}}_{1,2}\right)\d x\nonumber\\
&\quad -2\sum_{n=1}^{N}\dot{a}_{n}\int_{\RR^{4}}\vpp\left(\chi_{N}\partial_{x_{1}}\Psi_{n}\right)\d x\label{est:I5}.
\end{align}

\smallskip
\emph{Estimate on $\mathcal{I}_{6}$.} We claim
\begin{equation}\label{est:I6}
\begin{aligned}
\mathcal{I}_{6}=
&-\frac{1}{(1-2\delta)t}\int_{\Omega}\left((\partial_{x_{1}} \vp)^{2}-|\overline{\nabla}\vp|^{2}\right)\d x+2\int_{\RR^{4}}\chi_{N}(\partial_{x_{1}}\vp)G_{2}\d x\\
&+2\sum_{n=1}^{N}\dot{a}_{n}\int_{\RR^{4}}\left(\ell_{n}\partial_{x_{1}}\vp\right)\left(\chi_{N}\partial_{x_{1}}\Psi_{n}\right)\d x
-2\int_{\RR^{4}}\chi_{N}(\partial_{x_{1}}\vp){\rm{Mod}}_{2,2}\d x\\
&-6\sum_{n=1}^{N}\int_{\RR^{4}}\left(\ell_{n}Q_{n}\partial_{x_{1}} Q_{n}\right)\vp^{2}\d x
+O(C_{0}t^{-7}).
\end{aligned}
\end{equation}
First, using~\eqref{equ:vp} again, we decompose,
\begin{equation*}
\begin{aligned}
\mathcal{I}_{6}
=&2\int_{\RR^{4}}\chi_{N}(\partial_{x_{1}}\vp)\Delta\vp \d x+
6\int_{\RR^{4}}\chi_{N}R^{2}(\partial_{x_{1}}\vp)\vp \d x\\
&+2\int_{\RR^{4}}\chi_{N}\left(\partial_{x_{1}}\vp\right)G_{2}\d x
-2\int_{\RR^{4}}\chi_{N}(\partial_{x_{1}}\vp){\rm{Mod}}_{2,2}\d x\\
&+2\sum_{n=1}^{N}\dot{a}_{n}\int_{\RR^{4}}\left(\ell_{n}\partial_{x_{1}}\vp\right)\left(\chi_{N}\partial_{x_{1}}\Psi_{n}\right)\d x+\mathcal{I}_{6,1}+\mathcal{I}_{6,2}+\mathcal{I}_{6,3},
\end{aligned}
\end{equation*}
where
\begin{equation*}
\begin{aligned}
\mathcal{I}_{6,1}
&=2\int_{\RR^{4}}\chi_N (\partial_{x_{1}}\vp)\vp^{3}\d x,\quad \mathcal{I}_{6,2}
=2\sum_{i=1,3}\int_{\RR^{4}}\chi_{N}\left(\partial_{x_{1}}\vp\right)G_{i}\d x,\\
\mathcal{I}_{6,3}
&=6\int_{\RR^{4}}\chi_N \vp(\partial_{x_{1}}\vp)\left((U+V)^{2}+2R(U+V)+(R+U+V)\vp\right)\d x.
\end{aligned}
\end{equation*}
Note that, from~\eqref{derchi} and integration by parts, we have 
\begin{equation}\label{est:I6delta}
\begin{aligned}
2\int_{\RR^{4}}\chi_{N}(\partial_{x_{1}}\vp)\Delta\vp \d x=
-\frac{1}{(1-2\delta)t}\int_{\Omega}\left((\partial_{x_{1}}\vp)^{2}-|\overline{\nabla}\vp|^{2}\right)\d x.
\end{aligned}
\end{equation}
Using integration by parts again, 
\begin{equation*}
\begin{aligned}
6&\int_{\RR^{4}}\chi_{N}R^{2}(\partial_{x_{1}}\vp)\vp\d x\\
=&-6\int_{\RR^{4}}\chi_{N}(R\partial_{x_{1}}R)\vp^{2}\d x
-3\int_{\RR^{N}}(\partial_{x_{1}}\chi_{N})R^{2}\vp^{2}\d x\\
=&-6\sum_{n=1}^{N}\int_{\RR^{4}}(\ell_{n}Q_{n}\partial_{x_{1}}Q_{n})\vp^{2}\d x-\frac{3}{(1-2\delta)t}\int_{\Omega}R^{2}\vp^{2}\d x\\
&+6\sum_{n=1}^{N}\int_{\RR^{4}}(\ell_{n}-\chi_{N})(Q_{n}\partial_{x_{1}}Q_{n})\vp^{2}\d x-6\sum_{n\ne n'}\int_{\RR^{4}}\chi_{N}(Q_{n'}\partial_{x_{1}}Q_{n})\vp^{2}\d x.
\end{aligned}
\end{equation*}
From~\eqref{est:boot},~\eqref{est:L1W2}, the Lemma~\ref{le:int} (i), the decay property of $Q$ in~\eqref{est:aspQ} and the Sobolev inequality~\eqref{est:Sobo},
\begin{equation*}
\left|\frac{3}{(1-2\delta)t}\int_{\Omega}R^{2}\vp^{2}\d x\right|
\lesssim \sum_{n=1}^{N}t^{-1}\|\vp\|_{L^{4}}^{2}\|Q_{n}^{2}\|_{L^{2}(\Omega)}\lesssim C_{0}^{2}t^{-11},
\end{equation*}
\begin{equation*}
\left|\sum_{n\ne n'}\int_{\RR^{4}}\chi_{N}(Q_{n'}\partial_{x_{1}}Q_{n})\vp^{2}\d x\right|\lesssim
\sum_{n\ne n'}\|\vp\|_{L^{4}}^{2}\|Q_{n'}\partial_{x_{1}}Q_{n}\|_{L^{2}}\lesssim C_{0}^{2}t^{-9},
\end{equation*}
\begin{equation*}
\left|\sum_{n=1}^{N}\int_{\RR^{4}}(\ell_{n}-\chi_{N})(Q_{n}\partial_{x_{1}}Q_{n})\vp^{2}\d x\right|
\lesssim \sum_{n=1}^{N}\|\vp\|_{L^{4}}^{2}\|(\ell_{n}-\chi_{N})(Q_{n}\partial_{x_{1}}Q_{n})\|_{L^{2}}\lesssim C^{2}_{0}t^{-11}.
\end{equation*}
Based on the above estimates, for $T_{0}$ large enough, we have 
\begin{equation}\label{est:I6U2}
6\int_{\RR^{4}}\chi_{N}R^{2}(\partial_{x_{1}}\vp)\vp\d x
=-6\sum_{n=1}^{N}\int_{\RR^{4}}(\ell_{n}Q_{n}\partial_{x_{1}}Q_{n})\vp^{2}\d x+O(t^{-7}).
\end{equation}
Then, using~\eqref{est:boot},~\eqref{derchi}, integration by parts and the Sobolev inequality~\eqref{est:Sobo} again,
\begin{equation}\label{est:I61}
\mathcal{I}_{6,1}=-\frac{1}{(1-2\delta)t}\int_{\Omega}\vp^{4}\d x=O\left(t^{-1}\|\vec{\varphi}\|_{\E}^{4}\right)=O\left(C_{0}^{4}t^{-13}\right).
\end{equation}
From~\eqref{est:G1},~\eqref{est:G3} and~\eqref{est:boot}, for $T_{0}$ large enough
\begin{equation}\label{est:I62}
\begin{aligned}
\left|\mathcal{I}_{6,2}\right|
&\lesssim \sum_{i=1,3}\|\partial_{x_{1}}\vp\|_{L^{2}}\|G_{i}\|_{L^{2}}\\
&\lesssim \|\vec{\varphi}\|_{\E}\left(|\ba|^{2}+|\bb|^{3}+t^{-4}\right)\\
&\lesssim C_{0}t^{-3}\left(C_{0}^{4}t^{-4}\log^{-1} t+C_{0}^{6}t^{-6}+t^{-4}\right)\lesssim C_{0}t^{-7}.
\end{aligned}
\end{equation}
Last, using~\eqref{est:boot} and the Sobolev inequality~\eqref{est:Sobo} again,
\begin{equation}\label{est:I63}
\begin{aligned}
\left|\mathcal{I}_{6,3}\right|&\lesssim \|\vp\|_{L^{4}}\|\partial_{x_{1}}\vp\|_{L^{2}}\|U+V\|_{L^{8}}
\left(\|U+V\|_{L^{8}}+\|R\|_{L^{8}}\right)\\
&\quad +\|\vp\|^{2}_{L^{4}}\|\partial_{x_{1}}\vp\|_{L^{2}}\|R+U+V\|_{L^{\infty}}\\
&\lesssim \|\vec{\varphi}\|_{\E}^{2}\left(|\ba|+|\bb|+\|\vec{\varphi}\|_{\E}\right)\lesssim C_{0}^{4}t^{-8}.
\end{aligned}
\end{equation}
Gathering estimates~\eqref{est:I6delta},~\eqref{est:I6U2},~\eqref{est:I61},~\eqref{est:I62} and~\eqref{est:I63}, we obtain~\eqref{est:I6} for $T_{0}$ large enough.
 
We see that~\eqref{est:dP} follows from~\eqref{est:I4},~\eqref{est:I5} and~\eqref{est:I6}.

Proof of (iii). We decompose
\begin{equation}
\frac{\d}{\d t}\mathcal{G}=-2\int_{\RR^{4}}(\pt \vp) G_{2}\d x-2\int_{\RR^{4}}\vp \pt G_{2}\d x=\mathcal{I}_{7}+\mathcal{I}_{8}.
\end{equation}

\emph{Estimate on $\mathcal{I}_{7}$.} We claim
\begin{equation}\label{est:I7}
\mathcal{I}_{7}=-2\int_{\RR^{4}}\vpp G_{2}\d x+O(t^{-7}).
\end{equation}
Indeed, from~\eqref{equ:vp},
\begin{equation*}
\mathcal{I}_{7}=-2\int_{\RR^{4}}\vpp G_{2}\d x+2\int_{\RR^{4}}{\rm{Mod}}_{1}G_{2}\d x.
\end{equation*}
Note that, by~\eqref{equ:can} and change of variables, we have 
\begin{equation*}
2\int_{\RR^{4}}{\rm{Mod}}_{1}G_{2}\d x=2\sum_{n\ne n'}\dot{a}_{n'}\int_{\RR^{4}}\Psi_{n'}G_{2,n}\d x+2\sum_{n\ne n'}\sum_{k=1}^{K}\dot{b}_{n',k}\int_{\RR^{4}}\Phi_{n',k}G_{2,n}\d x.
\end{equation*}
Therefore, using~\eqref{est:ab},~\eqref{est:boot}, the Lemma~\ref{le:int} (i) and the decay property of $\psi$ and $\phi_{k}$ in~\eqref{est:phipsi}, for any $n\ne n'$ and $k=1,\dots, K$, we have 
\begin{equation*}
\begin{aligned}
\left|\dot{a}_{n'}\int_{\RR^{4}}\Psi_{n'}G_{2,n}\d x\right|
&\lesssim t^{-2}(|\ba|^{2}+|\bb|^{2})(\|\vec{\varphi}\|_{\E}+|\ba|^{2}+|\bb|^{2}+t^{-4})\\
&\lesssim C_{0}^{4}t^{-6}\left(C_{0}t^{-3}+C_{0}^{4}t^{-4}\log^{-1}t+C_{0}^{4}t^{-4}+t^{-4}\right)
\lesssim C_{0}^{8}t^{-9},
\end{aligned}
\end{equation*}
\begin{equation*}
\begin{aligned}
\left|\dot{b}_{n',k}\int_{\RR^{4}}\Phi_{n',k}G_{2,n}\d x\right|
&\lesssim t^{-3}(|\ba|^{2}+|\bb|^{2})(\|\vec{\varphi}\|_{\E}+|\ba|^{2}+|\bb|^{2}+t^{-4})\\
&\lesssim C_{0}^{4}t^{-7}\left(C_{0}t^{-3}+C_{0}^{4}t^{-4}\log^{-1}t+C_{0}^{4}t^{-4}+t^{-4}\right)
\lesssim C_{0}^{8}t^{-10}.
\end{aligned}
\end{equation*}
Combining above estimates and taking $T_{0}$ large enough, we obtain~\eqref{est:I7}.

\emph{Estimate on $\mathcal{I}_{8}$.} We claim
\begin{equation}\label{est:I8}
\mathcal{I}_{8}=-2\int_{\RR^{4}}\chi_{N}(\partial_{x_1}\vp)G_{2}\d x+O(t^{-7}).
\end{equation}
Note that, from the definition of $G_{2,n}$ in Lemma~\ref{le:G}, 
\begin{equation*}
\begin{aligned}
\pt G_{2,n}=&-\ell_{n}\partial_{x_{1}} G_{2,n}
+6\sum_{k=1}^{K}(\dot{a}_{n}b_{n,k}+a_{n}\dot{b}_{n,k})Q_{n}\Psi_{n}\Phi_{n,k}\\
&+6\dot{a}_{n}a_{n}Q_{n}\Psi_{n}^{2}
+6\sum_{k,k'=1}^{K}\dot{b}_{n,k}b_{n,k'}Q_{n}\Phi_{n,k}\Phi_{n,k'}.
\end{aligned}
\end{equation*}
Therefore, by integration by parts, we have 
\begin{equation*}
\mathcal{I}_{8}=-2\int_{\RR^{4}}\chi_{N}(\partial_{x_1}\vp)G_{2}\d x+\mathcal{I}_{8,1}+\mathcal{I}_{8,2}+\mathcal{I}_{8,3}+\mathcal{I}_{8,4},
\end{equation*}
where
\begin{equation*}
\begin{aligned}
\mathcal{I}_{8,1}&=2\sum_{n=1}^{N}\int_{\RR^{4}}(\chi_{N}-\ell_{n})(\partial_{x_{1}}\vp)G_{2,n}\d x,\qquad \qquad\qquad \\
\mathcal{I}_{8,2}&=-12\sum_{n=1}^{N}\dot{a}_{n}a_{n}\int_{\RR^{4}}\chi_{N} \vp (Q_{n}\Psi_{n}\Psi_{n})\d x,
\end{aligned}
\end{equation*}
\begin{equation*}
\begin{aligned}
\quad \quad \quad \mathcal{I}_{8,3}&=-12\sum_{n=1}^{N}\sum_{k,k'=1}^{K}\dot{b}_{n,k}b_{n,k'}\int_{\RR^{4}}\chi_{N}\vp (Q_{n}\Phi_{n,k}\Phi_{n,k'})\d x,\\
\mathcal{I}_{8,4}&=-12\sum_{(n,k)\in I^{0}}(\dot{a}_{n}b_{n,k}+a_{n}\dot{b}_{n,k})\int_{\RR^{4}}\chi_{N} \vp (Q_{n}\Psi_{n}\Phi_{n,k})\d x.\\
\end{aligned}
\end{equation*}
Using~\eqref{est:boot},~\eqref{est:L1W2}, the definition of $G_{2,n}$ in Lemma~\ref{le:G} and the decay properties of $Q$, $\psi$ and~$\phi_{k}$ in~\eqref{est:aspQ} and~\eqref{est:phipsi}, we have 
\begin{equation*}
\begin{aligned}
\left|\mathcal{I}_{8,1}\right|
&\lesssim \sum_{n=1}^{N}\left(|\ba|^{2}+|\bb|^{2}\right)\|\nabla \vp\|_{L^{2}}\|(\chi_{N}-\ell_{n})Q_{n}\Psi_{n}^{2}\|_{L^{2}}\\
&\quad +\sum_{(n,k)\in I^{0}}\left(|\ba|^{2}+|\bb|^{2}\right)\|\nabla \vp\|_{L^{2}}\|(\chi_{N}-\ell_{n})Q_{n}\Phi_{n,k}^{2}\|_{L^{2}}\\
&\lesssim C_{0}t^{-8}\left(C_{0}^{4}t^{-4}\log^{-1}t+C_{0}^{4}t^{-4}\right)\lesssim C_{0}^{5}t^{-12}.
\end{aligned}
\end{equation*}
From~\eqref{est:ab},~\eqref{est:boot} and the Sobolev inequality~\eqref{est:Sobo},
\begin{equation*}
\begin{aligned}
\left|\mathcal{I}_{8,2}\right|+\left|\mathcal{I}_{8,3}\right|+\left|\mathcal{I}_{8,4}\right|
&\lesssim
\left(|\dot{\ba}|+|\dot{\bb}|\right)\left(|{\ba}|+|{\bb}|\right)\|\vp\|_{L^{4}}\left(\|Q\psi^{2}\|_{L^{\frac{4}{3}}}+\sum_{k=1}^{K}\|Q\phi^{2}_{k}\|_{L^{\frac{4}{3}}}\right)\\
&\lesssim \|\vec{\varphi}\|_{\E}\left(|{\ba}|+|{\bb}|\right)(\|\vec{\varphi}\|_{\E}+|{\ba}|^{2}+|{\bb}|^{2}+t^{-4})\lesssim C_{0}^{7}t^{-8}.
\end{aligned}
\end{equation*}
Combining above estimates and taking $T_{0}$ large enough, we obtain~\eqref{est:I8}.

We see that~\eqref{est:dG} follows from~\eqref{est:I7} and~\eqref{est:I8}.

Proof of (iv). We decompose
\begin{equation*}
\frac{\d}{\d t}\mathcal{J}_{n}=2\dot{a}_{n}\int_{\RR^{4}}
\left(\ell_{n}\partial_{x_{1}}\vp-\vpp\right)\left(\ell_{n}-\chi_{N}\right)\partial_{x_{1}}\Psi_{n} \d x+\mathcal{I}_{9}+\mathcal{I}_{10}+\mathcal{I}_{11}+\mathcal{I}_{12},
\end{equation*}
where
\begin{equation*}
\begin{aligned}
\mathcal{I}_{9}&=2a_{n}\int_{\RR^{4}}\left(\pt\vpp\right) (\chi_{N}-\ell_{n})\partial_{x_{1}}\Psi_{n}\d x,\qquad \quad\\
\mathcal{I}_{10}&=2a_{n}\int_{\RR^{4}}\left(\ell_{n}\partial_{x_{1}}\pt\vp\right) (\ell_{n}-\chi_{N})\partial_{x_{1}}\Psi_{n}\d x,
\end{aligned}
\end{equation*}
\begin{equation*}
\begin{aligned}
\quad \mathcal{I}_{11}&=2a_{n}\int_{\RR^{4}}\left(\ell_{n}\vp-\vpp\right) (\ell_{n}-\chi_{N})\partial_{t}\partial_{x_{1}}\Psi_{n}\d x,\\
\mathcal{I}_{12}&=2a_{n}\int_{\RR^{{4}}}\left(\ell_{n}\vp-\vpp\right)\left(\pt (\ell_{n}-\chi_{N})\right)\partial_{x_{1}}\Psi_{n}\d x.
\end{aligned}
\end{equation*}
\emph{Estimate on $\mathcal{I}_{9}$.} We claim
\begin{equation}\label{est:I9}
\mathcal{I}_{9}=O(t^{-7}).
\end{equation}
Indeed, from~\eqref{equ:vp}, we have 
\begin{equation*}
\mathcal{I}_{9}=\mathcal{I}_{9,1}+\mathcal{I}_{9,2}+\mathcal{I}_{9,3}+\mathcal{I}_{9,4}+\mathcal{I}_{9,5},
\end{equation*}
where
\begin{equation*}
\begin{aligned}
\mathcal{I}_{9,1}&=2a_{n}\int_{\RR^{4}}G (\chi_{N}-\ell_{n})\partial_{x_{1}}\Psi_{n}\d x,\qquad \qquad \qquad \quad \\
\mathcal{I}_{9,2}&=2a_{n}\int_{\RR^{4}}{\rm{Mod}}_{2}
(\ell_{n}-\chi_{N})\partial_{x_{1}}\Psi_{n}\d x,
\end{aligned}
\end{equation*}
\begin{equation*}
\begin{aligned}
\quad \quad \quad \mathcal{I}_{9,3}&=2a_{n}\int_{\RR^{4}}\left(\Delta \vp\right) (\chi_{N}-\ell_{n})\partial_{x_{1}}\Psi_{n}\d x,\\
\mathcal{I}_{9,4}&=6a_{n}\int_{\RR^{4}}\left((R+U+V)^{2}\vp\right) (\chi_{N}-\ell_{n})\partial_{x_{1}}\Psi_{n}\d x,\\
\mathcal{I}_{9,5}&=2a_{n}\int_{\RR^{4}}\left(3(R+U+V)\vp^{2}+\vp^{3}\right) (\chi_{N}-\ell_{n})\partial_{x_{1}}\Psi_{n}\d x.
\end{aligned}
\end{equation*}
First, by~\eqref{est:G},~\eqref{est:boot},~\eqref{est:L1W2}, the decay property of $\psi$ in~\eqref{est:phipsi} and the Cauchy-Schwarz inequality,
\begin{equation*}
\begin{aligned}
\left|\mathcal{I}_{9,1}\right|
&\lesssim |\ba|\|G\|_{L^{2}}\|(\chi_{N}-\ell_{n})\partial_{x_{1}}\Psi_{n}\|_{L^{2}}\\
&\lesssim C_{0}^{2}t^{-3}\log^{-\frac{1}{2}}t\left(C_{0}^{4}t^{-4}\log^{-1}t+C_{0}^{4}t^{-4}+t^{-4}\right)
\lesssim C_{0}^{6}t^{-7}\log^{-\frac{1}{2}}t.
\end{aligned}
\end{equation*}
Second, from the expression of $\rm{Mod}_{2}$, we decompose
\begin{equation*}
\mathcal{I}_{9,2}=\mathcal{I}_{9,2,1}+\mathcal{I}_{9,2,2}+\mathcal{I}_{9,2,3}+\mathcal{I}_{9,2,4},
\end{equation*}
where
\begin{equation*}
\begin{aligned}
	\mathcal{I}_{9,2,1}=&2a_{n}\dot{a}_{n}\int_{\RR^{4}}
	\ell_{n}(\chi_{N}-\ell_{n})\left(\partial_{x_{1}}\Psi_{n}\right)^{2}\d x,\qquad \qquad \qquad \\
	\mathcal{I}_{9,2,2}=&2\sum_{n\ne n'}a_{n}\dot{a}_{n'}\int_{\RR^{4}}\ell_{n'}(\chi_{N}-\ell_{n})
	(\partial_{x_{1}}\Psi_{n})(\partial_{x_{1}}\Psi_{n'})\d x,\\
	\end{aligned}
	\end{equation*}
	\begin{equation*}
	\begin{aligned}
	\quad \quad \mathcal{I}_{9,2,3}=&2\sum_{k=1}^{K}a_{n}\dot{b}_{n,k}\int_{\RR^{4}}\ell_{n}(\chi_{N}-\ell_{n})
	\left(\partial_{x_{1}}\Psi_{n}\right)\left(\partial_{x_{1}}\Phi_{n,k}\right)\d x,\\
	\mathcal{I}_{9,2,4}=&2\sum_{n\ne n'}\sum_{k=1}^{K}a_{n}\dot{b}_{n',k}\int_{\RR^{4}}\ell_{n'}(\chi_{N}-\ell_{n})(\partial_{x_{1}}\Psi_{n})
	(\partial_{x_{1}}\Phi_{n',k})\d x.
	\end{aligned}
\end{equation*} 
Using~\eqref{est:ab},~\eqref{est:boot},~\eqref{est:L1W2},~the Lemma~\ref{le:int} (i), (iii) and the decay properties of $\phi_{k}$ and $\psi$ in~\eqref{est:phipsi}, we have 
\begin{equation*}
\begin{aligned}
\left|\mathcal{I}_{9,2,1}\right|
&\lesssim 
|\ba|\left(|\dot{\ba}|+|\dot{\bb}|\right)\int_{\RR^{4}}
\left|\chi_{N}-\ell_{n}\right| \left(\partial_{x_{1}}\Psi_{n}\right)^{2}\d x\\
&\lesssim C_{0}^{2}t^{-4}\log ^{-\frac{1}{2}}t \left(C_{0}t^{-3}+C_{0}^{4}t^{-4}\log^{-1} t+C_{0}^{4}t^{-4}+t^{-4}\right)\lesssim C_{0}^{6}t^{-7}\log ^{-\frac{1}{2}}t,
\end{aligned}
\end{equation*}
\begin{equation*}
\begin{aligned}
\left|\mathcal{I}_{9,2,2}\right|
&\lesssim 
|\ba|\left(|\dot{\ba}|+|\dot{\bb}|\right)\sum_{n'\ne n}\int_{\RR^{4}}
\left|\partial_{x_{1}}\Psi_{n}\right|\left|\partial_{x_{1}}\Psi_{n'}\right|\d x\\
&\lesssim C_{0}^{2}t^{-4}\log ^{-\frac{1}{2}}t \left(C_{0}t^{-3}+C_{0}^{4}t^{-4}\log^{-1} t+C_{0}^{4}t^{-4}+t^{-4}\right)\lesssim C_{0}^{6}t^{-7}\log ^{-\frac{1}{2}}t,
\end{aligned}
\end{equation*}
\begin{equation*}
\begin{aligned}
\left|\mathcal{I}_{9,2,3}\right|
&\lesssim 
|\ba|\left(|\dot{\ba}|+|\dot{\bb}|\right)\sum_{k=1}^{K}\int_{\RR^{4}}
\left|\chi_{N}-\ell_{n}\right| \left|\partial_{x_{1}}\Psi_{n}\right|\left|\partial_{x_{1}}\Phi_{n,k}\right|\d x\\
&\lesssim C_{0}^{2}t^{-5}\log ^{-\frac{1}{2}}t \left(C_{0}t^{-3}+C_{0}^{4}t^{-4}\log^{-1} t+C_{0}^{4}t^{-4}+t^{-4}\right)\lesssim C_{0}^{6}t^{-8}\log ^{-\frac{1}{2}}t,
\end{aligned}
\end{equation*}
and
\begin{equation*}
\begin{aligned}
\left|\mathcal{I}_{9,2,4}\right|
&\lesssim 
|\ba|\left(|\dot{\ba}|+|\dot{\bb}|\right)\sum_{n'\ne n}\sum_{k=1}^{K}\int_{\RR^{4}}
 \left|\partial_{x_{1}}\Psi_{n}\right|\left|\partial_{x_{1}}\Phi_{n',k}\right|\d x\\
&\lesssim C_{0}^{2}t^{-5}\log ^{\frac{1}{2}}t \left(C_{0}t^{-3}+C_{0}^{4}t^{-4}\log^{-1} t+C_{0}^{4}t^{-4}+t^{-4}\right)\lesssim C_{0}^{6}t^{-8}\log ^{\frac{1}{2}}t.
\end{aligned}
\end{equation*}

Third, by integration by parts and~\eqref{derchi}, we decompose
\begin{equation*}
\mathcal{I}_{9,3}=\mathcal{I}_{9,3,1}+\mathcal{I}_{9,3,2},
\end{equation*}
where
\begin{equation*}
\begin{aligned}
\mathcal{I}_{9,3,1}=&-2a_{n}\int_{\RR^{4}}\left(\partial_{x_{1}} \vp\right) (\partial_{x_{1}}\chi_{N})\partial_{x_{1}}\Psi_{n}\d x,\\
\mathcal{I}_{9,3,2}=&-2a_{n}\sum_{i=1}^{4}\int_{\RR^{4}}\left(\partial_{x_{i}} \vp\right) (\chi_{N}-\ell_{n})\partial_{x_{i}}\partial_{x_{1}}\Psi_{n}\d x.
\end{aligned}
\end{equation*}
From~\eqref{est:boot},~\eqref{derchi},~\eqref{est:L1W2} and the Cauchy-Schwarz inequality,
\begin{equation*}
\begin{aligned}
\left|\mathcal{I}_{9,3,1}\right|&\lesssim t^{-1}|\ba|\|\nabla \vp\|_{L^{2}}\|\partial_{x_{1}}\Psi_{n}\|_{L^{2}(\Omega)}\lesssim C_{0}^{3}t^{-7}\log^{-\frac{1}{2}}t,\\
\left|\mathcal{I}_{9,3,2}\right|&\lesssim \sum_{i=1}^{4}|\ba|\|\nabla \vp\|_{L^{2}}\|(\chi_{N}-\ell_{n})\partial_{x_{i}}\partial_{x_{1}}\Psi_{n}\|_{L^{2}}\lesssim C_{0}^{3}t^{-7}\log^{-\frac{1}{2}}t.
\end{aligned}
\end{equation*}
Fourth, by an elementary computation, we decompose
\begin{equation*}
\mathcal{I}_{9,4}=\mathcal{I}_{9,4,1}+\mathcal{I}_{9,4,2},
\end{equation*}
where
\begin{equation*}
\begin{aligned}
\mathcal{I}_{9,4,1}&=6a_{n}\int_{\RR^{4}}\vp (\chi_{N}-\ell_{n})R^{2}\partial_{x_{1}}\Psi_{n}\d x,\\
\mathcal{I}_{9,4,2}&=6a_{n}\int_{\RR^{4}}\left(2R(U+V)+(U+V)^{2}\right)\vp (\chi_{N}-\ell_{n})\partial_{x_{1}}\Psi_{n}\d x.
\end{aligned}
\end{equation*}
Note that, from~\eqref{est:boot},~\eqref{est:L1W2}, the Lemma~\ref{le:int} (i) and the Sobolev inequality~\eqref{est:Sobo},
\begin{equation*}
\begin{aligned}
\left|\mathcal{I}_{9,4,1}\right|
&\lesssim |\ba|\|\vp\|_{L^{4}}\|Q\|_{L^{4}}\left(\|(\chi_{N}-\ell_{n})Q_{n}\partial_{x_{1}}\Psi_{n}\|_{L^{2}}
+\sum_{n'\ne n}\|Q_{n'}\partial_{x_{1}}\Psi_{n}\|_{L^{2}}\right)\\
&\lesssim C_{0}^{3}t^{-5}\log^{-\frac{1}{2}}t\left(t^{-4}+t^{-3}\right)\lesssim C_{0}^{3}t^{-8}\log^{-\frac{1}{2}}t.
\end{aligned}
\end{equation*}
Note also that, using~\eqref{est:boot} and~\eqref{est:L1W2} again, 
\begin{equation*}
\begin{aligned}
\left|\mathcal{I}_{9,4,2}\right|
&\lesssim |\ba|\|\vp\|_{L^{4}}\|U+V\|_{L^{8}}\|(\chi_{N}-\ell_{n})\partial_{x_{1}}\Psi_{n}\|_{L^{2}}\left(\|R\|_{L^{8}}+\|U+V\|_{L^{8}}\right)\\
&\lesssim |\ba|(|\ba|+|\bb|)\|\vec{\varphi}\|_{\E}\|(\chi_{N}-\ell_{n})\partial_{x_{1}}\Psi_{n}\|_{L^{2}}\lesssim
C_{0}^{5}t^{-8}\log t^{-\frac{1}{2}}t.
\end{aligned}
\end{equation*}
Last, using~\eqref{est:boot} and the Sobolev inequality~\eqref{est:Sobo} again,
\begin{equation*}
\left|\mathcal{I}_{9,5}\right|\lesssim |\ba|\|\vp\|^{2}_{L^{4}}\|\partial_{x_{1}}\Psi_{n}\|_{L^{4}}\left(\|\vp\|_{L^{4}}+\|R+U+V\|_{L^{4}}\right)
\lesssim C_{0}^{4}t^{-8}\log^{-\frac{1}{2}}t.
\end{equation*}
Combining above estimates and taking $T_{0}$ large enough, we obtain~\eqref{est:I9}.

\emph{Estimates on $\mathcal{I}_{10}$.} We claim
\begin{equation}\label{est:I10}
\mathcal{I}_{10}=O(t^{-7}).
\end{equation}
First, from~\eqref{equ:vp} and integration by parts, we decompose
\begin{equation*}
\mathcal{I}_{10}=\mathcal{I}_{10,1}+\mathcal{I}_{10,2}+\mathcal{I}_{10,3},
\end{equation*}
where
\begin{equation*}
\begin{aligned}
\mathcal{I}_{10,1}&=2a_{n}\ell_{n}\int_{\RR^{4}}\vpp (\partial_{x_{1}}\chi_{N})(\partial_{x_{1}}\Psi_{n})\d x,\qquad \qquad \qquad\qquad  \qquad \\
\mathcal{I}_{10,2}&=2a_{n}\ell_{n}\int_{\RR^{4}}\vpp(\chi_{N}-\ell_{n})\partial_{x_{1}}^{2}\Psi_{n}\d x,
\end{aligned}
\end{equation*}
\begin{equation*}
\begin{aligned}
\mathcal{I}_{10,3}
=&2a_{n}\ell_{n}\sum_{n'=1}^{N}\dot{a}_{n'}\int_{\RR^{4}}(\chi_{N}-\ell_{n})(\partial_{x_{1}}\Psi_{n})(\partial_{x_{1}}\Psi_{n'})\d x\\
&+2a_{n}\ell_{n}\sum_{(n',k)\in I^{0}}\dot{b}_{n',k}\int_{\RR^{4}}(\chi_{N}-\ell_{n})(\partial_{x_{1}}\Psi_{n})(\partial_{x_{1}}\Phi_{n',k})\d x.
\end{aligned}
\end{equation*}
From~\eqref{est:boot},~\eqref{derchi} and~\eqref{est:L1W2},
\begin{equation*}
\begin{aligned}
\left|\mathcal{I}_{10,1}\right|&\lesssim t^{-1}|\ba|\|\vpp\|_{L^{2}}\|\partial_{x_{1}}\Psi_{n}\|_{L^{2}(\Omega)}\lesssim C_{0}^{3}t^{-7}\log^{-\frac{1}{2}}t,\\
\left|\mathcal{I}_{10,2}\right|&\lesssim |\ba|\|\vpp\|_{L^{2}}
\|(\chi_{N}-\ell_{n})\partial^{2}_{x_{1}}\Psi_{n}\|_{L^{2}}\lesssim C_{0}^{3}t^{-7}\log^{-\frac{1}{2}}t.
\end{aligned}
\end{equation*}
Then, using the same arguments as the estimates of $\mathcal{I}_{9,2}$, we also have 
\begin{equation*}
\left|\mathcal{I}_{10,3}\right|\lesssim C_{0}^{6}t^{-7}\log^{-\frac{1}{2}}t+C_{0}^{6}t^{-8}\log^{\frac{1}{2}}t.
\end{equation*}
Combining above estimates and taking $T_{0}$ large enough, we obtain~\eqref{est:I10}.

\emph{Estimate on $\mathcal{I}_{11}$.} Recall that 
\begin{equation*}
\partial_{t} \partial_{x_{1}}\Psi_{n}=-\ell_{n}\partial_{x_{1}}^{2}\Psi_{n}\quad \mbox{for all}\ n=1,\dots,N.
\end{equation*}
Thus, using~\eqref{est:boot} and~\eqref{est:L1W2} again,
\begin{equation}\label{est:I11}
\left|\mathcal{I}_{11}\right| \lesssim |\ba|\left(\|\nabla \vp\|_{L^{2}}+\|\vpp\|_{L^{2}}\right)
\|(\chi_{N}-\ell_{n})\partial_{x_{1}}^{2}\Psi_{n}\|_{L^{2}}\lesssim C_{0}^{3}t^{-7}\log^{-\frac{1}{2}}t.
\end{equation}

\emph{Estimate on $\mathcal{I}_{12}$.} From~\eqref{est:boot},~\eqref{derchi} and~\eqref{est:L1W2}, we have 
\begin{equation}\label{est:I12}
\left|\mathcal{I}_{12}\right|\lesssim t^{-1}|\ba|\left(\|\nabla \vp\|_{L^{2}}+\|\vpp\|_{L^{2}}\right)\|\partial_{x_{1}}\Psi_{n}\|_{L^{2}(\Omega)}\lesssim C_{0}^{3}t^{-7}\log^{-\frac{1}{2}}t.
\end{equation}
We see that~\eqref{est:dtJ} follows from~\eqref{est:I9},~\eqref{est:I10},~\eqref{est:I11} and~\eqref{est:I12}.
	\end{proof}

Last, we conclude the time variation of $\mathcal{K}$ from Lemma~\ref{le:timeEPGJ}.

\begin{lemma}[Time variation of $\mathcal{K}$]\label{le:timeK}
	 There exists $T_{0}$ large enough such that
	\begin{equation}\label{est:dK}
	-\frac{\d}{\d t}(t^{2}\mathcal{K})\lesssim C_{0} t^{-5}.
	\end{equation}
\end{lemma}
\begin{proof}
	Combining the estimates~\eqref{est:dE},~\eqref{est:dP},~\eqref{est:dG} and~\eqref{est:dtJ}, we decompose
	\begin{equation*}
	\frac{\d}{\d t}\mathcal{K}=\frac{\d}{\d t}\mathcal{E}+\frac{\d}{\d t}\mathcal{P}+\frac{\d}{\d t}\mathcal{G}+\sum_{n=1}^{N}\frac{\d}{\d t}\mathcal{J}_{n}=\mathcal{I}_{13}+\mathcal{I}_{14}+\mathcal{I}_{15}
	+O(C_{0}t^{-7}).
	\end{equation*}
	where
	\begin{equation*}
	\begin{aligned}
	\mathcal{I}_{13}&=-2\int_{\RR^{4}}\vpp\left({\rm{Mod}}_{2,2}+\chi_{N}\partial_{x_{1}}{\rm{Mod}}_{1,2}\right)\d x,\\
	\mathcal{I}_{14}&=2\int_{\RR^{4}}\vp\left(\Delta {\rm{Mod}_{1,2}}+3R^{2}{\rm{Mod}_{1,2}}\right)\d x-2\int_{\RR^{4}}\chi_{N}{\rm{Mod}_{2,2}}\partial_{x_{1}}\vp \d x,\\
	\mathcal{I}_{15}&=-\frac{1}{(1-2\delta)t}\int_{\Omega}
	\left((\partial_{x_{1}}\vp)^{2}+\vpp^{2}
	+2\frac{x_{1}}{t}(\partial_{x_{1}}\vp)\vpp-|\overline{\nabla}\vp|^{2}\right)\d x.
	\end{aligned}
	\end{equation*}
	Note that, from the definition of ${\rm{Mod}}_{1,2}$, ${\rm{Mod}}_{2,2}$ and $\chi_{N}$,
	\begin{equation*}
	\mathcal{I}_{13}=2\sum_{(n,k)\in I^{0}}\dot{b}_{n,k}\int_{\RR^{4}}\vpp\left(\ell_{n}-\chi_{N}\right) \partial_{x_{1}}\Phi_{n,k}\d x.
	\end{equation*}
	Therefore, from~\eqref{est:boot},~\eqref{est:L1W2} and the decay property of $\phi_{k}$ in~\eqref{est:phipsi},
	\begin{equation*}
	\begin{aligned}
	\left|\mathcal{I}_{13}\right|
	&\lesssim \sum_{(n,k)\in I^{0}}|\dot{\bb}|\|\vpp\|_{L^{2}}\|\left(\ell_{n}-\chi_{N}\right) \partial_{x_{1}}\Phi_{n,k}\|_{L^{2}}\\
	&\lesssim t^{-2}\|\vec{\varphi}\|_{\E}\left(\|\vec{\varphi}\|_{\E}+|\ba|^{2}+|\bb|^{2}+t^{-4}\right)\lesssim C_{0}^{5}t^{-8}.
	\end{aligned}
	\end{equation*}
	Note also that, by $-(1-\ell_{n}^{2})\partial^{2}_{x_{1}}\Phi_{n,k}-\bar{\Delta}\Phi_{n,k}-3Q_{n}^{2}\Phi_{n,k}=0$ and integration by parts,
	\begin{equation*}
	\begin{aligned}
	\mathcal{I}_{14}=&2\sum_{(n,k)\in I^{0}}\dot{b}_{n,k}\int_{\RR^{4}}\ell_{n}\partial_{x_{1}}\vp\left(\chi_{N}-\ell_{n}\right)\partial_{x_{1}}\Phi_{n,k}\d x\\
	&+6\sum_{(n,k)\in I^{0}}\sum_{n'\ne n}\dot{b}_{n,k}\int_{\RR^{4}}\vp\Phi_{n,k}\left(2Q_{n}Q_{n'}+Q_{n'}^{2}\right)\d x
	\end{aligned}
	\end{equation*}
	Therefore, using~\eqref{est:ab},~\eqref{est:boot},~\eqref{est:L1W2}, the Lemma~\ref{le:int} and the decay property of $\phi_{k}$ in~\eqref{est:phipsi} again, we have 
	\begin{equation*}
	\begin{aligned}
	\left|\mathcal{I}_{14}\right|&\lesssim  \sum_{(n,k)\in I^{0}}|\dot{\bb}|\|\nabla \vp\|_{L^{2}}\|\left(\ell_{n}-\chi_{N}\right) \partial_{x_{1}}\Phi_{n,k}\|_{L^{2}}\\
	&\quad +\sum_{(n,k)\in I^{0}}\sum_{n'\ne n}|\dot{\bb}|\|\vp\|_{L^{4}}\left(\|Q_{n}Q_{n'}\Phi_{n,k}\|_{L^{\frac{4}{3}}}+\|Q_{n'}^{2}\Phi_{n,k}\|_{L^{\frac{4}{3}}}\right)\\
	&\lesssim t^{-2}\|\vec{\varphi}\|_{\E}\left(\|\vec{\varphi}\|_{\E}+|\ba|^{2}+|\bb|^{2}+t^{-4}\right)\lesssim C_{0}^{5}t^{-8}.
	\end{aligned}
	\end{equation*}
	From $0<\delta<\frac{1}{20}$ and the definition of $\chi_{N}$, we have, for all $x=(x_{1},\bar{x})\in \Omega$,
	\begin{equation*}
	\left|\frac{x_{1}}{t}-\chi_{N}\right|\le \frac{2\delta}{1-2\delta}\left|\frac{x_{1}}{t}\right|+\frac{2\bar{\ell}\delta}{1-2\delta}\le 8\delta.
	\end{equation*}
	Thus, from~\eqref{est:Nomega},~\eqref{est:coerK} and the definition of $\delta$ and $\mathcal{N}_{\Omega}$,
	\begin{equation*}
	\begin{aligned}
	-t\mathcal{I}_{15}&\le \frac{1}{1-2\delta}\left(\mathcal{N}_{\Omega}+2\int_{\Omega}\left(\frac{x_{1}}{t}-\chi_{N}\right)(\partial_{x_{1}}\vp) \vpp\right)\d x\\
	&\le \frac{10}{9}\left(1+16\delta(1-\bar{\ell})^{-1}\right)\mathcal{N}_{\Omega}\le 2\mathcal{K}+2\nu^{-1}\left(t^{-6}+C_{0}^{5}t^{-6}\log^{-\frac{1}{2}}t\right).
	\end{aligned}
	\end{equation*}
	Gathering above estimates we have,
	\begin{equation*}
	-\frac{\d}{\d t}(t^{2}\mathcal{K})\lesssim C_{0}t^{-5}+C_{0}^{5}t^{-6}+C_{0}^{5}t^{-5}\log^{-\frac{1}{2}}t,
	\end{equation*} 
	which implies~\eqref{est:dK} for $T_{0}$ large enough.
	\end{proof}

\subsection{End of the proof of Propostion~\ref{pro:uni}}\label{SS:Pro}
Recall that, for the proof of Proposition~\ref{pro:uni}, we just need to prove the existence of $\boldsymbol{z}_{m}\in \mathcal{B}_{\RR^{|I|}}(S_{m}^{-\frac{7}{2}})$ such that $T^{*}(\boldsymbol{z}_{m})=T_{0}$. We start by closing all bootstrap estimates except the one for the unstable modes. Last, we prove the existence of suitable parameters $\boldsymbol{z}_{m}=(z_{n,j}^{m})_{(n,j)\in I}\in \mathcal{B}_{\RR^{|I|}}(S_{m}^{-\frac{7}{2}})$ using a topological argument.

\begin{proof}[Proof of Proposition~\ref{pro:uni}]
	For all $\boldsymbol{z}_{m}=(z_{n,j}^{m})_{(n,j)\in I}\in \mathcal{B}_{\RR^{|I|}}(S_{m}^{-\frac{7}{2}})$, we consider the solution $\vec{u}$ with the initial data at $T=S_{m}$ as defined in Lemma~\ref{le:ini}. From the definition of the initial data, we have 
	\begin{equation}\label{est:iniSm}
	\|\vec{\varphi}(S_{m})\|_{\E}\lesssim S_{m}^{-\frac{7}{2}}\quad \mbox{and}\quad a_{n}(S_{m})=b_{n,j}(S_{m})=0,
	\end{equation}
	for all $n=1,\dots, N$ and $(n,k)\in I^{0}$.
	
	\textbf{Step 1.} Closing the estimate in $\vec{\varphi}$. Integrating~\eqref{est:dK} on $[t,S_{m}]$ for any $t\in [T^{*},S_{m}]$ and using~\eqref{est:iniSm}, we have 
	\begin{equation*}
	\mathcal{K}(t)\lesssim C_{0}t^{-6}+\left(\frac{S_{m}}{t}\right)^{2}\left|\mathcal{K}(S_{m})\right|\lesssim C_{0}t^{-6}.
	\end{equation*}
	Thus, from~\eqref{est:Nomega} and~\eqref{est:coerK}, for $T_{0}$ large enough (depending on $C_{0}$),
	\begin{equation*}
	\|\vec{\varphi}(t)\|_{\E}^{2}\lesssim \mathcal{K}(t)+t^{-6}+C_{0}^{5}t^{-6}\log^{-\frac{1}{2}}t
	\lesssim C_{0}t^{-6}+C_{0}^{5}t^{-6}\log^{-\frac{1}{2}}t\lesssim C_{0}t^{-6}.
	\end{equation*}
	This strictly improves the estimate on $\vec{\varphi}$ in~\eqref{est:boot} for $C_{0}$ large enough.
	
	\textbf{Step 2.} Closing the estimate in $\ba$. First, from~\eqref{est:Psi34} and~\eqref{est:boot}, for all $t\in[T^{*},S_{m}]$,
	\begin{equation}\label{est:cn}
	\sum_{n=1}^{N}\left|c_{n}(t)\right|\lesssim \log^{-1}t \left(\|\vp\|_{L^{4}}\|\Psi_{n}\xi_{n}\|_{L^{\frac{4}{3}}}\right)\lesssim C_{0}t^{-2}\log^{-1}t.
	\end{equation}
	Then, from~\eqref{est:Rdota} and~\eqref{est:boot}, we have 
	\begin{equation*}
	\sum_{n=1}^{N}|\dot{a}_{n}+\dot{c}_{n}|\lesssim C_{0}t^{-3}\log^{-\frac{1}{2}}t+C_{0}^{2}t^{-4}.
	\end{equation*}
	Integrating above estimate on $[t,S_{m}]$ for any $t\in[T^{*},S_{m}]$ and then using~\eqref{est:cn}, 
	\begin{equation*}
	\begin{aligned}
	\sum_{n=1}^{N}|a_{n}(t)|&\lesssim \sum_{n=1}^{N}\left(|c_{n}(t)|+|c_{n}(S_{m})|\right)+C_{0}t^{-2}\log^{-\frac{1}{2}}t+C_{0}^{2}t^{-3}\\
	&\lesssim C_{0}t^{-2}\log^{-1}t+C_{0}t^{-2}\log^{-\frac{1}{2}}t+C_{0}^{2}t^{-3}\lesssim C_{0}t^{-2}\log^{-\frac{1}{2}}t,
	\end{aligned}
	\end{equation*}
	for $T_{0}$ large enough. This strictly improves the estimate on $\ba$ in~\eqref{est:boot} for taking $C_{0}$ large enough.
	
	\textbf{Step 3.} Closing the estimate in $\bb$. From~\eqref{est:ab} and~\eqref{est:boot}, we have 
	\begin{equation*}
	\sum_{(n,k)\in I^{0}}|\dot{b}_{n,k}|\lesssim C_{0}t^{-3}+C_{0}^{4}t^{-4}.
	\end{equation*}
		Integrating above estimate on $[t,S_{m}]$ for any $t\in[T^{*},S_{m}]$ and then using~\eqref{est:iniSm}, 
	\begin{equation*}
	\sum_{(n,k)\in I^{0}}|b_{n,k}(t)|\lesssim C_{0}t^{-2}+C_{0}^{4}t^{-3}\lesssim C_{0}t^{-2},
	\end{equation*}
	for $T_{0}$ large enough. This strictly improves the estimate on $\bb$ in~\eqref{est:boot} for taking $C_{0}$ large enough.
	
	\textbf{Step 4.} Closing the estimate in $(z_{n,j}^{-})_{(n,j)\in I}$. By~\eqref{est:dz} and~\eqref{est:boot}, for all $(n,j)\in I$, we have 
	\begin{equation*}
	\begin{aligned}
	\frac{\d}{\d t}\left(z_{n,j}^{-}\right)^{2}&=2\alpha_{n,j}\left(z_{n,j}^{-}\right)^{2}+O\left(\|\vec{\varphi}\|_{\E}\left(\|\vec{\varphi}\|_{\E}^{2}+|\ba|^{2}+|\bb|^{2}+t^{-4}\right)\right)\\
	&=2\alpha_{n,j}\left(z_{n,j}^{-}\right)^{2}+O\left(C_{0}^{3}t^{-9}+C_{0}^{5}t^{-7}\right)
	=2\alpha_{n,j}\left(z_{n,j}^{-}\right)^{2}+O\left(C_{0}^{5}t^{-7}\right).
	\end{aligned}
	\end{equation*}
	Integrating above estimate on $[t,S_{m}]$ for any $t\in[T^{*},S_{m}]$ and then using~\eqref{est:iniSm}
	\begin{equation*}
	\begin{aligned}
	\left(z_{n,j}^{-}\right)^{2}(t)
	&\lesssim e^{-2\alpha_{n,j}(S_{m}-t)}\left(z_{n,j}^{-}\right)^{2}(S_{m})+C_{0}^{5}\int_{t}^{S_{m}}e^{2\alpha_{n,j}(t-s)}s^{-7}\d s\lesssim C_{0}^{5}t^{-7}.
	\end{aligned}
	\end{equation*}
	This strictly improves the estimate on $(z_{n,j}^{-})_{(n,j)\in I}$ in~\eqref{est:boot} for taking $T_{0}$ and $C_{0}$ large enough.
	
	\textbf{Step 5.} Final argument on the unstable parameters. Let
	\begin{equation*}
	\mathcal{A}(t)=\sum_{(n,j)\in I}\left(z_{n,j}^{+}\right)^{2}\quad \mbox{and}\quad 
	\bar{\alpha}=\min_{(n,j)\in I}\alpha_{n,j}>0.
	\end{equation*}
	We claim, for any time $t\in [T^{*},S_{m}]$ where it holds $\mathcal{A}(t)=t^{-7}$, the following  transversality property holds,
	\begin{equation}\label{est:trans}
	\frac{\d}{\d t}\left(t^{7}\mathcal{A}(t)\right)\le -\bar{\alpha}.
	\end{equation}
	Indeed, from~\eqref{est:dz} and~\eqref{est:boot}, for any $t\in[T^{*},S_{m}]$ where it holds $\mathcal{A}(t)=t^{-7}$,
	\begin{equation*}
	\begin{aligned}
	\frac{\d}{\d t}\mathcal{A}(t)
	&=-2\sum_{(n,j)\in I}\alpha_{n,j}\left(z^{+}_{n,j}\right)^{2}+O\bigg(\sum_{(n,j)\in I}|z^{+}_{n,j}|\left(\|\vec{\varphi}\|_{\E}^{2}+|\ba|^{2}+|\bb|^{2}+t^{-4}\right)\bigg)\\
	&\le -2\bar{\alpha}t^{-7}+O\left(C_{0}^{2}t^{-\frac{19}{2}}+C_{0}^{4}t^{-\frac{15}{2}}\right)\le 
	-2\bar{\alpha}t^{-7}+C_{0}^{5}t^{-\frac{15}{2}},
	\end{aligned}
	\end{equation*}
	which implies~\eqref{est:trans} for taking $T_{0}$ large enough. The transversality relation~\eqref{est:trans} is enough to justify the existence of at least a couple $\boldsymbol{z}_{m}=(z_{n,j}^{m})_{(n,j)\in I}\in \mathcal{B}_{\RR^{|I|}}\left(S_{m}^{-\frac{7}{2}}\right)$ such that $T^{*}(\boldsymbol{z}_{m})=T_{0}$.
	
	The proof is by contradiction, we assume that for all $\boldsymbol{z}_{m}=(z_{n,j}^{m})_{(n,j)\in I}\in \mathcal{B}_{\RR^{|I|}}\left(S_{m}^{-\frac{7}{2}}\right)$, it holds $T_{0}<T_{*}(\boldsymbol{z}_{m})$. Then, a contradiction follows from the following discussion (see for instance more details in~\cite[Section 3.1]{CMkg} and~\cite[Lemma 4.2]{MM}).
	
	\emph{Continuity of $T^{*}(\boldsymbol{z}_{m})$.} The above transversality condition~\eqref{est:trans} implies that the map
	\begin{equation*}
	\boldsymbol{z}_{m}\in \mathcal{B}_{\RR^{|I|}}\left(S_{m}^{-\frac{7}{2}}\right)\mapsto T^{*}(\boldsymbol{z}_{m})\in [T_{0},S_{m}]
	\end{equation*}
	is continuous and 
	\begin{equation*}
	T^{*}({\boldsymbol{z}}_{m})=S_{m}\quad \mbox{for}\quad \boldsymbol{z}_{m}\in \mathcal{S}_{\RR^{|I|}}\left(S_{m}^{-\frac{7}{2}}\right).
	\end{equation*}
	
	\emph{Construction of a retraction}. We define
	\begin{align*}
	\mathcal{M}:\ \bar{\mathcal{B}}_{\RR^{|I|}}\left(S_{m}^{-\frac{7}{2}}\right)&\mapsto \mathcal{S}_{\RR^{|I|}}\left(S_{m}^{-\frac{7}{2}}\right)\\
	\boldsymbol{z}_{m}=(z_{n,j}^{m})_{(n,j)\in I}&\mapsto \left(\frac{T^{*}(\boldsymbol{z}_{m})}{S_{m}}\right)^{\frac{7}{2}}\left(z_{n,j}^{+}\left(T^{*}(\boldsymbol{z}_{m})\right)\right)_{(n,j)\in I}.
	\end{align*}
	From what precedes, $\mathcal{M}$ is continuous. Moreover, $\mathcal{M}$ restricted to
	$\mathcal{S}_{\RR^{|I|}}\left(S_{m}^{-\frac{7}{2}}\right)$ is the identity. The
	existence of such a map is contradictory with the no retraction theorem for continuous maps from the ball to
	the sphere. Therefore, the existence of $\boldsymbol{z}_{m}\in {\mathcal{B}}_{\RR^{|I|}}\left(S_{m}^{-\frac{7}{2}}\right)$ have proved. Then, the uniform estimate~\eqref{est:uni} is a consequence of bootstrap estimates~\eqref{est:boot}. The proof of Proposition~\ref{pro:uni} is complete.
	
	\end{proof}

\subsection{Proof of Theorem~\ref{the:main} from Proposition~\ref{pro:uni}}\label{SS:thm}
The proof is based on a standard compactness argument (see for example~\cite{CMTAMS,JM,Y5De}).
The main point is the following proposition of weak continuity of the flow near a compact set.

\begin{proposition}\label{pro:weak}
	There exists a constant $\varepsilon>0$ such that the following holds. Let $\mathcal{S}\subset \dot{H}^{1}\times L^{2}$ be a compact set and let $\vec{u}_{m}:[T_{1},T_{2}]\to \dot{H}^{1}\times L^{2}$ be a sequence of solutions of~\eqref{equ:wave} such that 
	\begin{equation*}
	\sup_{m\in \mathbb{N}}\left(\max_{t\in \left[T_{1},T_{2}\right]}{\rm{dist}}\left(\vec{u}_{m}(t),\mathcal{S}\right)\right)\le \varepsilon.
	\end{equation*}
	Suppose that $\vec{u}_{m}(T_{0})\rightharpoonup \vec{u}_{0}$ weakly in $\dot{H}^{1}\times L^{2}$. Then the solution $\vec{u}(t)$ of~\eqref{equ:wave} with the initial data $\vec{u}(T_{1})=\vec{u}_{0}$ is well-defined on $[T_{1},T_{2}]$ and 
	\begin{equation*}
	\vec{u}_{m}(t)\rightharpoonup \vec{u}(t),\quad \mbox{weakly in}\ \dot{H}^{1}\times L^{2}\quad \mbox{for all}\ t\in [T_{1},T_{2}].
	\end{equation*}
\end{proposition}

\begin{proof}
The proof relies on an argument based on the result of profile decomposition stated in~\cite[Proposition 2.8]{DKMJEMS}. See more details in~\cite[Appendix A.2]{JJ} and~\cite[Appendix]{JM}.
\end{proof}

We are in a position to complete the proof of Theorem~\ref{the:main}.

\begin{proof}[Proof of Theorem~\ref{the:main}]
	We consider the sequence of solutions $(\vec{u}_{m})_{m\in \mathbb{N}}$ given by Proposition~\ref{pro:uni}.
	On the time interval $[T_{0},S_{m}]$, the solution $\vec{u}_{m}$ is well-defined. Moreover, for $T_{0}$ large enough, we have 
	\begin{equation*}
	\max_{t\in \left[T_{0},S_{m}\right]}\|\vec{u}_{m}(t)-\sum_{n=1}^{N}\vec{Q}_{n}(t)\|_{\E}\le C t^{-2}\le \varepsilon,
    \end{equation*}
    where $\varepsilon>0$ is the constant of Proposition~\ref{pro:weak}.
    
    Since the sequence $(\vec{u}_{m}(T_{0}))_{m\in \mathbb{N}}$ is bounded in $\dot{H}^{1}\times L^{2}$, up to the extraction of a subsequence, there exists $\vec{u}_{0}\in \dot{H}^{1}\times L^{2}$ such that $\vec{u}_{m}(T_{0})\rightharpoonup \vec{u}_{0}$ weakly in $\dot{H}^{1}\times L^{2}$. Fix $S_{M}=M^{2}>T_{0}$ for $M\in \mathbb{N}$. Applied the Proposition~\ref{pro:weak} to the following compact set 
    \begin{equation*}
    \mathcal{S}_{M}=\left\{\sum_{n=1}^{N}\vec{Q}_{n}(t)\in \dot{H}^{1}\times L^{2}, t\in[T_{0},S_{M}]\right\},
    \end{equation*}
    we have, the solution $\vec{u}(t)$ of~\eqref{equ:wave} with the initial data $\vec{u}(T_{0})=\vec{u}_{0}$ is well-defined on $[T_{0},S_{M}]$ and $\vec{u}_{m}(t)\rightharpoonup \vec{u}(t)$ weakly in $\dot{H}^{1}\times L^{2}$ for all $t\in [T_{0},S_{M}]$. From~\eqref{est:uni} and the properties of weak convergence, we have 
    \begin{equation*}
    \max_{t\in \left[T_{0},S_{M}\right]}\left(\|\vec{u}(t)-\sum_{n=1}^{N}\vec{Q}_{n}(t)\|_{\E}\right)\le Ct^{-2}.
    \end{equation*}
    Since $S_{M}=M^{2}>T_{0}$ is arbitrary, the solution $\vec{u}(t)$ is well-defined and satisfies the conclusion of Theorem~\ref{the:main} on $[T_{0},\infty)$. The proof of Theorem~\ref{the:main} is complete.
    \end{proof}

\end{document}